%% file: main_rshe_ergodic.tex
\documentclass[11pt,a4paper]{article}
\usepackage[T1]{fontenc}
\usepackage{graphicx}
\graphicspath{ {images/} }
\usepackage{inputenc}
\usepackage{appendix}
\usepackage{nicefrac}
\usepackage{geometry}

\usepackage{caption}
\usepackage{subcaption}
\usepackage{lipsum}
\RequirePackage{amsthm,amsmath}
\usepackage[doi=false,isbn=false,url=false,eprint=false, style=numeric,maxbibnames=99,giveninits=true,maxcitenames=99,backend=bibtex]{biblatex}
\AtBeginBibliography{\footnotesize}

\renewrobustcmd*{\bibinitdelim}{\,}
\RequirePackage[colorlinks,citecolor=blue,urlcolor=blue]{hyperref}

\usepackage{enumerate}
\usepackage{xcolor}

\usepackage{amsfonts}
\usepackage{mathrsfs}	
\usepackage{amssymb}
\usepackage{dsfont}
\usepackage{bbm}
\usepackage{MnSymbol}

\theoremstyle{plain}
\newtheorem{theorem}{Theorem}[section]
\newtheorem{lemma}[theorem]{Lemma}
\newtheorem{proposition}[theorem]{Proposition}

\theoremstyle{definition}
\newtheorem{definition}[theorem]{Definition}

\newtheorem{assumption}[theorem]{Assumption}

\newtheorem{remark}[theorem]{Remark}

\numberwithin{equation}{section}
\numberwithin{figure}{section}
\usepackage[nodayofweek]{datetime}

\usepackage[nodayofweek]{datetime}

\newcommand{\bF}{\mathbb{F}}

\newcommand{\bN}{\mathbb{N}}
\newcommand{\bP}{\mathbb{P}}

\newcommand{\bR}{\mathbb{R}}
\newcommand{\bS}{\mathbb{S}}

\newcommand{\cC}{\mathcal{C}}

\newcommand{\cE}{\mathcal{E}}
\newcommand{\cF}{\mathcal{F}}

\newcommand{\cP}{\mathcal{P}}

\newcommand{\cX}{\mathcal{X}}
\newcommand{\cY}{\mathcal{Y}}

\newcommand{\E}{\mathbb{E}}

\renewcommand{\d}{\textnormal{d}}

\usepackage[colorinlistoftodos,bordercolor=orange,backgroundcolor=orange!20,linecolor=orange,textsize=tiny]{todonotes}

\title{
	{\textbf{Rearranged Stochastic Heat Equation: \\
	Ergodicity and Related Gradient Descent on ${\mathcal P}({\mathbb R})$}}\\
}
\author{François Delarue and William R.P. Hammersley \thanks{F. Delarue and W. Hammersley are supported by the European Research Council (ERC) under the European Union’s Horizon 2020 research and innovation programme (ELISA project, Grant agreement No. 101054746). Email: francois.delarue@univ-cotedazur.fr, william.hammersley@univ-cotedazur.fr} \\ Université Côte d'Azur, CNRS, Laboratoire J.A. Dieudonné}
 
\date{\today}
\begin{document}
	\maketitle
	
	\begin{abstract}
		This article provides a case study for a recently introduced diffusion in the space of probability measures over the reals, namely rearranged stochastic heat, which solves a stochastic partial differential equation valued in the set of symmetrised quantile functions over the unit circle. This contribution studies probability measure-valued flows perturbed by this noise with a special focus on gradient flows. This is done by introducing a drift to the rearranged stochastic heat equation by means of a vector field from the set of random variables over the unit circle into itself. When the flow is a gradient flow, the vector field may coincide with the Wasserstein derivative of a mean-field potential function. The resulting equation reads as a sort of McKean-Vlasov stochastic differential equation with an infinite dimensional common noise. This article provides conditions on the drift under which solutions exist uniquely for any time horizon and converge exponentially fast towards a unique equilibrium. When the drift derives from a potential, some metastability properties are obtained as the intensity of the noise is tuned to zero: it is shown that under a particular scaling regime, the gradient descent lingers near local minimizers for expected times of the same order as in the finite dimensional setting.
		\medskip
		
\vspace{2pt} 

	\end{abstract}
     
\section*{Introduction}

\input{sections/intro} 

\section{Weak Existence Uniqueness for the Drifted Equation}
\label{sec exist}
\input{sections/EU4Drifted}

\section{Exponential Convergence to Equilibrium} 

\label{sec conv}	 
\input{sections/expConv}

\section{Stochastic Gradient Descent}
\label{sec SGD}

\input{sections/SGD}

\subsection{Vanishing Viscosity}
\label{sec ito}	 
\input{sections/vanishVisc}


\appendix

\printbibliography 
\end{document}

%% file: sections/intro.tex

This article provides a case study for a nascent approach to regularisation by noise for mean field models as introduced by the authors in \cite{delarueHammersley2022rshe}, where a diffusion in the space ${\mathcal P}_2({\mathbb R})$ of probability measures on the real line with finite second moment is found via the solution to what we call the rearranged stochatic heat equation. This is the one dimensional stochastic heat equation with coloured noise constrained within the set of symmetric non-increasing functions over the unit circle $\bS$. 
When equipped with the $L^2$-norm, this space of symmetric non-increasing functions has the key feature to be isometric to ${\mathcal P}_2({\mathbb R})$ equipped with the $2^{nd}$-Wasserstein distance ${\mathcal W}_2$. In light of this correspondence, our approach can be viewed as randomising over the space of quantile functions. For more details see Section \ref{sec:rearrangement} and \cite{delarueHammersley2022rshe}. 
\vskip 4pt

\noindent \textit{Gradient descents.}
As an important step in the implementation of this heuristic, gradient descents in ${\mathcal P}_2({\mathbb R})$ are considered in this text. The analysis is effectuated at the level of a formal infinite particle (mean field) limit. Following the technique of \emph{lifting} from Lions \cite{lions2007-8lectures}, one rewrites the gradient dynamics over $L^2(\bS,\text{Leb}_{\bS})$ and considers the rearranged stochastic heat equation with drift given by the lifted gradient flow ($\text{Leb}_{\bS}$ denotes the Lebesgue measure on $\mathbb S$). 
To illustrate this approach, consider a minimisation problem for a potential $V$ written over the space of square integrable probability measures on the real line, 
\begin{equation}
	\notag
	\min_{\mu\in\cP_2(\mathbb{R})}\{V(\mu)\}.
\end{equation} 
Cognisant of the mean field perspective, one could hope to achieve minimisation of the above potential by considering the law of a McKean-Vlasov type gradient descent on a probability space $(\Omega,\cF,\bP)$,
\begin{equation}
	\label{eq:MKV:1}
	\tfrac{\d}{\d t} X_t(\omega)= -\partial_\mu V(\mu_t,X_t(\omega)),\quad \mu_t:=\bP\circ X_t^{-1},
\end{equation}
where $\partial_\mu V$ is here the Wasserstein derivative of $V$ (see for instance 
\cite[Chapter 5]{CarmonaDelarueI}). 
Given the issues of well-posedness of the above dynamics in general (amongst other motivations), two popular stochastic versions of the above have been well-studied in recent years, namely the idiosyncratic (private) and systemic (common) noise models arising from corresponding finite size mean field interacting particle systems. 
\vskip 4pt

\noindent \textit{Private vs. common noise.}
In the private noise model, 
\eqref{eq:MKV:1}
is subjected to additional noise, usually Brownian motion, defined on the same space 
$\Omega$ as the one upon which the law of $X_t$ is taken. In short, this corresponds to the asymptotic regime of mean field interacting particle systems where each particle is driven by its own independent (hence private) noise. We refer to Sznitman \cite{sznitman1991topics} for an introduction to this topic. Via a mechanism known as the propagation of chaos, the instantaneous dynamics of each particle in the asymptotic infinite particle system are stochastically independent yet statistically identical; they depend on the selfsame flow of induced marginal laws $(\mu_t)_{t \geq 0}$, which itself solves a deterministic non-linear Fokker-Planck-Kolomogorov equation. 

In the common noise setting, the common noise may be defined on a larger probability space
 and $\mu_t$ defined simply as the law of $X_t$ given the inputs of the noise that are defined outside the original probability space. Intuitively, this means that in the corresponding mean field interacting particle systems, independence of the driving noises is relaxed to exchangeability. Evidently, there is a wealth of options for correlating the noises in an exchangeable manner. {We quote just a few of the examples that have been amenable to study: see the  works by Dawson and Vaillancourt \cite{dawsonVaillancourt1995,vaillancourt1988} for the foundations, 
the recent papers \cite{ANGELI2023127301,hammersleySiskaSzpruch2021cnMKVwk}
for more specific properties of McKean-Vlasov equations with common noise, 
the contributions 
 \cite{CDLL,CarmonaDelarueII,carmona2015mfgcn,kolokoltsovTroeva2019mfgcn}
 for concrete illustrations in  
mean field game theory, 
\cite{ledgerSojmark2021}
for an excitatory model with common noise inspired from neurosciences or credit-risk management, 
and also
the articles
\cite{coghiDeuschelFrizMaurelli2020pathwiseMKVwithAddNoise,COGHI20211}
for relations with 
rough paths theory.}   

As mentioned earlier, the idiosyncratic noise model introduces a Brownian motion to the above McKean-Vlasov type dynamics, and consequently a Laplacian to the corresponding (hence second-order) Fokker-Planck-Kolmogorov equation. The resulting second-order 
Fokker-Planck-Kolmogorov equation
is usually regarded as a gradient descent (on the space of probability measures) deriving 
from a new potential, equal to the sum of the potential $V$ plus an entropy penalty. 
{See the seminal works 
\cite{Ambrosio,mccann,JKO,Otto1,,Otto2} on gradient flows; 
see also 
\cite{chizat2022mfLD,huRenSiskaSzpruch2021gradFlowMfLangevinNN,meiMontanariNguyen2018MFNN,nitandaWu2022convAnalysisMFLD}
for the connection 
with statistical learning 
and 
\cite{claisseConfortiRenWang2023mfOptRegByFisherInfo} for a variant based on the Fisher information instead of the entropy.} Whilst this new gradient descent is regularised by the presence of the Laplacian, it remains deterministic. Consequently, long time convergence to a global minimiser of the new entropy-regularised potential may fail outside of the convex setting: {see for instance 
\cite{dawson1983CriticalDA}}.

As for the case with common noise, many of the existing examples that have been addressed so far are limited to finite dimensional common noises whose structure has been demonstrated to be insufficient for smoothing out the dynamics in a relevant manner {and for ensuring a suitable form of ergodicity.
See however the recent contributions 
\cite{ANGELI2023127301,maillet2023note} for long time convergence results in presence of a finite dimensional common noise under a specific mean field interaction.} 
\vskip 4pt 
\noindent \textit{Infinite dimensional common noise.} 
Instead, it is natural to suppose that a regularising and explorative noise with satisfactory generalisability should randomise the dynamics over the ambient space on which the minimisation problem is posed. However, how to replace a finite dimensional Brownian motion, with a counterpart on the space of probability  measures? Candidates for this rôle are often referred to as Wasserstein diffusions. {Typically,}
these are processes that are coherent with the Riemannian-type structure on the space $\cP_2({\mathbb R})$ endowed with the $2^{nd}$ Wasserstein distance $W_2$. {Key examples are those of von Renesse and Sturm \cite{vRenesseSturm}, Konarovskyi and von Renesse \cite{KonarovskyivRenesse} and Dello Schiavo \cite{DelloSchiavo0}. Although associated with a different geometry, 
Fleming-Viot processes also exhibit interesting properties}, 
 see Stannat \cite{Stannat}.

Following the approach developed in 
\cite{delarueHammersley2022rshe}, {which differs from the constructions in the aforementioned references, our starting point is to 
directly insert an infinite dimensional noise in the McKean dynamics 	\eqref{eq:MKV:1}}. To illustrate the thrust of this idea, one begins by choosing the underlying probability space to be the unit circle ${\mathbb S}$, equipped with the Lebesgue measure, denoted 
$\textrm{\rm Leb}_{\mathbb S}$. Thus,
\eqref{eq:MKV:1}
becomes
\begin{equation}
	\label{eq:MKV:2}
	\tfrac{\d}{\d t} X_t(x)= -\partial_\mu V\bigl(\text{Leb}_{\mathbb S} \circ X_t^{-1},X_t(x)\bigr),\quad x\in\bS.
\end{equation}
Intuitively, the next step is to perturb the resulting gradient flow by an Ornstein-{Uhlenbeck} process with values in 
$L^2({\mathbb S})$, namely the stochastic heat equation driven by a coloured noise. As explained in 
\cite{delarueHammersley2022rshe}, this approach would result in a non-Markovian evolution for the distribution of solutions. Critically, two representatives $X_0 : {\mathbb S} \rightarrow {\mathbb R}$
and $X_0' : {\mathbb S} \rightarrow {\mathbb R}$ of the same distribution 
$\mu_0 \in {\mathcal P}_2({\mathbb R})$ (i.e., $\textrm{\rm Leb} \circ X_0^{-1}=
\textrm{\rm Leb} \circ (X_0')^{-1} = \mu_0$) would lead to two different flows 
in ${\mathcal P}_2({\mathbb R})$. 
	
	To overcome this drawback, the key idea in \cite{delarueHammersley2022rshe}
	is to constrain the stochastic heat equation to be valued in a specific subset of functions in one-to-one correspondence with the space of one dimensional probability distributions. Those functions 
	are called symmetric non-decreasing functions on the circle. They should 
	be regarded as `quantile functions' defined on ${\mathbb S}$, as opposed to $[0,1]$, the periodic structure of the circle here providing a more robust framework.
	One of the main results of  \cite{delarueHammersley2022rshe} is to prove that the constraint 
	put on the stochastic heat equation can be rigorously formalised in terms of a suitable reflection term. 
	Transposing this to \eqref{eq:MKV:2}, this leads us to consider
\begin{equation}
\label{eq:MKV:3}
\d X_t(x)= -\partial_\mu V\bigl(\text{Leb}_{\mathbb S}\circ X_t^{-1},X_t(x)\bigr) \d t
+ \Delta_x X_t(x) \d t+ \d W_t(x)+ \d \eta_t(x), \ \ x \in {\mathbb S}, \ t \geq 0,
\end{equation}
with $\eta$ playing the rôle of a reflection term and $W$ being {a coloured} Brownian motion with values in 
$L^2({\mathbb S})$, see \cite{daPratoZabczyk2014stochEqnsInfDim}. 
When expressed in the Fourier basis, 
the {covariance of $W$ is} diagonal. 
Its form is clarified in Subsection \ref{sec defRSHE}. Briefly put, the eigenvalues are assumed to decay at a certain rate with the Fourier modes. 
 
\vskip 6pt

\noindent \textit{Contributions of the paper.}
As the first contribution of this article, weak solutions to the \emph{drifted} version 
\begin{equation}
\label{eq:MKV:3:bb}
\d X_t(x)= F\bigl(X_t(\cdot)\bigr)(x) \d t
+ \Delta_x X_t(x) \d t+ \d W_t(x)+ \d \eta_t(x), \ \ x \in {\mathbb S}, \ t \geq 0,
\end{equation}
\color{black}
of the rearranged stochastic heat equation are constructed via Girsanov transformations in Theorem \ref{thm eu w drift}. Uniqueness in law is demonstrated as a corollary to the strong --pathwise-- solvability properties of the \emph{undrifted} version originally introduced in 
\cite{delarueHammersley2022rshe}. The existence and uniqueness of solutions holds under modest assumptions of {finiteness of} all polynomial moments of the $L^2$ norm of
{the} initial `symmetrised quantile' (Assumption \ref{ass init l2}), allowing for a random initialisation and a symmetric - but not necessarily non-increasing - drift $F$ (Assumption \ref{ass drift form}) 
{mapping $L^2$ into itself }and satisfying linear growth in the $L^2$ norm for both itself and its spatial $x$-derivative  (Assumption \ref{ass drift eu}).
{Notably, this covers a wide class of `mean field' drifts of the form $F(\textrm{\rm Leb}_{\mathbb S} \circ X_t^{-1},X_t(x))$ that may not be continuous in the measure argument, 
hence entailing a phenomenon of restoration of weak uniqueness under the action of $W$.}

The second contribution is to demonstrate the exponential convergence to equilibrium for this system in Theorem \ref{thm expStab}. This requires some addition assumptions. We first define a condition (Assumption \ref{ass drift erg}) that is tantamount to having the squared $L^2$ norm satisfy a Lyapunov-type condition with negative prefactor. The final additional assumption (Assumption \ref{ass drift bounded oscillation}) is that of bounded oscillations at bounded distance for the drift; a sufficient condition for this to hold is Lipschitz continuity of the drift \emph{outside} of a ball. 
Notably, exponential convergence holds with any minimal threshold required on the intensity of the noise (with respect to the magnitude of the drift) nor any Lipschitz/monotone structure imposed on the drift. This makes a stark contrast with most of the existing state of the art in the long-time analysis of 
(standard) McKean-Vlasov equations\footnote{{In addition to the aforementioned references on the long-time behaviour of second-order Fokker-Planck equations 
deriving from an interaction potential plus an entropy penalty,} 
see \cite{ahmedDing1993invMeasNLMP,
benachourRoynetteTalayVallois1998,
veretennikov2006ergMKV} for earlier works and \cite{bogachevRocknerShaposhnikov2019convStnryNLFPK,
bogachevRocknerShaposhnikovStationary2019,
butkovsky2014ergPropNLMCandMKV,
eberleGuillinZimmer2016quantHarris4diffnsAndMKV,
guillinLeBrisMonmarche2022convRateMKVnonCon,guillinMonmarche2020UniformLongTimeANDpoCforMKVkineticNon-convex} for some recent developments regarding convergence to equilibrium in various mean field contexts.
 We refer also to \cite{baoScheutzowYuan2022existInvMeasFctnlMKV,
duongTugaut2016stationaryMKVexistCharPhaseTrans,
hammersleySiskaSzpruch2021lyapunov} for existence results or examples of (potentially non-unique) equilibria for non-monotone models.} and this strongly advocates for the relevance of the dynamics introduced here (on the basis of \cite{delarueHammersley2022rshe}). 
{Our proof mainly relies on a coupling argument 
that makes it possible to apply Harris' theorem, see Theorem
\ref{thm Harris}. In this regard, it is purely probabilistic and differs from 
arguments based on functional estimates like Poincaré or logarithmic-Sobolev inequalities, as used 
in
 \cite{doringStannat2009logSobolevForWassDiff}
in the long-time analysis of the Wasserstein diffusion of von Renesse and Sturm and in \cite{Stannat,
stannat2000onValidityLogSobSymFlemViot} for results in the context of the Fleming-Viot processes.}

The third contribution is to address the stochastically perturbed gradient descent 
\eqref{eq:MKV:3}, {which is a particular case of \eqref{eq:MKV:3:bb}}. 
{A natural question is to} 
wonder about the properties
of the invariant measure in connection with the original shape of the potential $V$. As explained in Section 
\ref{sec SGD}, the invariant measure cannot be identified as Gibbsian. The colouring and constraint of the dynamics break down the symmetry properties that would be necessary to characterise the invariant measure further. We are however able to provide metastability properties of the stochastic gradient descent when the noise is sent to zero in a proper fashion. 
Part of the challenge here is to identify just how the noise must be diminished to obtain some interesting metastability properties. 
To do just this, we add a multiplicative intensity factor $\varepsilon$ in front of the noise $W$ in \eqref{eq:MKV:3},  
rescale the Laplace friction term by $\varepsilon^2$ and then require the {first eigenvalues of the covariance matrix of $W$, up to a rank of order $\varepsilon^{-1}$, to be of order 1} (so that the noise becomes almost white). In this regime, 
we prove that the stochastic gradient descent spends times of order $\exp(1/\varepsilon^2)$ 
at the bottom of wells of the potential $V$, see Theorem 
\ref{thm:meta}, which follows the traditional folklore on stochastic gradient descents in finite dimension. 
{The estimate is established under a suitable condition on the shape of the potential's well: typically, it is satisfied if the potential is assumed to be 
locally strictly \textit{displacement} convex, i.e.,
strictly convex  
with respect to 
the $2^{nd}$-Wasserstein distance) and is sufficiently regular. Although seemingly natural, this assumption requires however a modicum of care.  
We indeed prove below the following 
surprising result:
functions that are locally 
strictly convex and continuously differentiable in Lions' sense feature a strong form of rigidity; in short, the minima of such strictly convex wells (which are local minima on the whole space) 
are necessarily attained at Dirac masses. This property, which has no analogue in the Hilbertian setting (on a Hilbert space, the squared distance provides a trivial example of  
a strictly convex smooth function `centred' around any arbitrary vector), is in fact inherited, in the specific framework of ${\mathcal P}_2({\mathbb R})$, 
from the Jensen inequality. Meanwhile, it 
advocates for a weaker notion differentiability. Consistently, we just require the lift of the potential $V$ to be differentiable 
\emph{inside} the sole cone of quantile functions, which is weaker than asking it to be differentiable on the whole space of square-integrable random variables (as in the case of the Lions derivative). This weakening of the notion of derivative allows us to treat, in particular, potentials that coincide locally with the square of the $2^{nd}$-Wasserstein distance, $W_2(\cdot,\nu)$ for arbitrary $\nu\in\cP_2(\mathbb R)$. 
It also serves to further justify the rearranged model that we are exploring in this work.}
\vskip 6pt
  
\noindent \textit{Prospects.}
Our work should be considered as a proof of concept,
demonstrating the advantages to be gained from randomising McKean dynamics, in particular those deriving from a potential, by means of an appropriate infinite-dimensional (common) noise.
We believe this opens up promising research directions, which we now briefly highlight.

{In fact, our long-term project is to show that the exploratory properties of a common noise like the one used here can be of real interest in the study of mean-field minimisation problems, both algorithmically and numerically. In this respect, one of our objectives is to understand whether common noise can offer a relevant alternative to entropy regularization methods, some of which are limited in scope as we pointed out above. Naturally, an ambitious perspective would be to revisit mean-field approaches to artificial neural networks
(e.g., \cite{DBLP:conf/nips/ChizatB18,meiMontanariNguyen2018MFNN}) in such a context. Of course, much remains to be done before such applications can be envisaged in concrete terms. On a theoretical level, it is clear that the question becomes relevant in the higher dimensional framework. Any attempt to generalise to the $d \geq 2$ case requires an appropriate adaptation of the (non-drifted) rearranged stochastic heat equation, which we plan to study in the future. Moreover, the same question but for McKean-Vlasov models subjected to a private noise looks to be another very interesting though challenging problem.
For the $1d$ model addressed here, the metastability properties stated in Theorem 
\ref{thm:meta} {indicate} that our stochastic gradient descent does indeed possess exploration properties consistent with the landscape drawn by the potential. {Consequently,} a large deviation analysis for the same small intensity regime of the stochastic gradient descent would be highly valuable. In particular, analysis of the appropriate quasi-potential of Friedlin and Wentzell is a very natural prospective.} 
 \vskip 6pt
\color{black}

\noindent \textit{Organisation.}
The article is structured as follows. Section \ref{sec prelim} provides definitions and a short introduction to the rearranged stochastic heat equation. Section \ref{sec exist} establishes the existence and uniqueness of the drifted rearranged stochastic heat equation in the probabilistically weak sense. In Section \ref{sec conv}, we prove the exponential convergence to equilibrium for the system under study. 
In Section \ref{sec SGD}, we 
focus on some of the specific properties of the stochastic gradient descent. 
The metastability properties of the latter in the vanishing viscosity regime are addressed in Section \ref{sec ito}.

\section{Preliminary Material}
\label{sec:rearrangement}
In this section, some definitions and related results are recalled. For more details, see \cite{delarueHammersley2022rshe} and the cited literature. To begin, the circle {${\mathbb S}$} is parametrised by the interval $(-1/2,1/2]$ and with $0$ viewed as a privileged fixed point, {i.e., ${\mathbb S}:=({\mathbb R}+1/2)/{\mathbb Z}$}. 

\begin{proposition}
	\label{prop:rearrangement:def}
	Given a measurable function $f : \bS \rightarrow {\mathbb R}$, there exists a unique function, called the symmetric non-increasing rearrangement of $f$ and denoted $f^* : \bS \rightarrow [-\infty,+\infty]$, satisfying the following two properties:
	\begin{enumerate}
		\item $f^*$ is symmetric (with respect to $0$), is non-increasing and right-continuous on the interval $[0,1/2)$,
		and is left-continuous at $1/2$ (left- and right-continuity being understood for the topology on $[-\infty,+\infty]$),
		\item the image of the Lebesgue measure $\textrm{\rm Leb}_{{\mathbb S}}$
		by $f^*$ is the same as that by $f$, namely
		\begin{equation*}
			\forall a \in {\mathbb R}, \quad \textrm{\rm Leb}_{{\mathbb S}}\Bigl(\bigl\{ x \in {\mathbb S} : f^*(x) \leq a \bigr\} \Bigr) 
			= 
			\textrm{\rm Leb}_{{\mathbb S}}\Bigl(\bigl\{ x \in {\mathbb S} : f(x) \leq a \bigr\} \Bigr).
		\end{equation*} 
	\end{enumerate}
\end{proposition}
Those functions $f^*$ that are both square integrable with respect to Lebesgue measure on $\bS$ and satisfy the first item in the definition above are in one-to-one correspondence with the set ${\mathcal P}_2({\mathbb R})$ of probability measures on ${\mathbb R}$ having finite-second moment. See Baernstein \cite{baernstein2019symmetrizationInAnalysis} for further details and in particular Definition 1.29 therein for symmetric rearrangements in the Euclidean setting and Chapter 7 in the same book for a treatment of spherical symmetric rearrangements.
For a probability measure $\mu \in {\mathcal P}_2({\mathbb R})$, we then call $F_{\mu}^{-1}$ the pre-image of 
$\mu$ by the mapping that sends functions $f^*$ (as above) onto the probability measure $\textrm{\rm Leb}_{\mathbb S} \circ 
(f^*)^{-1}$. 

 The following definition will be used throughout:
\begin{definition}
	\label{def:1.2:U2}
	A function $f  : {\mathbb S} \rightarrow {\mathbb R}$ 
	is called symmetric non-increasing if $f=f^*$.  
	The collection of equivalence classes in $L^2({\mathbb S})$ containing a symmetric non-increasing function is denoted  by $U^2({\mathbb S})$. 
\end{definition} 
In what follows, elements of $L^2(\bS)$ that are Lebesgue almost everywhere symmetric comprise the set $L^2_{\rm sym}({\mathbb S})$. 
One of these elements is said to be \emph{non-increasing} (we refrain from tautological use of the word symmetric and believe that given the context of the circle, this should not be a source of confusion) if it coincides almost everywhere with an element of 
$U^2({\mathbb S})$. Notice that we may choose the latter representative to {be} uniquely defined as a symmetric non-increasing function. Thus, two elements of $U^2({\mathbb S})$ that coincide in $L^2({\mathbb S})$ 
coincide in fact everywhere on ${\mathbb S}$ (courtesy of the left- and right-continuity properties). 
Also, the following proposition is of clear importance. 
\begin{proposition}[Preservation of $L^p$ norms]
	\label{lem:isometry}
	With the same 
	notations as in Proposition 
	\ref{prop:rearrangement:def}, we have, for any $p \in [1,\infty]$, 
	$ \lVert f^* \rVert_p = \lVert f  \rVert_p$.
\end{proposition}


\subsection{The Heat Kernel and the Rearrangement Operator}
Critical to the construction {presented in \cite{delarueHammersley2022rshe}}
 is the composition of the rearrangement operation and the heat kernel. Two basic properties of this composition are recalled in the lemma below. 
Firstly, the periodic heat semigroup (with specific diffusivity parameter $1$) on the circle $\mathbb{S}$, which we denote $(e^{t\Delta_x})_{t \geq 0}$, 
{with kernel 
$(\Gamma(t,x))_{t \geq 0, x \in {\mathbb S}}$ 
that satisfies}

\begin{lemma}
	\label{lem:heatkernelisdecreasing}
	For any $t>0$, the function $x \mapsto \Gamma(t,x)$ is non-increasing on $(0,1/2)$ and non-decreasing on 
	{$(-1/2,0)$}, {i.e., 
	$\Gamma(t,\cdot)$ is equal to $\Gamma(t,\cdot)^*$.}
	
Furthermore, letting $f$ belonging to $L^1({\mathbb S})$ be symmetric non-increasing on the circle (i.e., $f=f^*$), then for any $t>0$, the convolution $f \star \Gamma(t,\cdot)$ is also symmetric non-increasing on the circle.
\end{lemma}
 
\subsection{Some Notation}

\textit{Analysis on the torus.} 
Recall that $\bS$ denotes the circle parametrised by the interval of length $1$. Also, the functions
\begin{equation*}
	\begin{split}
		e_m^{\Re} : \bS \ni x \mapsto \sqrt{2} \cos(2 m \pi x), \quad 
		e_m^{\Im} :  \bS \ni x  \mapsto \sqrt{2} \sin(2 m \pi x),  
	\end{split}
\end{equation*}
for any natural integer $m$, {with $e_0^{\Re} :\equiv 1$ and $e_0^{\Im} :\equiv 0$},  
form the complete Fourier basis on $L^2(\bS)$, where $L^2( \bS) $ is the space of square integrable functions on $\bS$. Due to the symmetry in our construct, we most use just the even (cosine) Fourier functions, which prompts us to use the shorter notation 
	$e_m$ for $e_m^{\Re}$.
	
	The Lebesgue measure on ${\mathbb S}$ is denoted $\textrm{\rm Leb}_{{\mathbb S}}$ with $\d \textrm{\rm Leb}_{{\mathbb S}}(x)$ written as $\d x$. For any real $p \geq 1$, we call $\| \cdot \|_p$
	the $L^p$ norm on 
	the space of measurable functions $f$ on $({\mathbb S},\textrm{\rm Leb}_{{\mathbb S}})$ with $\int_{\bS} |f(x)|^p dx<\infty$. 
	Similarly, when 
	$p=\infty$, the notation 
	$\| \cdot \|_\infty$ is used for the $L^\infty$ (supremum) norm, i.e. $\| f \|_\infty:=\textrm{essup}\{f(x):x\in\bS\}$. 
	{The inner product 
	between two elements $f$ and $g$ in $L^2({\mathbb S})$ is denoted $\langle f,g \rangle_2$, or $\langle f, g \rangle$, or  also $f \cdot g$ depending on the context, and more generally duality products between 
	functions and distributions on ${\mathbb S}$ is usually denoted $\langle \cdot, \cdot \rangle$.}
	For an element $f \in L^1(\bS)$ and for a non-negative integer $m$, we call $\hat{f}_m^{\Re} := \int_{\bS} f(x) e_m^{\Re}(x) \d x$ the 
	cosine
	Fourier mode of $f$ of index $m$ and $\hat{f}_m^{\Im} := \int_{\bS} f(x) e_m^{\Im}(x) \d x$
	the sine Fourier mode of $f$ of index $m$. {The Fourier mode $\hat{f}_0^{\Re}$ is sometimes denoted $\bar{f}$.}
	If $f$ is Lebesgue almost everywhere (written a.e. hereafter) symmetric, i.e. 
	$f(-x)=f(x)$ {a.e.}, all the sine Fourier modes are $0$ and we can just write 
	$\hat{f}_m=\langle f, e_m\rangle$ in place of $\hat{f}_m^{\Re}$. 
	In that case, $\hat{f}^m$ is a real number. 
	As 
	{we already mentioned}, 
	these symmetric 
	functions are used quite often in the text
	{and we} denote by $L^2_{\rm sym}({\mathbb S})$ the set of functions $f$ in 
	$L^2({\mathbb S})$ that are {a.e.} symmetric. More generally, for a parameter ${s} \in {\mathbb R}$, we denote by
	$H^{{s}}_{\rm sym}({\mathbb S})$ the Sobolev space of {a.e. }symmetric functions 
	({or Schwartz distributions if $s<0$})
	$f$ such that
	\begin{equation*}
		\| f \|_{2,\mu}^2 := \sum_{m \in {\mathbb N}_0} (m \vee 1)^{2 {s}}  \hat{f}_m^2 < \infty,
	\end{equation*}
	(${\mathbb N}$ is the collection of natural numbers 1, 2, 3, ... The set ${\mathbb N}_0$ is defined as 
	${\mathbb N}_0 = {\mathbb N} \cup \{0\}$). Of course, $H^0_{\rm sym}({\mathbb S})$ is nothing but $L^2_{\rm sym}(\bS)$. 
	The inner product on 
	$H^{\mu}_{\rm sym}({\mathbb S})$
	is denoted $\langle f,g \rangle_{2,\mu}:=\sum_{m\in\bN_0} (m \vee 1)^{2 \mu} \hat{f}_m\hat{g}_m$. 
	
	For any integer $k \geq 1$, we denote by 
	${\mathcal C}^k({\mathbb S})$ the space of $k$-times continuously differentiable functions on 
	${\mathbb S}$.  
\vskip 4pt

\noindent \textit{More functional spaces.}
The space of continuous functions from a metric space $\cX$ to another metric space $\cY$ 
is denoted $\cC(\cX,\cY)$.
The space of  bounded and measurable functions from a metric space $\cX$ to ${\mathbb R}$ is denoted $B_b(\cX)$. 
\vskip 4pt

\noindent \textit{Space of probability measures.} The space of probability measures on a Polish space ${\mathcal X}$ is denoted ${\mathcal P}({\mathcal X})$. It is sometimes equipped with the total variation 
distance $d_{\rm TV}$, defined by 
\begin{equation*} 
d_{\rm TV}(\mu,\nu) := \sup_{A \in {\mathcal B}({\mathcal X})} \vert \mu(A) - \nu(A)\vert,
\quad \mu,\nu \in {\mathcal P}({\mathcal X}),
\end{equation*}
where ${\mathcal B}({\mathcal X})$ is the Borel $\sigma$-field on ${\mathcal X}$. 
When ${\mathcal X}= {\mathbb R}$, we often restrict ourselves to the space ${\mathcal P}_2({\mathbb R})$ of probability measures on ${\mathbb R}$ with a finite second moment. 
We equip ${\mathcal P}_2({\mathbb R})$ with the $2$-Wasserstein distance 
${\mathcal W}_2$ defined by 
\begin{equation*} 
{\mathcal W}_2(\mu,\nu) = \inf_{\pi} \biggl[ \int_{{\mathbb R} \times {\mathbb R}} \vert x-y \vert^2 \d \pi(x,y) \biggr]^{1/2},
\quad \mu,\nu \in {\mathcal P}_2({\mathbb R}), 
\end{equation*} 
with the infimum being taken over probability measures $\pi$ over ${\mathbb R}^2$ admitting $\mu$ and $\nu$ as first and second marginal laws. 

The law of a random variable $X$ (with values in a Polish space ${\mathcal X}$) is denoted ${\mathscr L}(X)$. 
\vskip 4pt

\textit{Derivatives.} 
	There are several different derivatives used in this text, and thus the following notations are included. For a differentiable real-valued function on $\bS$, we write $\nabla_x f(x)$ for the \emph{spatial} derivative of $f(x)$. Additionally,
	we let $\Delta_x := \nabla_x^2$ {(sometimes, this is just written as $\Delta$)}. 
	
	When dealing with a function $F : L^2_{\rm sym}({\mathbb S}) \rightarrow {\mathbb R}$, the derivative is understood in the usual Fréchet sense. 
	The gradient (at an element $X \in L^2_{\rm sym}({\mathbb S})$) is denoted $D_X F(X)$ and is regarded as an element of 
	$L^2_{\rm sym}({\mathbb S})$. When 
	$F:
	{\mathcal P}_2({\mathbb R}) \rightarrow {\mathbb R}$, the derivative function is understood in Lions' sense, see Lions' lectures at the collège de France  \cite{LionsVideo} together with 
Chapter 5 in the 
book \cite{CarmonaDelarueI}. In short, the function $F$ is assumed to have a \emph{lift},
\begin{equation*} 
L^2_{\rm sym}({\mathbb S}) \ni X \mapsto F\bigl({\mathscr L}(X)\bigr)\in \bR,  
\end{equation*} 
that is Fréchet differentiable (here ${\mathscr L}(X)$ is nothing by $\textrm{\rm Leb}_{\mathbb S} \circ X^{-1}$). We abuse notation to write the lifted function as $F$. Its derivative {(at X)} has the form 
\begin{equation}
\label{eq:Lions:derivative}
  D_X F (X)(x) {:}= \partial_{\mu} F \bigl( {\mathscr L}(X) \bigr) \bigl( X(x) \bigr), \quad x \in {\mathbb S},
\end{equation}
where 
\begin{equation*} 
\Bigl( \partial_\mu F\bigl( {\mathscr L}(X) \bigr) : {\mathbb R} \ni y \mapsto \partial_\mu F\bigl( {\mathscr L}(X) \bigr)(y) \Bigr) \in L^2\bigl( {\mathbb R}, {\mathscr L}(X) ). 
\end{equation*} 
We sometimes call the functional $\partial_{\mu}F$ the Wasserstein derivative of $F$. In order to distinguish notationally the derivative on $\bR$ from the derivative on the circle, we will often use the parameter $y$ on the real line, or more generally on the Euclidean space $\bR^d$. Therefore, mixed derivatives will be denoted $\nabla_y\partial_\mu$.
\vskip 4pt 	 
\textit{Stochastic processes.} 
For a stochastic process $Y\equiv Y_\cdot=(Y_t)_{t\in [0,\infty)}$ with values in ${\mathbb R}$, its bracket or quadratic variation at time $t$ is denoted 
${\llangle Y \rrangle_t}$.
For a local martingale $(M_t)_{t \in [0,\infty)}$ (with values in ${\mathbb R}$), we call $({\mathcal E}_t(M))_{t \geq 0}$ its Doléans-Dade exponential, defined by 
\begin{equation*} 
{\mathcal E}_t(M) = \exp \bigl( M_t - \tfrac{1}{2} {\llangle M \rrangle}_t \bigr), \quad t \geq 0. 
\end{equation*}

\textit{Miscellaneous.} 	
	We now introduce some standard notations. For a real number $x$, we write $\lfloor x \rfloor$ for the floor of $x$,
	$\lceil x \rceil$ for the ceiling of $x$ and $x_+:=\max(x,0)$ 
	(resp. $x_- = \min(-x,0)$)
	for the positive (resp. negative) part of $x$.  
	For two reals $x$ and $y$, we let $x \vee y := \max(x,y)$ and $x \wedge y := \min(x,y)$.

	As for constants that are used in the various inequalities, they are
	usually written in the form $c_{a,b}$ or $C_{a,b}$, where the subscripts are quantities on which the current constant depends,
	and are implicitly allowed to vary from line to line (as long as they just depend on the same quantities $a$ and $b$).

\subsection{Definition of the Rearranged Stochastic Heat Equation}\label{sec defRSHE}
\label{sec prelim}
A noise, coloured in space to be $L^2$ valued, is defined by
\begin{equation}
	\label{def:tilde:noise}
	\begin{split}
		{W}_t :=  \sum_{m\in \bN_0}\lambda_m \beta^m_t e_m, \quad t \geq 0, 
	\end{split}
\end{equation}
where $\{\beta^m\}_{m\in \bN_0}$ is a family of independent Brownian motions. The colouring coefficients, $\{ \lambda_m \}_{m \in {\mathbb N}_0}$, are assumed strictly positive and such that, 
for some parameter $\lambda \in {(\nicefrac12,1)}$, and some sufficiently large $M_{\textrm{\rm col}}$, there exist constants $c_{\textrm{\rm col}}$ and $C_{\textrm{\rm col}}$ such that 
\begin{equation}
	\label{eq colouring} \ 
\forall m\geq M_{\textrm{\rm col}},\quad	c_{\textrm{\rm col}} m^{-\lambda} \leq \lambda_m \leq C_{\textrm{\rm col}} m^{-\lambda},
\end{equation}
{which allows one to let 
\begin{equation}
\label{eq:c_|ambda}
c_\lambda:=\sum_{m\in \bN_0}\lambda_m^2 < \infty.
\end{equation}
Also, without any loss of generality, we assume that $\sup_{m \in {\mathbb N}_0}  \lambda_m \leq 1$.}
 
{Returning to \eqref{eq:MKV:3:bb}
we consider the following hypothesis on the drift $F$}:
\begin{assumption}
  \label{ass drift form}
The coefficient $F$ reads in the form $	L^2_{\rm sym}(\bS) \ni {u :\bigl( {\mathbb S} \ni x \mapsto u (x)\bigr)} \mapsto  {F(u):\bigl(  {\mathbb S} \ni x \mapsto F(u)(x)\bigr)}\in L^2_{\rm sym}(\bS)$.
\end{assumption}
 
\begin{definition}
	\label{def:existence} For fixed $ {T \in (0,\infty]}$, set $I:={[0,T] \cap {\mathbb R}}$. On a given probability space $(\Omega,{\mathcal A},{\mathbb F},{\mathbb P})$ (with filtration $\bF$ satisfying the usual conditions) equipped with 
	an $L^2_{\rm sym}({\mathbb S})$-valued Brownian motion $(W_t)_{t \in I}$ 
	{of the form 
		\eqref{def:tilde:noise}
		(with 
		$\{\beta^m\}_{m\in \bN_0}$ being a family of  independent ${\mathbb F}$-Brownian motions)} 
	and with 
	an ${\mathcal F}_0$-measurable initial condition $X_0$, valued in $
	U^2({\mathbb S})$ (see 
	Definition \ref{def:1.2:U2}) having finite moments 
	of any order, say that a pair of processes  
	$(X_t,\eta_t)_{t \in I}$ solves the drifted
	rearranged stochastic heat equation driven by 
	$(W_t)_{t \in I}$ and started from $X_0$ up to time $T$ if 
	\begin{enumerate}
		\item $(X_t)_{t \in I}$ is a continuous ${\mathbb F}$-adapted process with values in $
		U^2({\mathbb S})$;
		\item $(\eta_t)_{t \in I}$ is a continuous ${\mathbb F}$-adapted process with values in $H^{-2}_{\rm sym}({\mathbb S})$, started from $0$, such that, $\bP$-almost surely, for any $u \in H^2_{\rm sym}({\mathbb S})$ that is non-increasing, the path
		$(\langle \eta_t,u \rangle)_{t \in I}$ is non-decreasing;
		\item $\bP$-a.s., for any $u \in H^{2}_{\rm sym}({\mathbb S})$, for all $t \in I$, 
		\begin{equation}\label{def RSHE 1}
			\begin{split}
				\langle  X_t,u \rangle    =   \langle X_0,u\rangle + \int_0^t \Bigl(  \langle F(X_r),u \rangle + \langle   {X}_r  ,\Delta     u \rangle \Bigr) \d r +\langle W_t  ,u \rangle + \langle \eta_t ,u\rangle;
			\end{split} 
		\end{equation}
		\item for any $t \in I$, 
		\begin{equation}
			\label{def RSHE 2}
			\lim_{\varepsilon\searrow0}\mathbb{E}\left[ \int_0^t  e^{\varepsilon\Delta}X_r   \cdot \d  \eta_r \right]= 0.
		\end{equation}
	\end{enumerate}
\end{definition} 
At times we say that a process or pair of processes satisfies \eqref{def RSHE 1}, this should be read to mean that it satisfies all the conditions of Definition \ref{def:existence} providing a solution to the (possibly drifted) rearranged stochastic heat equation. 

The integral in \eqref{def RSHE 2} is constructed in a path-by-path sense as follows {(which, for simplicity, we spell out in the case $I=[0,\infty)$)}: 

\begin{definition}
	\label{def:integral:RS}
	For any  two (deterministic) curves 
	$(n_t)_{t \geq 0}$ 
	and 
	$(z_t)_{t \geq 0}$, such that 
	\begin{enumerate}
		\item
		$t \mapsto n_t$ is a continuous function 
		from 
		$[0,\infty)$ to 
		$H^{-2}_{\rm sym}({\mathbb S})$;
		\item 
		for any 
		non-increasing 
		$u \in H^2_{\rm sym}({\mathbb S})$, the 
		function $[0,\infty) \ni t  \mapsto  \langle n_t , u \rangle$ is non-decreasing;
		\item $(z_t)_{t \geq 0}$ 
		{being a continuous function from 
			$[0,\infty)$ to  $L^2_{\rm sym}({\mathbb S})$},
	\end{enumerate} 
 one may define, almost surely, for 
	any $\varepsilon>0$  
	the integral process
	$( \int_0^t   e^{\varepsilon \Delta} z_r \cdot \d n_r )_{t \geq 0}$
	as the following limit, for the uniform topology on {every $[0,T]$},
	\begin{equation*}
		\int_0^t   e^{\varepsilon \Delta} z_r \cdot \d n_r
		:=
		\lim_{M \rightarrow \infty} 
		\sum_{m=0}^M 
		\biggl( 
		\int_0^t   \bigl\langle e^{\varepsilon \Delta} z_r,e_m \bigr\rangle  \d \langle n_r,e_m^+ \rangle
		- 
		\int_0^t   \bigl\langle e^{\varepsilon \Delta} z_r,e_m \bigr\rangle  \d \langle n_r,e_m^- \rangle
		\biggr).
	\end{equation*}
	{The} integrals inside the sum on the right hand side are constructed as Stieltjes integrals since one may indeed decompose each $(e_m)_{m \in {\mathbb N}_0}$, as the difference of two symmetric non-increasing functions.	
The integral process satisfies the estimate:
	\begin{equation}
		\label{eq:integral:RS:bound:0:bis}
		\forall T \geq 0,
		\quad 
		\sup_{t \in [0,T]} 
		\biggl\vert 
		\int_0^t  e^{\varepsilon \Delta} z_r
		\cdot
		\d   n_r  
		\biggr\vert 
		\leq {c_{\varepsilon}} \| n_T \|_{2,-2} \times \sup_{t \in [0,T]} \| z_t \|_2.
	\end{equation}

For further details of this construction, see \cite{delarueHammersley2022rshe}. 
\end{definition} 

%% file: sections/EU4Drifted.tex
 
\subsection{Assumptions}

{In this section, we study 
\eqref{def RSHE 1} under the two assumptions below.} 

\begin{assumption}[Polynomial Moments for Initial Condition]\label{ass init l2}
The initial distribution $\mu_0$ on $U^2(\bS)$
satisfies 
$\int_{U^2(\bS)}  \| h \|^p_2  d\mu_0 (h) < \infty$, for 	any $p \geq 1$.
\end{assumption} 

\begin{assumption}[Growth of Drift]\label{ass drift eu}
The drift, denoted $F$, satisfies the following growth assumptions:
	\begin{align}
	&\lVert F(u) \rVert_2^2  \leq c_F(1+\lVert u \rVert_2^2), \qquad u \in 
	U^2({\mathbb S}), 
	\label{eq bd4der:1}
	\\
	& \lVert \nabla_x  ( F(u) ) \rVert_{2 }^2 
		   \leq 
		    C_F\left( 1+ \lVert \nabla_x u \rVert_{2}^2\right), \qquad u \in H^1_{\rm sym}({\mathbb S}) \cap U^2({\mathbb S}), 
		    \label{eq bd4der}
		    \end{align} 
\end{assumption}
 
\begin{remark}
As spelled out above, it suffices to define $F$ on $U^2({\mathbb S})$, but, in fact, one could easily extend $F$ to the entire $L^2_{\rm sym}({\mathbb S})$
in a way that would preserve Assumption 
\ref{ass drift eu}. For instance, one can consider $F(u^*)$, for $u \in L^2_{\rm sym}({\mathbb S})$. 
Condition 
\eqref{eq bd4der:1} would remain true for $u \in L^2_{\rm sym}({\mathbb S})$ thanks to 
Proposition 
\ref{lem:isometry}. 
Condition 
\ref{eq bd4der} would remain true for $u \in H^1_{\rm sym}({\mathbb S})$ thanks to the so-called 
Pólya–Szegő inequality, which says that $\lVert \nabla_x u^* \rVert_2 \leq \lVert \nabla_x u \rVert_2,$
see for instance \cite[Theorem 7.4]{baernstein2019symmetrizationInAnalysis}.
For example, 
such an extension could be used in the study 
of the (corresponding) construction 
scheme addressed in 
\cite{delarueHammersley2022rshe}. 
\end{remark}

\color{black}

 Below, we 
{provide {a priori} results and estimates satisfied by 
solutions to  {\eqref{eq:MKV:3:bb}} (if any). The reader may skip the proofs on a first reading.}
Recalling several results from \cite{delarueHammersley2022rshe}, we first 
obtain the very useful expansion (which we use repeatedly in the paper):
\begin{proposition}\label{prop ito fsquare}
Let Assumptions \ref{ass drift form}, \ref{ass init l2} and \ref{ass drift eu} hold. Then, 
{for $(X_t)_{0 \le t \le T}$ a solution to the (possibly drifted) rearranged SHE {on a certain $[0,T]$} (as in Definition \ref{def:existence})},  with probability 1, {$(X_t)_{0 \le t \le T}$
 takes
 values in $L^2([0,T],H^1_{\rm sym}({\mathbb S}))$}
 and,  for any $t\in [0,T]$,
\begin{equation}
\label{ito:square}
\begin{split} 
	\d \bigl( f_t \lVert X_t \rVert_2^2 \bigr) &= 
	\bigl( 2 f_t  \langle X_t , F(X_t) \rangle 
	  -2 f_t \lVert \nabla_xX_t\rVert_2^2
	   + c_\lambda f_t
	+ \dot{f}_t \|X_t \|_2^2 \bigr) \d t  
	  +2 f_t \langle X_t  , \d W_t \rangle,  
	 \end{split}
\end{equation}
where $f$ is any differentiable function from $[0,\infty)$ to ${\mathbb R}$, with 
$(\dot{f}_t)_{t \geq 0}$ as derivative, and 
$c_\lambda$ is given by 
\eqref{eq colouring}. 
\end{proposition}
\begin{proof}
The proof follows from a combination of arguments found in \cite{delarueHammersley2022rshe}. Namely, for each $m\in\bN_0$ and arbitrary $\varepsilon >0$, one applies the classical Itô formula to $f_t(\langle e^{\varepsilon\Delta}e_m,X_t\rangle)^2$. Then, as in the proof of 
{Corollary 4.12} in \cite{delarueHammersley2022rshe}, one may sum over the Fourier modes and remove the constant $\varepsilon$. This yields,
\begin{align} 
\label{eq:ito:edeltaX2}
&\| e^{\varepsilon \Delta} X_t \|_2^2
- 
\| e^{\varepsilon \Delta} X_0 \|_2^2
+ \int_0^t \| e^{\varepsilon \Delta} \nabla_x X_s \|_2^2 \d s
\\
&= 2 \int_0^t \langle F(X_s), e^{2 \varepsilon \Delta} X_s \rangle  \d s 
+  \| e^{\varepsilon \Delta} W_t \|_2^2 + 2 \int_0^t \langle e^{2\varepsilon \Delta} X_s, 
\d \eta_s \rangle
+ 2 \int_0^t \langle e^{2 \varepsilon \Delta} X_s, \d W_s \rangle, \nonumber
\end{align} 
for $t \geq 0$. One then concludes by letting $\varepsilon$ tend to $0$, using 
positivity of the integral with respect to $\eta$ (item 2 in Definition \ref{def:existence}) 
together with item 4 in Definition 
\ref{def:existence}  
\color{black}
\end{proof}
  
\begin{proposition}\label{prop zeroMode}
Let Assumptions \ref{ass drift form}, \ref{ass init l2} and \ref{ass drift eu} hold. For any $\delta >0$
and for $\alpha =  { \sqrt{c_F}} + 4 \delta$, it holds, 
for 	$(X_t)_{0 \le t \le T}$ a solution to the (possibly drifted) rearranged SHE {on an interval $[0,T]$} (as in Definition \ref{def:existence}), {with probability 1} 
and
for any $0 \leq t_1 \leq  t_2 \leq T$,
\begin{equation*} 
\begin{split}
&{\mathbb E} \biggl[ \exp \biggl( 2 \delta  \int_{t_1}^{t_2}
\exp(-{2} \alpha s) \Bigl[ {4 \delta} \| X_s \|_2^2 + \| \nabla_x X_s \|_2^2 \Bigr] \d s \biggr) 
 \, \vert \, {\mathcal F}_{t_1}\biggr]
\\
&\hspace{15pt} 
 \leq  \exp\Bigl(  2 \bigl( c_\lambda 
	+
	 {2 \sqrt{ c_F}}
	\bigr)  \delta (t_2-t_1) \Bigr)   \exp \bigl( 2 \delta   \| X_{t_1} \|^2_2 
 \bigr).
 \end{split}
\end{equation*} 
\end{proposition} 
\begin{proof}
Without any loss of generality, assume 
$t_1=0$. 
We then apply \eqref{ito:square} with $f_t =  \exp(- {2} \alpha t)$ for some $\alpha>0$. 
Using the bound $\langle u,F(u) \rangle \leq \sqrt{c_F} \| u \|_2 \sqrt{1+ \| u \|_2^2} \leq 
\sqrt{c_F}(1+\| u \|^2)$, we obtain 
\begin{align}
\label{eq:ito:expansion:polynomial} 
	&\d \bigl( f_t \lVert X_t \rVert_2^2 \bigr) 
	\\
	&{\leq} - 2 f_t \Bigl[ \bigl(  \alpha - 
	{\sqrt{c_F}} \bigr) \| X_t \|_2^2 + \| \nabla_x X_t \|_2^2 \Bigr] \d t + 
	\Bigl( c_\lambda 
	+
	{2  \sqrt{c_F}}
	\Bigr) 
	f_t \d t + 2 f_t \langle X_t  , \d W_t \rangle. \nonumber
\end{align}
Integrating in time from $t \wedge \tau$, for an arbitrary stopping time $\tau$, and then exponentiating, we deduce that, for any $\delta >0$,  
\begin{equation} 
\label{eq:ito:square:ft}
\begin{split}
&{\mathbb E} \biggl[ \exp \biggl( 2 \delta  \int_0^{t \wedge \tau} 
 f_s \Bigl[ \bigl(  \alpha -    {\sqrt{c_F}}   \bigr) \| X_s \|_2^2 + \| \nabla_x X_s \|_2^2 \Bigr] \d s \biggr) 
 \, \vert \, {\mathcal F}_0
 \biggr]
 \\
 & \leq
 \exp \biggl( 	\Bigl( c_\lambda 
	+
	 {2 \sqrt{ c_F}}
	\Bigr)  \delta
 \int_0^t f_s \d s \biggr)
 \times  
  {\mathbb E} \biggl[ \exp \biggl(  \delta \| X_0 \|^2_2 
 +
 2 \delta \int_0^{t \wedge \tau} f_s \langle X_s, \d W_s \rangle 
 \biggr)
 \biggr].
\end{split}
\end{equation} 
Denoting by $(M_t)_{0 \le t \le T}$ the local martingale 
$(M_t:= \int_0^t f_s \langle X_s,\d W_s \rangle_2)_{0 \le t \le T}$, we recall the very standard estimate:
\begin{equation*} 
\begin{split}
{\mathbb E} \bigl[ \exp\bigl( 2 \delta M_{t \wedge \tau} \bigr) \, \vert \, {\mathcal F}_0 \bigr] 
&=
{\mathbb E} \bigl[ \exp\bigl( 2 \delta M_{t \wedge \tau} - 4 \delta^2 {\llangle M \rrangle}_{t \wedge \tau} \bigr) \exp\bigl( 4 \delta^2  {\llangle M \rrangle}_{t \wedge \tau} \bigr) \, \vert \, {\mathcal F}_0 \bigr] 
\\
&\leq 
{\mathbb E} \bigl[ \exp\bigl( 4 \delta M_{t \wedge \tau} - 8 \delta^2 {\llangle M \rrangle}_{t \wedge \tau} \bigr)  \, \vert \, {\mathcal F}_0 \bigr]^{1/2} 
{\mathbb E} \bigl[
\exp\bigl( 8 \delta^2 {\llangle M \rrangle}_{t \wedge \tau} \bigr) \, \vert \, {\mathcal F}_0 \bigr]^{1/2}
\\
&\leq {\mathbb E} \bigl[
\exp\bigl( 8 \delta^2 {\llangle M \rrangle}_{t \wedge \tau} \bigr) \, \vert \, {\mathcal F}_0 \bigr]^{1/2},
\end{split} 
\end{equation*} 
where $({\llangle M \rrangle}_t)_{0 \le t \le T}$ is the bracket process of $(M_t)_{0 \le t \le T}$ and can be upper bounded by 
\begin{equation*} 
{\llangle M \rrangle}_t \leq  
 \int_0^t f^2_s \| X_s \|_2^2 \d s. 
\end{equation*} 
Inserting into 
\eqref{eq:ito:square:ft}, we obtain 
\begin{equation*} 
\begin{split}
&{\mathbb E} \biggl[ \exp \biggl( 2 \delta  \int_0^{t \wedge \tau} 
 f_s \Bigl[ \bigl(  \alpha -    {\sqrt{ c_F}}  \bigr) \| X_s \|_2^2 + \| \nabla_x X_s \|_2^2 \Bigr] \d s \biggr) 
\, \vert \, {\mathcal F}_0 \biggr]
 \\
 & \leq \exp \Bigl( \bigl( c_\lambda 
	+
	 {2 \sqrt{c_F}}
	\bigr) \delta t \Bigr)   \exp \bigl(  \delta  \| X_0 \|^2_2 
 \bigr) 
 {\mathbb E} \biggl[ 
 \exp \biggl(
 8\delta^2 \int_0^{t \wedge \tau} f_s \| X_s \|^2_2 \d s
 \biggr)
 \, \vert \, {\mathcal F}_0
 \biggr]^{1/2}.
 \end{split}
\end{equation*} 
Setting now $\alpha = { \sqrt{c_F}}+ 4 \delta$ and then for $\tau$ a localising stopping time that guarantees that the right-hand side is finite, we easily get a bound for the left-hand side. Letting $\tau$ tend to $\infty$ (which is possible), 
we complete the proof. 
\end{proof}
 
\subsection{Weak Existence and Uniqueness}
\begin{theorem}
	\label{thm eu w drift}
	Let assumptions \ref{ass drift form}, \ref{ass init l2}, and \ref{ass drift eu} hold. Then, 
	there exists a probability space as in 
	 Definition \ref{def:existence} with $X_0\sim \mu_0$ and a
	 pair 
	  $(X,\eta)$ solving (on this probability space) the drifted rearranged stochastic heat equation \eqref{def:existence} {on $[0,\infty)$}.
	  The law of this pair on ${\mathcal C}([0,\infty),U^2({\mathbb S})) \times 
	  {\mathcal C}([0,\infty),H^{-2}_{\rm sym}({\mathbb S}))$ is unique. 
	  
Moreover, for any $T>0$, 
 the laws (on ${\mathcal C}([0,T],U^2({\mathbb S})  {\times H_{\rm sym}^{-2}({\mathbb S})} )$) of the solutions 
to 
the undrifted and 
drifted versions of 
\eqref{def RSHE 1}, with 
the same initial condition (satisfying Assumption \ref{ass init l2}), are equivalent. \color{black} 
\end{theorem}
 
The proof is divided in several steps. It relies on several useful notations. For a fixed $T>0$, we denote by 
$\Omega_{T}^0 := {\mathcal C}([0,T], U^2({\mathbb S}) \times H^{-2}_{\rm sym}({\mathbb S}))$ the canonical space for carrying a solution up to time $T$ and we equip it with the Borel $\sigma$-field. 
When $T=\infty$, we just write 
$\Omega^0$. 
The canonical filtration is denoted ${\mathbb F}^0 = ({\mathcal F}_t)_{0 \leq t \leq T}$ (with $T$ being possibly equal to $\infty$). 
Similarly, the canonical process is denoted 
$(X_t^0,\eta^0_t)_{0 \le t \le T}$. We then let: 
\begin{equation*}
\begin{split}
&\beta_t^m(X^0,\eta^0) := \lambda_m^{-1} \Bigl( \langle X_t^0,e_m \rangle - \langle X_0^0,e_m \rangle + 4 \pi^2 m^2 \int_0^t \langle X_s^0,e_m \rangle \d s - \langle \eta_t ,e_m \rangle \Bigr), \quad m \in {\mathbb N}_0,  
 \\
 &W_t(X^0,\eta^0) := \sum_{m \in \bN_0} \lambda_m \beta_t^m(X^0,\eta^0) e_m,
 \end{split} 
 \end{equation*} 
 and 
 \begin{equation*} 
 Z_{{t}}^\pm(X^0,\eta^0) := 
 \cE_{{t}}\left\{\pm \int_0^\cdot\sum_{m\in \bN_0}\lambda_m^{-1}\langle F(X_s^0),e_m\rangle \d \beta^m_s(X^0,\eta^0)   \right\},
 \end{equation*} 
 with $\pm$ being understood as either 
 $+$ or $-$.

\begin{lemma}
\label{lem:weak uniqueness}
Let assumptions \ref{ass drift form}, \ref{ass init l2} and \ref{ass drift eu} hold. Then,
there exists $T_0>0$ such that, for any probability ${\mathbb P}$ on the space ${\Omega}^0$ under which the canonical process 
is a solution to 
\eqref{def RSHE 1}
in the sense of 
	 Definition \ref{def:existence} and for any $t \geq s \geq 0$ with $t-s \leq T_0$,
\begin{equation*} 
\quad {\mathbb E}^0 \bigg[ \frac{Z^\pm_{t}(X^0,\eta^0)}{Z^\pm_{s}(X^0,\eta^0)} \, \vert \, {\mathcal F}_s \bigg] = 1,
\end{equation*} 
 with the two symbols $\pm$ being understood as either 
 $+$ or $-$ (and taking the same value). 
\end{lemma} 

Assume for a while that Lemma
\ref{lem:weak uniqueness}
holds true (the proof is deferred to the end of the section) 
and we shall now deduce 
Theorem \ref{thm eu w drift}. 

\begin{proof}[Proof of Theorem \ref{thm eu w drift}]
We call ${\mathbb P}^0$ the law  on $\Omega^0$ of the (unique) solution to the driftless equation. 
\vskip 4pt

\textit{First Step.} 
We first construct a weak solution on the canonical space $\Omega^0_T$ for a finite $T>0$. 
To do so, 
we let $n= \lceil T/T_0\rceil$ (with $T_0$ as in the statement of Lemma
\ref{lem:weak uniqueness}, applied to the driftless equation and under ${\mathbb P}^0$). 
We then write
\begin{equation*} 
Z_{T}^+(X^0,\eta^0)
=
\prod_{k=1}^{n} \frac{Z^+_{kT/n}(X^0,\eta^0)}{Z^+_{(k-1)T/n}(X^0,\eta^0)},
\end{equation*} 
with the obvious convention that $Z_0(X^0,\eta^0)=1$. 
By successive conditionings, it is straightforward to deduce from Lemma 
\ref{lem:weak uniqueness} that 
\begin{equation*}
{\mathbb E}^{0} 
\bigl[
Z^+_{T}(X^0,\eta^0) \, \vert \, {\mathcal F}_0 \bigr] = 1, 
\end{equation*} 
and then 
\begin{equation*}
{\mathbb E}^{0} 
\bigl[
Z^+_{T}(X^0,\eta^0)  \bigr] = 1.  
\end{equation*} 
This makes it possible to let 
\begin{equation*} 
{\mathbb P}_T := Z^+_T(X^0,\eta^0) \cdot {\mathbb P}^0.
\end{equation*}  
It is standard to check that that the canonical process satisfies items 1, 2 and 3 of Definition \ref{def:existence} since the defining characteristics of the reflection process remain unaltered by the Girsanov transformation.
 
\vskip 4pt

\textit{Second Step.} We now establish item 4 in 
Definition \ref{def:existence}. 
Of course we already have item 4 under ${\mathbb P}^0$. It says in particular that 
\begin{equation*} 
 \int_0^{T} \bigl\langle e^{\varepsilon  \Delta} X_s^0, \d \eta_s^0 \bigr\rangle  
\end{equation*} 
tends to $0$ in ${\mathbb P}^0$-probability as $\varepsilon$ tends to $0$. Since 
${\mathbb P}_T$ is absolutely continuous with respect to ${\mathbb P}^0$, the same holds under 
${\mathbb P}_T$. 
 
Generally speaking, item 4 then follows from the same proof as in 
 \cite[Propositions 4.11 and 4.16]{delarueHammersley2022rshe}. Here is a sketch of it. By expanding (in time) 
$\| e^{\varepsilon \Delta}X^0_t \|^2$ under ${\mathbb P}$, we obtain an identity similar to  
\eqref{eq:ito:edeltaX2}:
\begin{equation*} 
\begin{split} 
&\| e^{\varepsilon \Delta} X_t^0 \|_2^2
- 
\| e^{\varepsilon \Delta} X_0^0 \|_2^2
+ \int_0^t \| e^{\varepsilon \Delta} \nabla_x X_s^0 \|_2^2 \d s
\\
&= 2 \int_0^t \langle F(X_s^0), e^{2 \varepsilon \Delta} X_s^0 \rangle  \d s 
+  \| e^{\varepsilon \Delta} W_t \|_2^2 + 2 \int_0^t \langle e^{2\varepsilon \Delta} X_s^0, 
\d \eta_s^0 \rangle
+ 2 \int_0^t \langle e^{2 \varepsilon \Delta} X_s^0, \d W_s \rangle,
\end{split}
\end{equation*} 
for $t \in [0,T]$ (and with $W$ being a shorter notation for $W(X^0,\eta^0)$). Raising to the power $p$, for an arbitrary real $p >1$, it is possible to upper bound 
\begin{equation*} 
{\mathbb E} 
\biggl[ \biggl\vert  \int_0^{T} \langle e^{2\varepsilon \Delta} X_s^0, \d \eta_s^0 \rangle \biggr\vert^p \biggr]
\end{equation*} 
in terms of $\sup_{0 \leq t \leq T} {\mathbb E} [ \| X_t^0 \|_2^{2p}]$ and ${\mathbb E} [( \int_0^{T} 
 \| e^{\varepsilon \Delta} \nabla_x X_s^0 \|_2^2 \d s)^p]$. By
 \eqref{eq:ito:expansion:polynomial} 
 (choosing 
 $\alpha  >  {\max(1, c_F)}$), we can obtain a bound for the above term. 
Together with 
the fact that we already have convergence in ${\mathbb P}$-probability, we deduce by uniform integrability that 
item 4 is satisfied. \color{black}
\vskip 4pt

\textit{Third Step.}
Now that the canonical process satisfies Definition 
\ref{def:existence} 
under ${\mathbb P}_T$, it is standard to construct a probability ${\mathbb P}$ on $\Omega^0$ such that 
${\mathbb P}(A) = {\mathbb P}_T(A)$ for any $A \in {\mathcal F}_T^0$ (see for instance \cite[Corollary 5.2, Chapter 3]{MR1121940}). 
Under ${\mathbb P}$, the canonical process is a weak solution. 

It remains to prove uniqueness in law of such a solution. 
Assume that we are given two probability spaces $(\Omega^i,{\mathcal A}^i,{\mathbb F}^i,{\mathbb P}^i)$, $i=1,2$, 
equipped each with a solution to 
the drifted version of \eqref{def:existence}. 
We can easily transport these two solutions onto the canonical space and thus assume that 
${\mathbb P}^1$ and ${\mathbb P}^2$ are probabilities on $\Omega^0$. Proceeding as in 
the first step, we deduce from 
Lemma 
\ref{lem:weak uniqueness}
that, for any $T>0$,  
\begin{equation*}
{\mathbb E}^{i} 
\bigl[
Z^-_{T}(X^0,\eta^0)  \bigr] = 1, \quad i=1,2. 
\end{equation*} 
Repeating the second step, 
we obtain that, under the two probability measures, 
\begin{equation*}
Z^-_{T}(X^0,\eta^0) \cdot {\mathbb P}^i, \quad i =1,2,
\end{equation*} 
the canonical process $(X_t^0,\eta_t^0)_{0 \le t \le T}$ solves the undrifted
equation. By pathwise uniqueness for the latter, we deduce that 
the above two probability measures coincide (on ${\mathcal F}_T^0$). 
This shows that ${\mathbb P}^1$ and ${\mathbb P}^2$ are equal. 
\end{proof}

  \subsection{Proof of Auxiliary Lemma  \ref{lem:weak uniqueness}}
   
\begin{proof}[Proof of Lemma \ref{lem:weak uniqueness}]
 On $(\Omega,\cF,\bP)$, we are given a solution $(X^0_t,\eta^0_t)_{0 \leq t \leq T}$ to the driftless rearranged stochastic heat equation up to time $T$ in the sense of Definition \ref{def:existence} with $X_0^0\sim \mu_0$ as initial condition. 
For an integer $M \geq M_{\textrm{\rm col}}$ and two reals $t_1,t_2 \in [0,T]$, 
with $t_1 \le t_2$,
\begin{align*}
		&   \exp\left\{ \tfrac{1}{2} \int_{t_1}^{t_2}\sum_{m\geq 0}\lambda_m^{-2}\langle F(X^0_s),e_m\rangle^2\d s \right\} \\
		&=    \exp\left\{ \tfrac{1}{2} \int_{t_1}^{t_2}\left( \sum_{m\leq M}+\sum_{m>M}\right)\lambda_m^{-2}\langle F(X^0_s),e_m\rangle^2\d s \right\}  \\
		&\leq   \exp\left\{ \tfrac{1}{2} \int_{t_1}^{t_2} 
		\biggl( {\max_{m=1,\cdots,M} 
		\lambda_m^{-2}}
		\sum_{m\geq 0}\langle F(X^0_s),e_m\rangle^2 + 
		{c_{\textrm{\rm col}}^{-2} M^{2(\lambda-1)}}
		\sum_{m>M} m^2\langle F(X^0_s),e_m\rangle^2 \biggr) \d s \right\}  \\		
		&=   \exp\left\{   \tfrac{1}{2} {\max_{m=1,\cdots,M} \lambda_m^{-2}}
		\int_{t_1}^{t_2} \lVert  F(X^0_s)\rVert_2^2 \d s+ \tfrac{1}{2} {c_{\textrm{\rm col}}^{-2}} M^{2(\lambda-1)}\int_{t_1}^{t_2} \bigl\lVert \nabla_x \bigl( F(X^0_s) \bigr)   \bigr\rVert_{2}^2 \d s \right\}.
	\end{align*}
The  inequality above follows from the colouring condition \eqref{eq colouring}. Now, one applies the growth conditions \eqref{eq bd4der}, 
\begin{align}
		 &   \exp\left\{ \tfrac{1}{2}
		  {\max_{m=1,\cdots,M} \lambda_m^{-2}}
		 \int_{t_1}^{t_2} \lVert  F(X_s^0)\rVert_2^2 \d s+ \tfrac{1}{2}  {c_{\textrm{\rm col}}^{-2}} M^{2(\lambda-1)}\int_{t_1}^{t_2} 
		 \bigl\lVert \nabla_x \bigl( F(X_s^0) \bigr)   \bigr\rVert_{2}^2 \d s \right\}  
		 \nonumber
		 \\
		 &\leq    \exp\left\{ \tfrac{1}{2}
		 {\max_{m=1,\cdots,M} \lambda_m^{-2}}
		 \int_{t_1}^{t_2} c_F\left( 1+\lVert X_s^0 \rVert_2^2\right) \d s+\tfrac{1}{2}
		 {c_{\textrm{\rm col}}^{-2}} 
		 M^{2(\lambda-1)}\int_{t_1}^{t_2} C_F \left( 1+\lVert \nabla_x X_s^0 \rVert_2^2 \right) \d s \right\}  \nonumber
		 \\
		 &=   \exp\left\{ \tfrac{1}{2} \Bigl( c_F  {\max_{m=1,\cdots,M} \lambda_m^{-2}} +
		 C_F  {c_{\textrm{\rm col}}^{-2}}  M^{2(\lambda-1)}\Bigr) \bigl( t_2-t_1 \bigr) \right\} 	 \label{eq novikov2}
		 \\
		 &\hspace{15pt}  \times  \exp\left\{ \tfrac{1}{2}  c_F {\max_{m=1,\cdots,M} \lambda_m^{-2}} \int_{t_1}^{t_2} \lVert X_s^0 \rVert_2^2  \d s+\tfrac{1}{2}  C_F  {c_{\textrm{\rm col}}^{-2}}  M^{2(\lambda-1)}\int_{t_1}^{t_2} \lVert  \nabla_x X_s^0 \rVert_2^2 \d s \right\}.
	\nonumber
	\end{align}
We then let 
\begin{equation*} 
\delta := \max\Bigl( \frac{c_F  {\max_{m=1,\cdots,M} \lambda_m^{-2}}}2, \frac{C_F  {c_{\textrm{\rm col}}^{-2}}   M^{2(\lambda-1)}}2, 1 \Bigr) 
\end{equation*} 
and we invoke 
Proposition \ref{prop zeroMode} with $t_2-t_1 \leq T_0:=\ln(2)/({2\alpha})$
where $\alpha ={\max(1,c_F)} + 4 \delta$. This says that the right-hand side in \eqref{eq novikov2} has a finite conditional expectation when $t_2-t_1 \le T_0$. 
By Novikov's condition for Girsanov theorem, we complete the proof. 
\end{proof}

%% file: sections/expConv.tex

In addition to 
Assumptions  \ref{ass drift form} and \ref{ass drift eu}, we will use the following:   
 
\begin{assumption}\label{ass drift erg}
There exist constants $C_L$,$c_1$ and $c_2$ such that, for $u\in U^2({\mathbb S})$,  
	\begin{equation} \label{eq ass drift erg}
	\langle F(u),u\rangle \leq C_L+c_1\sum_{m\neq 0}\hat{u}^2_m-c_2\hat{u}_0^2 = C_L +c_1\lVert u-\bar{u}\rVert_2^2 -c_2\bar{u}^2,
	\end{equation} 
where $c_1< {1}/{c_P^2}$,
 $c_2>0$ ($c_P$ is the Poincaré constant for $\bS$,  {i.e., 
 $\int_{{\mathbb S}} \vert f(x) - \int_{\mathbb S} f(y) \d y \vert^2
 \d x \leq c_P^2 \int_{\mathbb S}
 \vert \nabla f(x) \vert^2 \d x$ for any smooth function 
 $f : {\mathbb S} \rightarrow {\mathbb R}$}) and 
 $\bar u := \hat{u}_0 =  \int_{\mathbb S} u(x) \d x$.
\end{assumption} 

\begin{assumption}[Bounded oscillation at bounded distance]
\label{ass drift bounded oscillation}
For any $R>0$, there exists a  constant $C_O(R)$ such that, for any two $u,v \in  U^2({\mathbb S})$ with $\| u - v \|_2 \leq R$,  
	\begin{equation} \label{eq ass drift oscillation}
	\| F(u) - F(v) \|_2 \leq C_O(R).
	\end{equation} 
\end{assumption}
  
  Assumption \ref{ass drift erg}  details a so-called \emph{one sided} condition that amounts to providing the square of the $L^2(\bS)$ norm of as a Lyapunov condition for the system with a negative prefactor. The drift is allowed a little more flexibility in the higher Fourier modes due to the dissipative nature of the Laplacian over the non-constant Fourier basis elements. 

Assumption \ref{ass drift bounded oscillation} typically holds true if $F$ is Lipschitz continuous outside a ball $B(0,R_0)$ of $U^2(\bS)$ of centre $0$ and of radius $R_0$ for a certain $R_0\geq 1$. 
Take indeed $u,v \in U^2(\bS)$ with {$\max(\|u \|_2, \| v \|_2) \geq R$} for some $R>2 R_0$. If $\| u - v \|_2 > R/2$, then $\| F(u) - F(v) \|_2 \leq  C R \leq 2 C \| u - v \|_2$ thanks to 
\eqref{eq bd4der:1} {(for a constant $C$ depending on $c_F$)}. If $\| u - v \|_2 \leq R/2$, then the segment $[u,v]$ is outside the ball $B(0,R_0)$ and $\| F(u) - F(v) \|_2 \leq C \| u - v \|_2$ thanks to the Lipschitz property of $F$ outside $B(0,R_0)$
{(with $C$ now depending on the Lipschitz constant of $F$)}.

We provide the following example: 
 
\begin{proposition} 
\label{prop:SGD:example}
Let $F : U^2({\mathbb S}) \rightarrow {L^2_{\rm sym}({\mathbb S})}$ {be a functional mapping constant random variables (on ${\mathbb S}$) onto constant random variables} such that, for three non-negative constants $C$, $c$ and $R$:
\begin{enumerate}[(i)]
\item The function $F$ has oscillations of C-linear growth, in the following sense: 
\begin{equation*} 
\forall u,v \in U^2({\mathbb S}), \quad \bigl\| F(u) - F(v) \bigr\|_2 \leq C \bigl( 1+ \| u - v \|_2 \bigr);
\end{equation*}  
\item 
For any $u \in U^2({\mathbb S})$,  
$\overline{F(u)}  
\bar{u}  \leq  - c \, 
\bar{u}^2$ if  $\vert  \bar{u}   \vert \geq R$, 
where 
$\overline{F(u)} := 
 \int_{\mathbb S} F(u)(x) \d x$. 
\end{enumerate}
Then, the function $F$ satisfies 
Assumptions 
\ref{ass drift erg}
and
\ref{ass drift bounded oscillation}.
\end{proposition}

\begin{remark}
\label{rem:examples}
Obviously, the oscillation is at most of linear growth
(item $(i)$ 
in Proposition 
\ref{prop:SGD:example})
 if $F$ is $C$-Lipschitz continuous.   
\end{remark}
 
\begin{proof}[Proof of Proposition 
\ref{prop:SGD:example}]
The only difficulty is to check 
Assumption 
\ref{ass drift erg}. The left-hand side in 
\eqref{eq ass drift erg} writes (in a generic manner)
$ \langle F(u),u\rangle =  \overline{F(u) u}$. 
We have 
\begin{equation*}
\begin{split}
\langle F(u),u\rangle = 
  \overline{F(u) u}
&= 
\overline{F(u)} \bar{u}
+  \overline{F(u) (u - \bar{u})}
 =  \overline{F(u)} \bar{u}
+ 
\overline{\bigl( F(u) - \overline{F(u)} \bigr) \bigl(u - \bar{u}\bigr)}. 
\end{split} 
\end{equation*} 
By condition $(i)$ in the statement {(which implies $\| F(u) - \overline{F(u)} \|_2
\leq
\| F(u) - F(\bar u) \|_2 \leq 
C (1+ \| u - \bar u \|_2)$, because $F(\bar u)$ is a constant random variable)}, the second term
{in the right-hand side}
 is less than $C (1+ \|  u-\bar u \|_2^2)$
{for a possibly new value of $C$}. 
And by condition $(ii)$, the first one is less than $- c \bar{u}^2$ when 
$\vert  \bar u \vert \geq R$. When $\vert  \bar u \vert \leq R$, condition $(i)$ says that 
(allowing $C$ to change {from one inequality to another and to depend on $F(0)$, which may be identified with a real})
\begin{equation*}
\begin{split} 
\bigl\vert \overline{F(u)}  
\bar{u}
\bigr\vert \leq R \cdot 
\bigl\vert  \overline{F(u)}  \bigr\vert
&\leq C R 
\bigl( 
1+
{\| u \|_2}
 \bigr) 
  \leq C R\bigl( 1+ 
 \|  u-\bar u \|_2^2
 + \vert \bar u \vert \bigr) 
 \leq C R\bigl( 1+ R +  \|  u-\bar u \|_2^2  \bigr).
  \end{split}
  \end{equation*} 
  Since  $\vert \bar u \vert \leq R$, we can add $R^2$ and subtract 
  $\bar u^2$. 
  We get 
  \eqref{eq ass drift erg}.
\end{proof}

\subsection{Main Statement}

Let $X^u$ denote the (weakly) unique solution to the drifted rearranged stochastic heat equation started from $u \in U^2({\mathbb S})$. Then one may define the semigroup $P_\cdot$ over $U^2(\bS)$, by letting $P_t \varphi (u):=\mathbb{E}[\varphi(X^{u}_t)]$ {for $t \geq 0$ and $\varphi : U^2({\mathbb S}) \rightarrow {\mathbb R}$ bounded and measurable}, and the transition kernels 
{$\cP_\cdot(u,\cdot)$, by letting $\cP_t(u,\cdot):=\mathscr{L}(X^u_t)$}, so that 
$P_t \varphi(u) = \langle \cP_t(u,\cdot),\varphi \rangle$, where $\langle \cdot , \cdot \rangle$ stands here for the duality bracket {(between functions and measures on 
$U^2({\mathbb S})$)}. This section establishes the following result:
\begin{theorem}
\label{thm expStab}
Let Assumptions \ref{ass drift form}, \ref{ass drift eu}, \ref{ass drift erg} and \ref{ass drift bounded oscillation} hold. The semigroup $P_\cdot$ has a unique invariant measure $\mu_{\textrm{\rm inv}}$ and there are positive constants $c$ and $\gamma$ such that 
\begin{equation}
\label{eq expStab}
{
\forall u \in U^2({\mathbb S}), \ 
\forall t \geq 0,}
\quad 
d_{\rm TV}
\bigl(  \cP_t(u,\cdot), \mu_{\textrm{\rm inv}}\bigr) \leq Ce^{-\gamma t}\bigl(1+{\| u \|_2^2}\bigr),
\end{equation}
{with 
$d_{\rm TV}$ denoting the total variation distance between probability measures on $U^2({\mathbb S})$.}
\end{theorem}

The proof relies on Harris' theorem, 
in a form taken from \cite{hairerMattinglyScheutzow2011harris}, 
see Theorem \ref{thm Harris} below. 
The key point is to show that for two solutions started within a ball, there exists a coupling of the solutions after some elapsed time having non-trivial probability of being equal. 
This is established in 
Subsection 
\ref{subse:3.3} below.

Another useful (and easier) result is 

\begin{lemma}
	\label{lem unif est pair}
	Let 
	${\mathcal V} :   U^2({\mathbb S}) \ni u \mapsto \| u \|_2^2$. Then, 
	there exist two positive constants 
	$c$ and $C$  only depending on the parameters in the aforementioned assumptions such that 	
	for any $t >0$ and $u \in U^2({\mathbb S})$, 
	\begin{equation*} 
	P_t {\mathcal V}(u) \leq e^{-ct} {\mathcal V}(u) +C. 
	\end{equation*}
	\end{lemma}

\begin{proof}   
Writing 
\eqref{ito:square} in integral form, applying conditional expectation and the Poincaré inequality, we get, for any $s \leq t_1 < t_2$, 
	\begin{equation*}
		\begin{split}
			&\E\left[ \lVert    X_{t_2}  \rVert^2_2 \,  | \, \cF_s  \right] 
			\\
			&=  \E\left[ \lVert    X_{t_1}  \rVert^2_2 \, | \, \cF_s  \right]   + 2 \int_{t_1}^{t_2}\E \left[ \langle  X_r,   F(X_r) \rangle \,  | \, \cF_s    \right] \d r	-2 \int_{t_1}^{t_2} 
			\E\left[ \lVert \nabla_x X_r  \rVert_2^2 \, |\, \cF_s   \right] \d r  + c_\lambda (t_2-t_1) \\ 
			&\leq  \E\left[ \lVert    X_{t_1}  \rVert^2_2 \, | \, \cF_s  \right]   + 2 \int_{t_1}^{t_2} \E \left[ \langle X_r,   F(X_r) \rangle \,  | \, \cF_s  \right] \d r 
				-\tfrac{2}{c_P^2} \int_{t_1}^{t_2} \E\left[ 
			\| X_r - \bar X_r \|_2^2 \, | \, \cF_s 
			 \right] \d r 
			 \\
			 &\hspace{15pt} + c_\lambda(t_2-t_1),   
		\end{split}
	\end{equation*}
	where we recall that $\bar X_r = \int_{\mathbb S} X_r(x) \d x$. 	
Applying Assumption \ref{ass drift erg},
	\begin{equation*}
		\begin{split}
			&\E\left[ \lVert X_{t_2}  \rVert^2_2 \,  | \, \cF_s  \right]  
			\\
			&\leq   \E\left[ \lVert X_{t_1}  \rVert^2_2  \,   | \,  \cF_s  \right]   +   2\int_{t_1}^{t_2} \E \left[\left(c_1 -\tfrac{1}{c_P^2}  \right) 	\| X_r - \bar X_r \|_2^2  -c_2 \bar{X}_r^2 \,   | \, \cF_s    \right] \d r  
			  + (2C_L+c_\lambda) (t_2-t_1) 
			\\
			&\leq   \E\left[ \lVert X_{t_1}  \rVert^2_2  \,   | \,  \cF_s  \right]   -     2\left( \left( \tfrac{1}{c_P^2}-c_1  \right)\wedge c_2 \right)\int_{t_1}^{t_2} 
			\E \left[  \lVert   X_r  \rVert_2^2     \,   | \,  \cF_s  \right] \d r 
			  + (2C_L+c_\lambda) (t_2-t_1).  
		\end{split}
	\end{equation*}
	Denote $c_\omega : = 2[({1}/{c_P^2}-c_1 )\wedge c_2 ]$. Then, by a standard decay estimate and the Markov property, 
	\begin{equation*}
		\begin{split}
			\E\left[ \lVert X_t  \rVert^2 \,   | \,  \cF_s  \right]  	\leq & \lVert    X_s  \rVert_ 2^2 e^{- c_\omega (t-s) }+ c_\omega^{-1}\left( {2C_L+c_\lambda}\right)(1-e^{-c_\omega (t-s)})    \\ 
			\leq & \lVert    X_s  \rVert_ 2^2 e^{- c_\omega (t-s) } + c_\omega^{-1}\left( {2C_L+c_\lambda}\right),   \\ 
		\end{split}
	\end{equation*}
	which completes the proof. 
\end{proof}

\color{black}

\subsection{Coupling for Processes Initialised Inside a Ball}
\label{subse:3.3} 
The next step is to establish couplings
with a non-trivial probability of cohesion
 for processes started from within balls having arbitrary radius:
 \begin{proposition}\label{prop couplingFromBall}
Fix $R>0$. There exist $c^*,T^*>0$, depending on $R,\, C_{\rm col}$ and the assumptions on $F$, such that 
for $\| u \|_2,\| v \|_2 \leq R$, %
 \begin{equation*} 
  \begin{split} 
d_{\rm TV} \bigl( \cP_{T^*}(u,\cdot) , \cP_{T^*}(v,\cdot) \bigr) 
 \leq 1 - c^*.
\end{split}  
 \end{equation*}
 \end{proposition}

The rest of this subsection is dedicated to the proof of 
Proposition 
\ref{prop couplingFromBall}. \color{black} 
To demonstrate the coupling, we consider introducing a linear drift to the rearranged stochastic heat equation:
\begin{equation}
\label{eq:RSHE:reverting:lambda} 
\d  X_t^{\mu} = - \mu X_t^{\mu} \d t + \Delta X_t^{\mu} \d t + \d W_t + \d \eta_t
\end{equation}
for a parameter $\mu >0$. The flow generated by \eqref{eq:RSHE:reverting:lambda}, started from $u\in U^2(\mathbb{S})$, is denoted 
$(X_t^{\mu,u},\eta_t^{\mu,u})_{t \geq 0}$. 
We refer to \eqref{eq:RSHE:reverting:lambda}
as the \textit{linearly} drifted 
RSHE (as opposed to the $F$-drifted equation
introduced in Definition 
\ref{def:existence}). 
Theorem	\ref{thm eu w drift}
applies to this setting. {Because this drift is Lipschitz continuous (with respect to 
$\| \cdot \|_2$), 
uniqueness (and thus existence as well, by a standard adaptation of Yamada--Watanabe theorem) holds in the strong sense. 
Moreover, 
there is no difficulty in proving the following 
variant of  Proposition 4.17 in 
\cite{delarueHammersley2022rshe}, which addresses the spatial regularity of the flow.}

\begin{lemma}
\label{lem:reg:OU:RSHE}
For any $u,v \in U^2({\mathbb S})$, 
with probability 1,
for all $t \geq 0$, 
\begin{equation*} 
\| X_t^{\mu,u} - X_t^{\mu,v} \|_2^2 \leq \exp \bigl( - 2\mu t \bigr) 
\| u -v \|_2^2, 
\quad \int_0^t 
\| \nabla_x (X_s^{\mu,u} - X_s^{\mu,v} )\|_2^2 \d s \leq 
\frac{1}{2 } \| u -v \|_2^2, 
\end{equation*}  
and
\begin{equation*}
\int_0^t 
\| X_s^{\mu,u} - X_s^{\mu,v} \|_2^2 \d s \leq 
\frac{1}{2\mu} \| u -v \|_2^2, 
\end{equation*}

\end{lemma} 

\begin{proof}
We just provide a sketch of the proof. Following 
the proof of Proposition 4.17 in 
\cite{delarueHammersley2022rshe}, we have, for all $s,t \geq 0$, 
\begin{equation*}
\begin{split}
&\| X_t^{\mu,u} - X_t^{\mu,v} \|_2^2+ 2 \int_s^t \bigl\| \nabla_x \bigl( X_r^{\mu,u} - X_r^{\mu,v} \bigr) 
\bigr\|_2^2 \d r + 2 \mu \int_s^t \|   X_r^{\mu,u} -   X_r^{\mu,v} \|_2^2 
\d r 
\\
&\leq \| X_s^{\mu,u} - X_s^{\mu,v} \|_2^2,
\end{split}
\end{equation*} 
from which the result easily follows. 
\end{proof}

This lemma, in combination with  Theorem 5.9 in \cite{delarueHammersley2022rshe}, yields the following: 
\begin{lemma}
\label{prop:5:5:rev}
There exists a constant $c$ such that, for any $\delta>0$, $T>\delta$ and $u,v \in U^2({\mathbb S})$,  
\begin{equation*} 
d_{\rm TV} \bigl( {\mathscr L}(X_T^{\mu,u}) , {\mathscr L}(X_T^{\mu,v}) \bigr) \leq 
c_\lambda \delta^{-\frac{1+\lambda}{2}}
\exp\bigl( -  \mu (T-\delta ) \bigr)
 \| u - v \|_2, 
 \end{equation*} 
where the left-hand side denotes the total variation distance between 
${\mathscr L}(X_T^{\mu,u})$ and 
${\mathscr L}(X_T^{\mu,v})$, probability measures on 
$U^2({\mathbb S})$. 
\end{lemma} 

\begin{proof} 
We introduce the useful notation 
\begin{equation*} 
P_t^\mu \varphi(u) := {\mathbb E} \bigl[ \varphi(X_t^{\mu,u}) \bigr], \quad u \in U^2({\mathbb S}), 
\end{equation*} 
for $\varphi \in B_b(L^2_{sym}({\mathbb S}))$. 
Then, for $T>\delta$,
\begin{equation*} 
P^\mu_{T} \varphi(u) - P^\mu_{T} \varphi(v)
= {\mathbb E} \bigl[ P_\delta^{\mu} \varphi  (X_{T-\delta}^{\mu,u}) 
\bigr] - 
{\mathbb E} \bigl[ P_\delta^{\mu} \varphi  (X_{T-\delta}^{\mu,v} ) \bigr].
\end{equation*} 
There is no difficulty in adapting Theorem 5.9 from \cite{delarueHammersley2022rshe} to this \emph{linearly drifted} setting, proving that 
the function $P_\delta^\mu \varphi$ is Lipschitz continuous, with a Lipschitz constant bounded by 
$c_\lambda \delta^{-(1+\lambda)/{2}} \| \varphi\|_\infty$. Importantly, the constant $c_\lambda$ is independent of $\mu \in (0,+\infty)$ because the regularity of the flow, as stated in 
Lemma \ref{lem:reg:OU:RSHE}, is independent of $\mu$. 
And then,
\begin{equation*} 
\begin{split} 
\bigl\vert P^\mu_{T} \varphi(u) - P^\mu_{T} \varphi(v) \bigr\vert 
&\leq c_{\lambda} \delta^{-\frac{1+\lambda}{2}} \| \varphi \|_\infty {\mathbb E} \bigl[ \| X_{ {T-\delta}}^{\mu,u} - X_{ {T-\delta}}^{\mu,v} \|_2 \bigr] \\
 & \leq c_{\lambda} \delta^{-\frac{1+\lambda}{2}} \exp \bigl( - \mu {(T-\delta)} \bigr) \|\varphi \|_\infty  \| u - v \|_2,
\end{split}
\end{equation*}
with the last line following from 
Lemma 
\ref{lem:reg:OU:RSHE}.
\end{proof}

We return to the study of the RSHE drifted by $F$. For any fixed $T>0$, we can rewrite the equation solved by (the \emph{linearly drifted})
$X^{\mu,u}$
as
\begin{equation*}
\begin{split} 
\d X_t^{\mu,u} &= \Bigl( F(X_t^{\mu,u}) {\mathbf 1}_{\{ t \le T\}}
- \mu X_t^{\mu,u} {\mathbf 1}_{\{ t >T\}} 
\Bigr)
\d t 
\\
&\hspace{15pt} + \Delta X_t^{\mu,u} \d t +  \Bigl[\d W_t - 
\bigl(  F(X_t^{\mu,u}) + \mu X_t^{\mu,u}   \bigr) {\mathbf 1}_{\{ t \le T\}} \d t \Bigr] + \d \eta_t^{\mu,u}, \quad t \geq 0,
\end{split} 
\end{equation*}  
and define
\begin{equation} 
\label{eq:PT:F,x}
{\mathbb P}_T^{F,u} :=
{\mathcal E}_T \biggl\{ \int_0^{\cdot} \sum_{m \in {\mathbb N}_0} \lambda_m^{-1}
\bigl\langle F(X_t^{\mu,u}) + \mu X_t^{\mu,u}   , e_m \bigr\rangle \d \beta_s^{m} 
\biggr\} \cdot {\mathbb P} .
\end{equation} 
By an easy adaptation of {Theorem \ref{thm eu w drift}
and Lemma
\ref{lem:weak uniqueness}}, 
${\mathbb P}_T^{F,u}$ is a probability measure. Furthermore, under ${\mathbb P}_T^{F,u}$, 
$(X_t^{\mu,u})_{0 \le t \le T}$ solves the $F$-drifted RSHE  on the interval 
$[0,T]$ with $u$ as initial condition, and 
$(X_t^{\mu,u})_{t \geq T}$ solves the linearly drifted RSHE on $[T,+\infty)$ with $X_T^{\mu,u}$ as initial condition at time $T$.
We claim 
\begin{lemma} 
\label{lem:dTV:by:entropy}
Let $\delta$ be a  time duration in $(0,1)$
and $u\in U^2(\mathbb{S})$ be such that 
$\| u \|_2 \leq R$. 
{With $\lambda$ as in \eqref{eq colouring}}, there exists a constant $C_{F,\lambda,R}$ (independent of $\mu$, $\delta$ and $T$) such that
\begin{equation*} 
d_{\rm TV} \bigl(  {\mathbb P}^{F,u}_T  ,
 {\mathbb P}_{T+\delta}^{F,u}
 \bigr) \leq  C_{F,\lambda,R} \, 
 {(1+\mu)} \, \delta^{\frac{1-\lambda}{2}}.
 \end{equation*} 
\end{lemma}
\begin{proof} 
The proof relies on the relative entropy 
${\mathbb E}^{ {\mathbb P}^{F,u}_{T}} [ - \ln ( { \d {\mathbb P}^{F,u}_{T+\delta}}/{\d {\mathbb P}^{F,u}_{T}} ) 
].$
Here, 
\begin{equation*} 
\begin{split}
\ln \Bigl( \frac{ \d {\mathbb P}^{F,u}_{T+\delta}}{\d {\mathbb P}^{F,u}_{T}} \Bigr)
&=
\sum_{m \in {\mathbb N}_0} \int_{T}^{T+\delta} 
  \lambda_m^{-1}
\bigl\langle  F(X_t^{\mu,u}) + \mu X_t^{\mu,u} , e_m \bigr\rangle \d \beta_t^{m} 
\\
&\hspace{15pt} - 
\frac12 
\sum_{m \in {\mathbb N}_0} \int_{T}^{T+\delta} 
  \lambda_m^{-2}
\bigl\vert \bigl\langle  F(X_t^{\mu,u}) + \mu X_t^{\mu,u} , e_m \bigr\rangle
\bigr\vert^2
 \d t.
 \end{split} 
\end{equation*} 
And then, 
\begin{equation*} 
{\mathbb E}^{ {\mathbb P}^{F,u}_{T}} \Bigl[ - \ln \Bigl( \frac{ \d {\mathbb P}^{F,u}_{T+\delta}}{\d {\mathbb P}^{F,u}_{T}} \Bigr) 
\Bigr]
= 
\frac12 
{\mathbb E}^{ {\mathbb P}^{F,u}_{T}} \biggl[
\sum_{m \in {\mathbb N}_0} \int_{T}^{T+\delta} 
  \lambda_m^{-2}
\bigl\vert \bigl\langle  F(X_t^{\mu,u}) + \mu X_t^{\mu,u} , e_m \bigr\rangle
\bigr\vert^2
 \d t
\biggr]. 
\end{equation*} 
To estimate the above, we proceed as follows.
\begin{align} 
&\sum_{m \in {\mathbb N}_0} \int_{T}^{T+\delta} 
  \lambda_m^{-2}
\bigl\vert \bigl\langle  F(X_t^{\mu,u}) + \mu X_t^{\mu,u} , e_m \bigr\rangle
\bigr\vert^2
 \d t \nonumber
\\
&=  \int_{T}^{T+\delta}  \biggl( \sum_{{m = 0}}+\sum_{{m \in {\mathbb N}}}\biggr)
  \lambda_m^{-2}
\bigl\vert \bigl\langle  F(X_t^{\mu,u}) + \mu X_t^{\mu,u} , e_m \bigr\rangle
\bigr\vert^2
 \d  t
 \label{eq:mu:F:HSPLIT}
\\
&\leq     \int_{T}^{T+\delta} \biggl(  { \lambda_0^{-2}}  \lVert F(X_t^{\mu,u}) + \mu X_t^{\mu,u} \rVert_2^2  
	+	  \sum_{{m \in {\mathbb N}}} \lambda_m^{-2}
\bigl\vert \bigl\langle  F(X_t^{\mu,u}) + \mu X_t^{\mu,u} , e_m \bigr\rangle
\bigr\vert^2  
\biggr)
 \d t.
 \nonumber		
\end{align}  
Next, by H\"older's inequality and 
\eqref{eq colouring}, 
\begin{equation} 
\label{eq:mu:F:Holder}
\begin{split} 
&{\mathbb E}^{ {\mathbb P}^{F,u}_{T}}
\biggl[ \int_{T}^{T+\delta} 
\biggl( 
\sum_{ {m \in {\mathbb N}}} 
  \lambda_m^{-2}
\bigl\vert \bigl\langle  F(X_t^{\mu,u}) + \mu X_t^{\mu,u} , e_m \bigr\rangle
\bigr\vert^2
\biggr) 
 \d t
  \, \Bigr\vert \, {\mathcal F}_T \biggr]
\\
& \leq C_{\rm col} {\mathbb E}^{ {\mathbb P}^{F,u}_{T}} \biggr[
\sum_{m \in {\mathbb N}} \int_{T}^{T+\delta} 
   m^{2 \lambda}
\bigl\vert \bigl\langle  F(X_t^{\mu,u}) + \mu X_t^{\mu,u} , e_m \bigr\rangle
\bigr\vert^2
 \d t  \, \Bigr\vert \, {\mathcal F}_T \biggr]
 \\
 &\leq C_{\rm col}\biggl( {\mathbb E}^{ {\mathbb P}^{F,u}_{T}} \biggr[ \sum_{m \in {\mathbb N}}  \int_{T}^{T+\delta} 
  m^2
\bigl\vert \bigl\langle  F(X_t^{\mu,u}) + \mu  X_t^{\mu,u} , e_m \bigr\rangle
\bigr\vert^2
 \d t  \, \Bigr\vert \, {\mathcal F}_T \biggr] \biggr)^\lambda 
 \\
&\hspace{15pt} \times \biggl( {\mathbb E}^{ {\mathbb P}^{F,u}_{T}} \biggr[ \sum_{m \in {\mathbb N}}   \int_{T}^{T+\delta} 
\bigl\vert \bigl\langle  F(X_t^{\mu,u}) + \mu  X_t^{\mu,u} , e_m \bigr\rangle
\bigr\vert^2
 \d t  \, \Bigr\vert \, {\mathcal F}_T \biggr] \biggr)^{1-\lambda}
 \\
 & \leq C_{\rm col} \biggl( {\mathbb E}^{ {\mathbb P}^{F,u}_{T}}
 \biggr[
 \int_T^{T+\delta} \bigl\| \nabla_x \bigl( F(X_t^{\mu,u}) + \mu X_t^{\mu,u} 
 \bigr) \bigr\|_2^2 \d t
  \, \Bigr\vert \, {\mathcal F}_T \biggr]
 \biggr)^\lambda
 \\
&\hspace{15pt} \times \biggl( {\mathbb E}^{ {\mathbb P}^{F,u}_{T}} \biggr[ \int_{T}^{T+\delta} 
\bigl\| F(X_t^{\mu,u}) + \mu  X_t^{\mu,u} 
\bigr\|_2^2
 \d t  \, \Bigr\vert \, {\mathcal F}_T \biggr] \biggr)^{1-\lambda}.
\end{split} 
\end{equation}  
And then, 
using the fact that, under ${\mathbb P}^{F,u}_T$, the process $(X_s^{\mu,u})_{T \leq s \leq T+\delta}$ 
solves the linearly drifted RSHE, with $X_T^{\mu,u}$ as initial condition,  
and {invoking} Proposition \ref{prop ito fsquare}, 
\begin{equation*} 
\sup_{T \le t \le T +\delta} 
 {\mathbb E}^{{\mathbb P}_T^{F,u}} \bigl[ 
 \| X_t^{\mu,u} \|_2^2 \, \bigr\vert \, {\mathcal F}_T
 \bigr]
+
{\mathbb E}^{{\mathbb P}_T^{F,u}} \biggl[ \int_T^{T+\delta} \| \nabla_x X_t^{\mu,u} \|_2^2 \d 
 t \, \Bigr\vert \, {\mathcal F}_T \biggr] \leq 2 \Bigl( \| X_T^{\mu,u} \|^2_2 + c_\lambda \delta \Bigr).
 \end{equation*} 
Therefore, from Assumption \ref{ass drift eu}, 
\eqref{eq:mu:F:HSPLIT}
and
\eqref{eq:mu:F:Holder}, 
\begin{equation*} 
\begin{split}
{\mathbb E}^{{\mathbb P}_T^{F,u}} \biggl[ 
\sum_{m \in {\mathbb N}_0} \int_{T}^{T+\delta} 
  \lambda_m^{-2}
\bigl\vert \bigl\langle  F(X_t^{\mu,u}) + \mu X_t^{\mu,u} , e_m \bigr\rangle
\bigr\vert^2
 \d t \, \Bigr\vert \, {\mathcal F}_T \biggr]
 &\leq C_{F,\lambda}\,  {(1+\mu)^2}  \, \delta^{1-\lambda} \, \Bigl( 1+ \| X_T^{\mu,u} \|^2_2 \Bigr). 
 \end{split}
 \end{equation*} 
We deduce that 
\begin{equation*}
\begin{split} 
&{\mathbb E}^{ {\mathbb P}^{F,u}_{T}} \Bigl[ - \ln \Bigl( \frac{ \d {\mathbb P}^{F,u}_{T+\delta}}{\d {\mathbb P}^{F,u}_{T}} \Bigr) 
\Bigr]
\leq C_{F,\lambda}\,  {(1+\mu)^2} \, \delta^{1-\lambda} \,  \bigl( 1 +  {\mathbb E}^{{\mathbb P}^{F,u}_{T}}  \bigl[ \vert X_T^{\mu,u} \vert^2 \bigr] \bigr), 
 \end{split}
\end{equation*} 
where we observe that the moment appearing in the right-hand side is nothing but the second-order moment 
of the $F$-drifted  
RSHE starting from $u$, which can be bounded independently of $T$ and $u$ (as long as $\| u \|_2 \leq R$) {courtesy of Lemma \ref{lem unif est pair}}. In other words, the rightmost factor in the above may be bounded by a constant $C_R$. {By Pinsker's inequality 
\cite[Lemma 2.5, p.88]{Tsybakov}, 
 the result easily follows.}
\end{proof}

We now consider another initial condition $v \in U^2({\mathbb S})$. {Recalling that
$(X_t^{\mu,v},\eta^{\mu,v})_{t \geq 0}$ solves 
\eqref{eq:RSHE:reverting:lambda} 
with $v$ as initial condition, 
we claim} 
\begin{lemma}
\label{lem:DKL:1}
There exists a constant $C_{F,\lambda,R}$ $($independent of $\delta$, $T$, and $\mu)$ such that, for $\| u \|_2,\| v \|_2 \leq R$, 
 \begin{equation*} 
d_{\rm TV} \bigl( {\mathscr L}_{{\mathbb P}_{T+\delta}^{F,u}}(X_{T+\delta}^{\mu,u}) , {\mathscr L}_{{\mathbb P}_{T+\delta}^{F,u}}(X_{T+\delta}^{\mu,v}) \bigr) \leq  C_{F,\lambda,R} \left(\delta^{-\frac{1+\lambda}{2}}\exp( - \mu  {T}  ) +
 {(1+
 \mu)} \delta^{\frac{1-\lambda}{2}} \right),   
 \end{equation*}
 {where ${\mathscr L}_{{\mathbb P}_{T+\delta}^{F,u}}(\cdot)$
denotes the law of a generic random variable under 
${\mathbb P}_{T+\delta}^{F,u}$.}
\end{lemma}

\begin{proof}
 {For a test function $\varphi  \in B_b(L^2_{sym}({\mathbb S}))$} and a real $\delta >0$, 
we compute
\begin{equation*}
\begin{split} 
&{\mathbb E}^{{\mathbb P}_T^{F,u}} \bigl[ \varphi\bigl(X_{T+\delta}^{\mu,v}\bigr) \bigr]
- 
{\mathbb E}^{{\mathbb P}_T^{F,u}} \bigl[ \varphi\bigl(X_{T+\delta}^{\mu,u}\bigr) \bigr]
 ={\mathbb E}^{{\mathbb P}_T^{F,u}} \Bigl[  {\mathbb E}^{{\mathbb P}_T^{F,u}}
\bigl[ \varphi\bigl(X_{T+\delta}^{\mu,v}\bigr)  
- 
 \varphi\bigl(X_{T+\delta}^{\mu,u}\bigr) \vert {\mathcal F}_T \bigr] \Bigr].
\end{split} 
\end{equation*} 
Since $\d {\mathbb P}_T^{F,u}/\d {\mathbb P}$ is ${\mathcal F}_T$-measurable, 
we have  
\begin{equation*}
  {\mathbb E}^{{\mathbb P}_T^{F,u}}
\bigl[ \varphi\bigl(X_{T+\delta}^{\mu,v}\bigr)  
- 
 \varphi\bigl(X_{T+\delta}^{\mu,u}\bigr) \vert {\mathcal F}_T \bigr]
 = P_\delta^\mu \varphi \bigl( X_T^{\mu,v} \bigr) 
 - P_\delta^\mu \varphi\bigl( X_T^{\mu,u} \bigr),
 \end{equation*}
 where 
 $(P_t^{\mu})_{t \geq 0}$ denotes the semi-group generated by 
 the linearly drifted RSHE, see 
the proof of 
Lemma \ref{prop:5:5:rev}. {Because the estimate on the difference 
$\| X_T^{\mu,v} - X_T^{\mu,u} \|_2$
used in the proof of 
Lemma
 \ref{prop:5:5:rev}
holds almost surely, see
Lemma 
\ref{lem:reg:OU:RSHE}, the result of Lemma
 \ref{prop:5:5:rev} remains true under 
${\mathbb P}_T^{F,u}$}, namely
 \begin{equation*} 
d_{\rm TV} \bigl( {\mathscr L}_{{\mathbb P}_T^{F,u}}(X_{T+\delta}^{\mu,v}) , {\mathscr L}_{{\mathbb P}_T^{F,u}}(X_{T+\delta}^{\mu,u}) \bigr) \leq 
c_\lambda \, \delta^{-\frac{1+\lambda}2} 
\exp ( -  \mu  {T}  )
 \| u - v \|_2. 
 \end{equation*}  
Then, by Lemma 
\ref{lem:dTV:by:entropy}, 
 \begin{equation*} 
d_{\rm TV} \bigl( {\mathscr L}_{{\mathbb P}_{T+\delta}^{F,u}}(X_{T+\delta}^{\mu,u}) , {\mathscr L}_{{\mathbb P}_{T+\delta}^{F,u}}(X_{T+\delta}^{\mu,v}) \bigr) \leq 
c_\lambda \, \delta^{-\frac{1+\lambda}2}  
\exp ( - \mu  {T}  )
 \| u - v \|_2 + 2 \, C_{F,\lambda,R} \,  {(1+\mu)} \, \delta^{\frac{1-\lambda}{2}}.  
 \end{equation*} 
 This completes the proof, recalling that $u,v$ have $L^2$ norms bounded by $R$. 
 \end{proof}

 \begin{lemma}
\label{lem:DKL:2}
There exist a constant $C_{F,\lambda,R}$, independent of $\mu$, and a constant $C_\mu$, only depending on $\mu$, such that, for $\| u \|_2,\| v \|_2 \leq R$, 
\begin{equation*} 
d_{\rm TV} \bigl( {\mathscr L}_{{\mathbb P}_{T+\delta}^{F,u}}(X_{T+\delta}^{\mu,v}) , {\mathscr L}_{{\mathbb P}_{T+\delta}^{F,v}}(X_{T+\delta}^{\mu,v}) \bigr) \leq  
 1 - {\tfrac{1}{2}} \exp \Bigl( - C_{F,\lambda,R} \bigl( (T+\delta) + C_\mu \bigr)
\Bigr).
\end{equation*} 
\end{lemma} 
\begin{proof}
 We compute
 the relative entropy of 
 ${\mathbb P}_{T+\delta}^{F,v}$
 with respect to 
 ${\mathbb P}_{T+\delta}^{F,u}$. Using 
 the expression of each of the two densities, see 
  \eqref{eq:PT:F,x}, we obtain 
 \begin{equation}
\label{eq:dTB:last:step:mu:-1}
 \begin{split}
& {\mathbb E}^{{\mathbb P}_{T+\delta}^{F,u}}
 \Bigl[ - 
 \ln \Bigl( \frac{\d {\mathbb P}_{T+\delta}^{F,v}}{\d {\mathbb P}_{T+\delta}^{F,u}}
 \Bigr)\Bigr]
 \\
 &\hspace{5pt} =
 \frac12
 {\mathbb E}^{{\mathbb P}_{T+\delta}^{F,u}}
 \biggl[ 
 \int_0^{T+\delta}  \sum_{m \in {\mathbb N}_0} \lambda_m^{-2}
\bigl\langle  F(X_t^{\mu,u}) - F(X_t^{\mu,v}) + \mu (X_t^{\mu,u} - X_t^{\mu,v}), e_m \bigr\rangle^2
\d t
\biggr].
\end{split} 
\end{equation} 
Proceeding as in \eqref{eq:mu:F:HSPLIT} and \eqref{eq:mu:F:Holder}, we have
\begin{equation} 
\label{eq:dTB:last:step:mu:0}
\begin{split}
& {\mathbb E}^{{\mathbb P}_{T+\delta}^{F,u}} \int_0^{T+\delta}  \sum_{m \in {\mathbb N}_0} \lambda_m^{-2}
\bigl\langle  F(X_t^{\mu,u}) - F(X_t^{\mu,v}) + \mu (X_t^{\mu,u} - X_t^{\mu,v}), e_m \bigr\rangle^2
 \d t
\\
&\leq 
{\lambda_0^{-2}}
  {\mathbb E}^{{\mathbb P}_{T+\delta}^{F,u}}
 \int_0^{T+\delta} 
\bigl\| 
F(X_t^{\mu,u}) - F(X_t^{\mu,v}) +  \mu (X_t^{\mu,u} - X_t^{\mu,v})
\bigr\|_2^2 \d t 
\\
&\hspace{15pt}   + 
 C_{\rm col}    \biggl( 
  {\mathbb E}^{{\mathbb P}_{T+\delta}^{F,u}}
  \int_0^{T+\delta} 
\bigl\| \nabla_x 
\bigl( 
F(X_t^{\mu,u}) - F(X_t^{\mu,v}) +  \mu (X_t^{\mu,u} - X_t^{\mu,v}) \bigr) 
\bigr\|_2^2 \d t  \biggr)^{\lambda} 
\\
&\hspace{30pt}
\times  
 \biggl(
  {\mathbb E}^{{\mathbb P}_{T+\delta}^{F,u}}
   \int_0^{T+\delta} 
\bigl\| 
F(X_t^{\mu,u}) - F(X_t^{\mu,v}) +\mu (X_t^{\mu,u} - X_t^{\mu,v}) 
\bigr\|_2^2 \d t
\biggr)^{1-\lambda}.
\end{split}
\end{equation}
 By Lemma 
 \ref{lem:reg:OU:RSHE}, we obtain 
 \begin{equation}
 \label{eq:dTB:last:step:mu:1}
\begin{split}
&\int_0^{T+\delta}   \bigl\|   \nabla_x   \bigl(X_t^{\mu,u} - X_t^{\mu,v} \bigr) 
\bigr\|_2^2 
\d t \leq  {2} R^2, 
\quad \mu  \int_0^{T+\delta}   \|   X_t^{\mu,u} - X_t^{\mu,v}  
\|_2^2 
\d t
 \leq  {2 R^2}. 
\end{split}
\end{equation}
By Assumption 
\ref{ass drift bounded oscillation}, 
\begin{equation} 
\label{eq:dTB:last:step:mu:2}
\int_0^{T+\delta}  \| F(X_t^{\mu,u}) - F(X_t^{\mu,v})\|_2^2 \d t
\leq (T+\delta) \bigl( C_O(2R) \bigr)^2.  
\end{equation} 
Next,
by Assumption \ref{ass drift eu} and \eqref{eq:dTB:last:step:mu:1}, 
\begin{equation} 
\label{eq:dTB:last:step:mu:3}
\begin{split}
&\int_0^{T+\delta}  \bigl\| \nabla_x \bigl(  F(X_t^{\mu,u}) - F(X_t^{\mu,v}) \bigr) \bigr\|_2^2 \d t
\\
&\leq  {2} C_F  
\int_0^{T+\delta} \Bigl(  {2} +  \bigl\| \nabla_x X_t^{\mu,u} \bigr\|_2^2 + 
\bigl\| \nabla_x X_t^{\mu,v} \bigr\|_2^2
\Bigr)  \d t
\\
&\leq  {2} C_F\left(
 {2 (T+\delta)}
+
 {4}
 R^2 +  {3} \int_0^{T+\delta}   \bigl\| \nabla_x X_t^{\mu,u} \bigr\|_2^2  \d t\right).
\end{split}
\end{equation}  
Following the first step 
in the proof of 
Lemma 
\ref{lem unif est pair}
(which is based on \eqref{ito:square}), 
\begin{align}
\label{eq:new:proof:c1cP}
	&
2 \bigl( 1 - c_1 c_P^2 \bigr) 
  {\mathbb E}^{{\mathbb P}_{T+\delta}^{F,u}}
   \int_0^{T+\delta} 
\bigl\| 
\nabla_x X_t^{\mu,u}  
\bigr\|_2^2 \d t
	\leq \| u \|_2^2 + 
	\bigl( 2 C_L +  c_\lambda 
	\bigr) 
\bigl( T + \delta \bigr). 
\end{align} 
Collecting 
\eqref{eq:dTB:last:step:mu:-1},
\eqref{eq:dTB:last:step:mu:0},
\eqref{eq:dTB:last:step:mu:1},
\eqref{eq:dTB:last:step:mu:2}, 
\eqref{eq:dTB:last:step:mu:3}
and
\eqref{eq:new:proof:c1cP},
\begin{equation*} 
\begin{split}
 {\mathbb E}^{{\mathbb P}_{T+\delta}^{F,u}}
 \Bigl[ - 
 \ln \Bigl( \frac{\d {\mathbb P}_{T+\delta}^{F,v}}{\d {\mathbb P}_{T+\delta}^{F,u}}
 \Bigr)\Bigr]
& \leq C_{F,\lambda}  \Bigl[ \Bigl( 1+ \bigl(C_O(2R)\bigr)^2+R^2 \Bigr)(T+\delta) + C_R(1+\mu^2) \Bigr]  
\\
&\leq 
C_{F,\lambda,R} \bigl( (T+\delta) + C_\mu \bigr).
\end{split}
\end{equation*} 
{The result follows 
from
Tsybakov inequality
(also referred to as the Bretagnolle-Huber inequality), see  
\cite[(2.25), p.89]{Tsybakov}.}
 \end{proof} 
 
 We now complete the proof
 of the coupling inequality:
\begin{proof}[Proof of  Proposition \ref{prop couplingFromBall}.]
 
Fix $R>0$. We deduce from Lemmas \ref{lem:DKL:1}
 and
\ref{lem:DKL:2} that
  \begin{equation*} 
  \begin{split} 
&d_{\rm TV} \bigl( {\mathscr L}_{{\mathbb P}_{T+\delta}^{F,u}}(X_{T+\delta}^{\mu,u}) , {\mathscr L}_{{\mathbb P}_{T+\delta}^{F,v}}(X_{T+\delta}^{\mu,v}) \bigr) 
\\
&\leq 
C_{F,\lambda,R} \left(\delta^{-\frac{1+\lambda}{2}}\exp\bigl( - \mu   {T} \bigr) +  {(1+\mu)} \delta^{\frac{1-\lambda}{2}} \right)   + 1 -  {\frac12} \exp \bigl( -C_{F,\lambda,R} ( (T+\delta) + C_\mu )
\bigr).
\end{split}  
 \end{equation*}
Choosing $\delta = \exp(-\mu T/2)$ (with $T$ large so that $\delta \le T/2$, which is doable --independently of $\mu$-- if $\mu\geq 1$ for instance)
and using the fact that $\lambda \in (\nicefrac{1}2,1)$, we get 
 \begin{equation*} 
  \begin{split} 
&d_{\rm TV} \bigl( {\mathscr L}_{{\mathbb P}_{T+\delta}^{F,u}}(X_{T+\delta}^{\mu,u}) , {\mathscr L}_{{\mathbb P}_{T+\delta}^{F,v}}(X_{T+\delta}^{\mu,v}) \bigr) 
\\
&\leq 1 +  C_{F,\lambda,R} \Bigl[  
\exp\Bigl(  - \mu  {\frac{T}2} \Bigr)
+  {(1+\mu)} \exp \Bigl( - \frac{(1-\lambda)\mu T}{4} \Bigr) \Bigr]   - \frac12 \exp \bigl( - C_{F,\lambda,R} ( (T+\delta) + C_\mu )
\bigr)
\\
&\leq 1 + \exp \bigl( - 2 C_{F,\lambda,R} T \bigr) \Bigl[ C_{F,\lambda,R} \exp \Bigl( \bigl(2 C_{F,\lambda,R} -  {\frac{\mu}2} \bigr) T \Bigr)
\\
&\hspace{15pt} +  {(1+\mu) \, C_{F,\lambda,R} } \exp \Bigl( \bigl(2 C_{F,\lambda,R}  -\frac{(1-\lambda)\mu }{4} \bigr) T \Bigr)   
 -  {\frac12} \exp \bigl( - C_{F,\lambda,R}C_{\mu}   \bigr) 
\Bigr]. 
\end{split}  
 \end{equation*}
 Choosing $(1-\lambda)\mu  > 8 C_{F,\lambda,R}$, we see that the limit, as $T$ tends to $+\infty$, 
 of the term inside the bracket is negative. 
 
 Recalling that 
 $ {\mathscr L}_{{\mathbb P}_{T+\delta}^{F,u}}(X_{T+\delta}^{\mu,u}) =  \cP_{T+\delta}(u,\cdot) $
 and 
 $ {\mathscr L}_{{\mathbb P}_{T+\delta}^{F,v}}(X_{T+\delta}^{\mu,v}) =  \cP_{T+\delta}(v,\cdot) $, 
 we complete the proof. 
 \end{proof} 
%

\subsection{Harris Theorem}
We now conclude via the Harris theorem that we do indeed have exponential convergence to stability. We reference this theorem as it appears in Hairer, Mattingley and Scheutzow \cite{hairerMattinglyScheutzow2011harris}.

\begin{theorem}[Harris]
\label{thm Harris} Let $(P_t)_{t \geq 0}$ be a Markov semigroup over a Polish space $\cX$ such that there exists a Lyapunov function ${\mathcal V}$, i.e. a function ${\mathcal V}:\cX \rightarrow [0,\infty]$ such that $P_t {\mathcal V}(x)<\infty$ for all $x\in\cX$ and $t\geq 0$ 
satisfying, for every $x \in {\mathcal X}$ and $t\geq 0$, the bound 
\begin{equation}
\label{eq lyapHMS}
P_t{\mathcal V}(x)\leq C_{\mathcal V}e^{-\gamma_{\mathcal V} t}{\mathcal V}(x)+ K_{\mathcal V}, 
\end{equation}
for some constants $C_{\mathcal V},\gamma_{\mathcal V},K_{\mathcal V}>0$.
Assume further that the level sets of ${\mathcal V}$, defined by $A_{\mathcal V}(C):=\{x\in \cX : {\mathcal V}(x)\leq C \}$, are small, i.e. for every $C$, there exists a time $ T>0$ and a constant $\varepsilon>0$ such that for every $x,y \in A_{\mathcal V}(C)$, 
\begin{equation}
\label{eq small}
d_{\rm TV}\bigl( \cP_T(x,\cdot), \cP_T(y,\cdot)\bigr) \leq  1 -\varepsilon,
\end{equation}
{with $(\cP_t)_{t \geq 0}$ denoting the transition kernels as in the statement of 
Theorem 
\ref{thm expStab}.}

Then, $(P_t)_{t \geq 0}$ has a unique invariant measure $\mu_{\textrm{\rm inv}}$ and, for any $t \geq 0$, 
$
\lVert \cP_{{t}}(x,\cdot)-\mu_{\textrm{\rm inv}} \rVert_{TV}\leq Ce^{-\gamma t}(1+{\mathcal V}(x))$ for some positive constants $C$ and $\gamma$.
\end{theorem}

\begin{proof}[Proof of Theorem \ref{thm expStab}] 
The square of the $L^2$ norm, $\lVert\cdot\rVert_2^2$, serves as a Lyapunov function for the semigroup $(P_t)_{t \geq 0}$. This can be observed from
Lemma 
\ref{lem unif est pair}.  Smallness of the level sets $A_{\lVert\cdot\rVert_2^2}(C)$ follows from 
Proposition 
\ref{prop couplingFromBall}. 
\color{black}
%
\end{proof}

%% file: sections/SGD.tex

\subsection{The Drift as the Derivative of a Potential}
\label{subse:derivative:Q}

We now address the key example of when $F$ derives from a potential, namely
\begin{equation*}
F(u)(x) = -\partial_{\mu} V\bigl(
{\rm Leb}_{\mathbb S} \circ u^{-1}\bigr)(u(x)), \quad x \in {\mathbb S},
\end{equation*}
where $u$ in the left-hand side is an element of 
$U^2({\mathbb S})$ and 
$V$ in the right-hand side is a mean-field function, i.e., $V : {\mathcal P}({\mathbb R}) \rightarrow 
{\mathbb R}$. We call $V$ a potential of $F$. In the first instance, $V$ is assumed to be continuously differentiable in Lions' manner, but later we will allow the derivative to be defined in a weaker sense, see Definition 
\ref{def:derivative:mathscr DV}.
 
The corresponding drifted form of the rearranged stochastic heat equation reads as
\begin{equation}
\label{SGD}
\d X_t(x) = - \partial_{\mu} V \bigl( {\mathscr L}(X_t) \bigr)\bigl( X_t(x) \bigr) \d t + 
\Delta_x X_t(x) \d t + \d W_t(x) + \d \eta_t(x), \quad x \in {\mathbb S}, \ t \geq 0, 
\end{equation}
where ${\mathscr L}(X_t)$ is the law of $X_t$, when regarded as a real-valued random variable 
on ${\mathbb S}$ (equipped with the Lebesgue measure), i.e. 
${\mathscr L}(X_t)= 
\textrm{\rm Leb}_{\mathbb S} \circ X_t^{-1}$.  
We then refer to the above dynamics as a stochastic gradient descent on the space of probability measures.
It coincides with 
\eqref{eq:MKV:3}.  
The following remarks are in order: 
\begin{enumerate}
\item It is an easy exercise to reformulate the conditions $(i)$ and $(ii)$ in the statement of 
Proposition \ref{prop:SGD:example}. We leave this to the reader. 
\item When $x \in {\mathbb S} \mapsto \partial_\mu V(\textrm{\rm Leb}_{\mathbb S} \circ u^{-1})(u(x))$ satisfies 
$(i)$ in the statement of Proposition 
\ref{prop:SGD:example}
and $V$
is invariant by translation, i.e. 
$V(\mu) = V(\mu \circ \tau_y^{-1})$ for any $y \in {\mathbb R}$, with 
$\tau_y$ denoting the translation (in ${\mathbb R}$) of vector $y$, then, for any $c>0$, 
the derivative of the (penalized) potential $V(\mu) + c \int_{\mathbb R} \vert y \vert^2 \d \mu(y)$ satisfies 
(up to a sign $-$)
$(i)$ and $(ii)$ in the statement of Proposition 
\ref{prop:SGD:example}. Indeed, invariance by translation implies
\begin{equation*} 
{\mathbb E} \bigl[ \partial_\mu V\bigl( {\mathscr L}(X)\bigr)(X) \bigr] =0,
\end{equation*} 
for any random variable $X$ constructed on any probability space (with ${\mathbb E}$ as expectation). 
For instance, this is the case if 
\begin{equation*} 
V(\mu) =   \int_{\mathbb R} \int_{\mathbb R} w(y-z) \d \mu(y) \d\mu(z),
\end{equation*} 
for $w : {\mathbb R} \rightarrow {\mathbb R}$ a  function
with bounded second order derivatives. Invariance by translation (of this example) is indeed obvious. One has 
\begin{equation*} 
\partial_\mu V(\mu)(y) = \int_{\mathbb R} \nabla w(y-z) d\mu(z) -
\int_{\mathbb R} \nabla w(z-y) d\mu(z),
\end{equation*} 
from which the  condition $(i)$
(in presence of an additional penalty  $c \int_{\mathbb R} \vert y \vert^2 d \mu(y)$) easily follows. 
\end{enumerate}

In what follows, we address a relaxed version of \eqref{SGD} in which 
the derivative of the potential $V$ is understood in a weaker sense. This extended notion of 
derivative relies on the following definition: 

\begin{definition}
\label{def:derivative:mathscr DV}
We say that a function $V : U^2({\mathbb S}) \rightarrow {\mathbb R}$ is continuously differentiable (on $U^2({\mathbb S})$) if there exists 
a continuous function 
${\mathscr D} V : U^2({\mathbb S}) \rightarrow L_{\rm sym}^2({\mathbb S})$ such that, for any two elements $u$ and $v$ in $U^2({\mathbb S})$,  
\begin{equation*} 
V(v) = V(u) + \int_0^1 
\bigl\langle 
{\mathscr D} V\bigl( t v + (1- t) u \bigr), v-u
\bigr\rangle_2   \d t. 
\end{equation*} 
\end{definition} 

Of course, it should be noticed that, if $V$ is Lions continuously differentiable, then it is 
continuously differentiable in the above sense, with 
\begin{equation*}
{\mathscr D} V(u)(x) = \partial_\mu V\bigl( \textrm{\rm Leb}_{\mathbb S} \circ u^{-1} \bigr)\bigl( u(x)\bigr), 
\quad \textrm{Leb}_{\mathbb S} \  \textrm{\rm a.e.} \ x \in {\mathbb S}. 
\end{equation*}
However, the converse is not true. Take for instance the function 
$V(u) = \| u-u_0\|_2^2$, for $u \in U^2({\mathbb S})$,
where $u_0$ is a fixed element of $U^2({\mathbb S})$. Then, there is no difficulty to see that 
\begin{equation*} 
{\mathscr D} V(u)(x) = 2 \bigl(u(x)-u_0(x)\bigr).
\end{equation*} 
However, the function may not be Lions continuously differentiable 
as it coincides in fact with the mean field function $\mu \mapsto {\mathcal W}_2(\mu,\mu_0)^2$ where 
$\mu_0$ denotes the law of $u_0$ (under the Lebesgue measure), see for instance \cite{refId0}. 

Intuitively, the difficulty with the Lions derivative lies in the \textit{lifting} operation. Indeed, 
the lift on the entire $L^2_{\textrm{\rm sym}}({\mathbb S})$-space
of a real-valued function 
$V$ originally defined 
on $U^2({\mathbb S})$
 is given by $L^2_{\textrm{\rm sym}}({\mathbb S}) \ni u \mapsto V(u^*)$,
but the rearrangement mapping $u \mapsto u^*$ may be singular. 
In contrast, when one restricts the notion of differentiability to $U^2({\mathbb S})$, one avoids 
all the issues related with the rearrangement.  
This is an important observation that allows for the consideration of more examples. 

Naturally, there is a price to pay for weakening the notion of derivative. In brief, Lions' representation
is no longer true in full generality, meaning that we can find instances of $V$ and of $u \in U^2({\mathbb S})$ such that the derivative 
${\mathscr D} V(u)$ is not $\sigma(u)$-measurable. A prototype for this is 
\begin{equation*}
V(u) = \int_{\mathbb S} g\bigl(x,u(x)\bigr) \d x,
\end{equation*}
for a smooth function $g : {\mathbb S} \times {\mathbb R} \ni (x,y) \mapsto g(x,y)$. 
Clearly, the derivative reads 
${\mathscr D} V(u)(x) = \partial_y g(x,u(x))$, for $x \in {\mathbb S}$, 
and may not be 
a function of $u$ (think of $u$ as being a constant function). 

\begin{remark}[Potential structure of the stochastic gradient Descent]
Returning to \eqref{SGD}, one may be interested in having an interpretation of the Laplacian as the derivative of a potential. In other words, it is natural 
 to wonder about the effective potential that drives the dynamics. 
As one might expect, we 
may indeed
 regard 
the dynamics in 
\eqref{SGD} as a stochastic gradient descent deriving from the potential
\begin{equation}
\label{eq:potential:L2} 
V(u) + \tfrac{1}2 \bigl\| \nabla u\bigr\|_2^2, \quad 
u \in U^2({\mathbb S}),
\end{equation}
whose G\^ateaux derivative (inside the cone $U^2({\mathbb S})$), at any smooth $u$ and along any $v-u$ for $v \in U^2({\mathbb S})$, indeed reads
(using \eqref{eq:Lions:derivative})
\begin{equation*}
\begin{split}
&\frac{\d}{\d \varepsilon} \Bigl[ V\bigl(  {u} + \varepsilon (v-u) \bigr) + \tfrac{1}{2} \bigl\|  {  \nabla (u+\varepsilon(v-u))} \bigr\|_2^2
\Bigr]
 {\bigg|_{\varepsilon=0} }= \bigl\langle  {\mathscr D} V(u),v-u 
\bigr\rangle_{2}
-   \langle \Delta u , v-u
   \rangle_{2}.
\end{split}
\end{equation*}
Despite the potential structure of the equation \eqref{SGD}, it is fair to say that we have no information about the precise shape of 
the invariant measure. In particular, we do not have an example
with an explicit 
Gibbs factorization. To the best of our understanding, this is a generic fact given our construction of the rearranged SHE 
because the non-drifted dynamics is indeed non-symmetric. To see this,
we claim that the reflection term in the rearranged SHE should be thought as 
a projection term in $L_{\rm sym}^2({\mathbb S})$. Whilst the proof of this latter assertion would go beyond the scope of this paper, we feel it sufficient to refer here to Brenier's work \cite{Brenier1} in the deterministic case for 
some intuition. Then, symmetry of the dynamics could be obtained 
if the noise itself were isotropic in $L_{\rm sym}^2({\mathbb S})$, which is 
for instance exactly what is done by Zambotti in \cite{zambotti2001intByParts} for the reflected SPDE of Nualart and Pardoux \cite{NualartPardoux}. 
In our case, the noise is anisotropic: this is a crucial assumption 
in the construction carried out  
\cite{delarueHammersley2022rshe} and, at this stage, we are unable to remove it. 

While anisotropy of the noise may at first glance look like a serious drawback of 
our model, we show in the next section that one can use a coloured noise 
that acts isotropically on a large (finite) collection of Fourier modes. 
Such a construction allows us to recover some of the metastability features of 
Gibbs measures at a low temperature. 
\end{remark}

%% file: sections/vanishVisc.tex

The purpose of this subsection is to provide
some metastability properties of the (stochastic) gradient descent when the noise has a small intensity. The results that are proven hold for a convenient scaling between the intensities of the noises applied to each Fourier mode 
of the solution and the intensity of the Laplace friction term.  

As we have almost no information on the structure of the invariant measure of the stochastic gradient descent, the metastability properties that are stated below are important. 
To summarise, they say that, under the aforementioned convenient scaling between the noises and the Laplace friction term, 
the time spent by the stochastic gradient descent in a local well of the potential $V$ 
grows exponentially fast in relation to $\varepsilon^{-2}$, where $\varepsilon$ 
is the parameter encoding the intensity of the noise. 

To clarify, here is the main equation we are dealing with in this section: 
\begin{equation}
\label{eq:SGD:small:visco}
\begin{split}
\d X_t^{\varepsilon}(x) &= - {\mathscr D} V \bigl(  X_t^{\varepsilon}  \bigr)(x) \d t 
 + 
\varepsilon^{2\alpha} \Delta_x X_t^{\varepsilon}(x) \d t + 
\varepsilon \d W_t^{\varepsilon}(x) + \d \eta_t^{\varepsilon}(x), \quad t \geq 0, \ x \in {\mathbb S}, 
\end{split}
\end{equation}
for a potential  {$V : U^2({\mathbb S}) \rightarrow {\mathbb R}$} that is continuously differentiable in the sense of the derivative ${\mathscr D}$ introduced in Definition 
\ref{def:derivative:mathscr DV}, for some intensity parameter $\varepsilon >0$ (driving the intensity of the noise)
and for some exponent $\alpha >0$ (driving the strength of the friction term), where 
$(W_t^{\varepsilon})_{t \geq 0}$ is the coloured noise having spectral decomposition
\begin{equation}
\label{eq:notations:visco:1}
W_t^{\varepsilon}(\cdot) 
: = \sum_{k \in  {\mathbb N_0}} \lambda_k^{\varepsilon} {\beta}_t^k {e_k(\cdot)}, 
\end{equation}
with 
\begin{equation}
\label{eq:notations:visco:2} 
\lambda_k^{\varepsilon} :=
\left\{ 
\begin{array}{ll}
\sqrt{\vartheta_1} &\quad \textrm{\rm if} \quad 
k \leq \varepsilon^{-\beta}
\\
\sqrt{\vartheta_2} &\quad \textrm{\rm if} 
\quad \varepsilon^{-\beta} < k \leq 
\psi \varepsilon^{-\beta}
\\
k^{-\lambda} &\quad \textrm{\rm if} 
\quad 
k > \psi \varepsilon^{-\beta}
\end{array}
\right.,
\end{equation}
for some positive parameters $\beta,\vartheta_1,\vartheta_2$ and $\psi$ (with $ \psi >1$) whose values will be fixed later on {(as before, $\lambda \in (1/2,1)$)}. 
The main statement of this section is: 
\begin{theorem}
\label{thm:meta}
Assume that $V$
is continuously differentiable in the sense of 
Definition 
\ref{def:derivative:mathscr DV}
and that ${\mathscr D} V$ satisfies assumptions \ref{ass drift eu}, \ref{ass drift erg} and \ref{ass drift bounded oscillation}. {Assume further that there exist an $X_0 \in H^1_{\rm sym}({\mathbb S})$, with $\mathscr D V(X_0)=0$,
and a real $a>0$, such that ${\mathscr D} V$ 
is Lipschitz continuous 
on
${\mathcal O} := U^2({\mathbb S}) \cap B_{L^2}(X_0,a)$ (with 
$B_{L^2}(X_0,a)$ being the $L^2({\mathbb S})$-ball of centre $X_0$ and radius $a$)
and satisfies, for any $X \in {\mathcal O}$,}
\begin{equation}
\label{eq:monotonicity:weak} 
\bigl\langle {\mathscr D} V(X) ,X-X_0\bigr\rangle \geq  \kappa \| X - X_0 \|_2^2. 
\end{equation} 
{Set $\alpha=\beta=2$ in 
\eqref{eq:notations:visco:2}, 
denote  $(X_t^{\varepsilon})_{t \geq 0}$ the solution to \eqref{eq:SGD:small:visco}
with $X_0$ as initial condition 
and let 
$\tau_{\varepsilon} := \inf \{ t \geq 0 : \| X_t^{\varepsilon}  - X_0\|_2 \geq a \}$.}
Then, for any $a_0>1$, such that $a \in [1/a_0,a_0]$, we can tune $\vartheta_1,\vartheta_2$ and $\psi$ in 
\eqref{eq:notations:visco:2} in such a manner that, for a constant $C>1$ depending on $\vartheta_1$, $\vartheta_2$, $\kappa$, $\psi$ and $a_0$ (but not on $a$),  
and for $\varepsilon \in (0,\varepsilon_0)$, 
with $\varepsilon_0$ depending on $\vartheta_1$, $\vartheta_2$, $\kappa$, $\psi$, $a_0$ and $\| \nabla X_0 \|_2$,
\begin{equation}
\label{eq:stoppingtime:bounds:CV}  
\exp \bigl( \frac{a^2}{C \varepsilon^2} \bigr)
\leq 
{\mathbb E} \bigl( \tau_{\varepsilon} \bigr) \leq \exp \Bigl( \frac{C a^2 ( a^2 + \| X_0 - \bar{X_0} \|_2^2)}{\varepsilon^2} \Bigr),
\end{equation}
where {$\bar X_0 := \int_{\mathbb S} X_0(x) \d x$}.
\end{theorem}
 
\begin{remark} 
The following remarks are in order: 
\vskip 4pt

{1. We prove in item 2 immediately below that $X_0$ is necessarily a (strict) minimiser of 
$V$ on ${\mathcal O}$. Basically, this follows from the 
condition 
${\mathscr D} V(X_0)=0$ together with 
\eqref{eq:monotonicity:weak}.} Unfortunately, one cannot replace the assumption ${\mathscr D} V(X_0)=0$ with the assumption that $X_0$ minimises $V$, even in the framework of \eqref{eq:monotonicity:weak}. Briefly, a minimiser might not satisfy ${\mathscr D} V(X_0)=0$ because $U^2(\mathbb S)$ has a boundary in $L^2(\mathbb S)$. Typically, a minimiser satisfies 
${\mathscr D} V(X_0)=0$ if the gradient of $X_0$ remains away from zero, which makes it possible 
to consider, in any {smooth (symmetric) direction $f: {\mathbb S} \rightarrow {\mathbb R}$, perturbations $X_0 + \delta f$ that remain in $U^2({\mathbb S})$ for $\delta$ small enough.} Furthermore, is it easy to show that the condition is satisfied for $V(X):=\lVert X-X_0 \rVert_2^2$.
\vskip 4pt

{2. In the framework of the statement}, 
 we can find a constant $C>1$
(which may not be the same as the constant $C$ appearing in the statement of the theorem) such that 
\begin{equation}
\label{eq:potential:L2norm}
\frac1{C} \| X - X_0 \|_2^2 \leq 
V\bigl( \textrm{\rm Leb}_{\mathbb S} \circ X^{-1} \bigr) - V\bigl( \textrm{\rm Leb}_{\mathbb S} \circ X_0^{-1} \bigr) \leq C \| X - X_0 \|^2_2, \quad X \in {\mathcal O}.
\end{equation} 
The above display says that the stopping time $\tau_{\varepsilon}$ 
in
\eqref{eq:stoppingtime:bounds:CV}  
can be equivalently reformulated as 
$\tau_{\varepsilon} := \inf \{ t \geq 0 : V(X_t^{\varepsilon}) - V(X_0) \geq a^2 \}.$ 
{(The two definitions of $\tau_\varepsilon$ do not coincide exactly but yield two quantities of the same order.)}
The proof
of \eqref{eq:potential:L2norm}
 is rather straightforward and relies on the expansion 
(see Definition \ref{def:derivative:mathscr DV})
\begin{equation*}
V(X) = V(X_0) +\int_0^1 \bigl\langle  {\mathscr D} V\bigl(t X + (1-t) X_0\bigr),X-X_0 \bigr\rangle_2 \d t.
\end{equation*} 
By 
\eqref{eq:monotonicity:weak}, this yields
\begin{equation*} 
\begin{split} 
V(X) \geq V(X_0) + \kappa \int_0^1 t \| X - X_0 \|_2^2 \d t \geq V(X_0) + \frac{\kappa}2 \| X - X_0 \|_2^2. 
\end{split} 
\end{equation*} 
{This proves in particular that $X_0$ is a local strict minimiser of $V$.}
Lipschitz continuity of ${\mathscr D} V$ provides the converse bound.

The two bounds in 
\eqref{eq:stoppingtime:bounds:CV}  
are thus consistent with 
the standard bounds for metastable diffusion processes in 
finite dimension, since the exponent $a^2$ therein should be regarded 
as the height of the metastable well formed by the potential around $X_0$. 
The additional constraint $a \in [1/a_0,a_0]$ however, is non-standard. 
Although it is a clear restriction on the scope of the result, it must be stressed that $a_0$ can be chosen as large as needed in the statement. From 
a practical point of view, this says that 
the only wells that are left aside in the analysis are 
 very small or very large ones, which looks acceptable. 
The additional factor $a^2 + \| X_0 - \bar{X_0} \|_2^2$ in the upper bound is also non-standard and also restricts the scope of the result 
to initial conditions whose variance is bounded by known constants. 

Despite the slightly non-standard form of the two bounds, 
the result remains (from our perspective) valuable. It should be regarded as a proof of concept highlighting the possible interest of an infinite dimensional common noise in the study of a minimization problem on the space of probability measures. 
Of course, one may think to push the analysis further. For instance,
metastability questions could be addressed by large deviations, as done in the seminal monograph by  Freidlin and Wentzell
\cite{FreidlinWentzell}. {In parallel to this question}, we plan to address numerical examples 
in order to complement this theoretical work. 
\vskip 4pt

{3. In comparison with standard metastability bounds, the two peculiarities} explained above come from the proof of the upper bound, which is challenging due to the presence of the reflecting term in the dynamics
and because of the infinite dimensional nature of the problem. In particular, 
we explain in the next items below that, when working on the space 
${\mathcal P}_2({\mathbb R})$, we cannot reasonably expect 
the wells of the potential $V$ 
to have a smooth shape. This makes a substantial difference 
with the analogue in finite dimension and, combined with the presence of a reflection, presents significant difficulties.  
\vskip 4pt

{4.} As we have just alluded to, 
the 
condition 
\eqref{eq:monotonicity:weak}
is typically satisfied if $V$ is assumed to be locally convex 
on the neighbourhood ${\mathcal O}$ of $X_0$. Whilst this may look an obvious remark, it raises several questions about the notion of local convexity
on the space ${\mathcal P}_2({\mathbb R})$.  
Here, $\kappa$-convexity on ${\mathcal O}$ may be formulated as
\begin{equation}
\label{eq:kappa:convexity:Q:sense}
V (Y) \geq V(X) +
\bigl\langle {\mathscr D} V(X),Y-X \bigr\rangle_2 +  \frac{\kappa}{2} 
\| X - Y \|_2^2, 
\end{equation} 
for any $X,Y \in {\mathcal O}$. 
A typical instance is $V(X) = \| X - X_0 \|_2^2$.

It must be stressed that the above notion of convexity is consistent with the more standard notion of convexity that would require the
map 
\begin{equation}
\label{eq:canonical:representative} 
{\mathcal P}_2({\mathbb R}) \ni \mu \mapsto V \bigl(   {F_\mu^{-1}}  \bigr) 
\end{equation} 
to be locally $\kappa$-displacement convex (see for instance 
\cite[Chapter 5]{CarmonaDelarueI}), where we recall that $F_\mu^{-1}$ is the canonical
representative of $\mu$ in the space 
  $U^2({\mathbb S})$.
Observe indeed that the   
  map 
  \eqref{eq:canonical:representative}
  is the canonical identification of $V$ when written as a function on the space ${\mathcal P}_2({\mathbb R})$ (and not 
  as a function on $U^2({\mathbb S})$). 

Even though (local) convexity may be equivalently reformulated in terms of the more standard notion of 
displacement (local) convexity, 
important differences may arise when combining convexity with more standard notions of differentiability. 
In particular, 
assuming 
the function
\eqref{eq:canonical:representative} 
to be both locally $\kappa$-convex and continuously differentiable in Lions' sense (i.e., the derivative $\partial_\mu$ does exist)
is {significantly} more demanding than assuming 
$V$ to be both locally $\kappa$-convex and continuously differentiable in ${\mathscr D}$ sense. Indeed, 
if
$V$ (as in 
\eqref{eq:canonical:representative})  is Lions differentiable, then its Lions' derivative at $X$ reads $x \mapsto \partial_\mu V( {\rm Leb}_{\mathbb S} \circ X^{-1})(X(x))$ and is thus $\sigma(X)$-measurable. As we already explained in 
Subsection \ref{subse:derivative:Q}, the latter may not be true 
for ${\mathscr D}V(X)$, unless $V$ is Lions' differentiable, in which case both derivatives coincide. This is the reason why 
the approach based on the derivative ${\mathscr D}$ is more general. To wit, 
for $X_0$ a local minimizer of $V$, define for every integer $n \in {\mathbb N}_0$
$\phi_n : y \in {\mathbb R}   \mapsto 2^{-n} \lfloor 2^n y \rfloor$
and call $X_0^n$ the conditional expectation of $X_0$ given $\phi_n(X_0)$ (the probability space being 
here understood as ${\mathbb S}$ equipped with $\textrm{\rm Leb}_{\mathbb S}$). 
Obviously, the $\sigma$-field generated by $\phi_n(X_0)$ is included in the $\sigma$-field generated by 
$\phi_{n+1}(X_0)$ and, by martingale convergence theorem, one deduces that $\| X^n_0 -  X_0\|_2 \rightarrow 0$
as $n$ tends to $\infty$. 
In particular, for $n$ large enough, both $X^n_0$ and $X_0$ are in the convex set where 
$V$ is $\kappa$-convex. Choosing $Y$ as $X_0$ and $X$ as $X^n_0$ in \eqref{eq:kappa:convexity:Q:sense}, the key is to observe that 
\begin{equation*} 
\int_{\mathbb S} \partial_\mu V \bigl( \textrm{\rm Leb}_{\mathbb S} \circ (X^n_0)^{-1} \bigr) 
\bigl( X_0^n(x) \bigr) \bigl( X_0(x) - X_0^n(x) \bigr) \d x =0,
\end{equation*}  
because $X_0^n$ is the conditional expectation of $X_0$ given $\phi_n(X_0)$. As a result of 
local $\kappa$-convexity, $X_0$ (which is assumed to be a local minimizer of $V$) is necessarily equal to $X_0^n$ and is thus constant 
on every event of the form $\{ k/2^n \leq X_0 < (k+1)/2^n\}$, for 
$k \in {\mathbb N}$.
In fact, the same holds if one replaces
$\phi_n$ by $\phi_n : {\mathbb R} \ni y \mapsto 
2^{-n} \lfloor 2^n y+1/2 \rfloor$, from which we deduce that,  {for $n$ large enough again}, $X_0$ is constant on every
event of the form $\{ k/2^n + 1/2^{n+1}\leq X_0 < (k+1)/2^n+1/2^{n+1}\}$, for 
$k \in {\mathbb N}$. The combination of the latter two observations 
says that the value of $X_0$ on 
$[k/2^n, (k+1)/2^n)$ is the same as the value on 
$[k/2^n+1/2^{n+1},
(k+1)/2^n+1/2^{n+1})$ (because the two intervals overlap)
and is also the same as the value on 
$[(k+1)/2^n, (k+2)/2^n)$, as the result of which $X_0$ is necessarily constant!
This is a variant of Jensen's inequality, which puts 
a strong form of rigidity on locally convex smooth functions. 
 
When $V$ is not required to be Lions differentiable, an obvious example 
of a convex potential with a non-trivial minimizer is  $V(X) = \| X - X_0 \|_2^2$, for a non-trivial $X_0$.  
\vskip 4pt

5. 
Rigidity of smooth locally convex functions on ${\mathcal P}({\mathbb R})$ forces us to 
focus on weaker functions in the statement of Theorem 
\ref{thm:meta}, as otherwise the class of examples would be rather limited. 
Whilst this explains our choice to assume the potential $V$ to be merely 
continuously differentiable in ${\mathscr D}$ sense, this makes the proof more difficult. Indeed, a natural argument 
to show Theorem \ref{thm:meta} would consist in expanding  $(V(X_t^{\varepsilon}))_{t \geq 0}$
(by means of an It\^o's formula {that we intend to give in a forthcoming contribution}), but this would
require $V$ to be smooth enough (in Lions' sense) and thus would contradict our assumption. Instead, 
our strategy is to expand $(\| X_t^{\varepsilon} - X_0 \|_2^2)_{t \geq 0}$, but, by doing so, we pay for the non-smoothness 
(in Lions' sense) of the function $U^2({\mathbb S}) \ni X \mapsto \| X- X_0 \|_2^2$ whose derivative at $X$ {may} not have a mean-field structure, i.e. {may not be} $\sigma(X)$-measurable. 
As a result of this, 
the reflecting term is still present when expanding $(\| X_t^{\varepsilon} - X_0 \|_2^2)_{t \geq 0}$, which is a serious difficulty. This explains why the bound 
\eqref{eq:stoppingtime:bounds:CV}
is less accurate than the usual ones in finite dimension.  
\vskip 4pt

6. Whenever $V$ is locally $\kappa$-convex and 
smooth in Lions' sense, we know 
from item 3 that the local minimizer $X_0$ is a constant. 
As before, $V(X)$ behaves locally like 
$\| X-X_0 \|_2^2$, 
but differently 
from the case when $X_0$ has a non-zero variance, 
there is no difficulty for expanding 
$(\| X_t^{\varepsilon} - X_0\|_2^2)_{t \geq 0}$ in this situation. 
As a result, our method of proof allows us to recover the 
usual lower and upper bounds for the 
mean exit time ${\mathbb E} (\tau_\varepsilon)$, see 
Remark 
\ref{rem:4.5}. 
This shows that the non-standard form of the main estimate in 
the statement of Theorem 
\ref{thm:meta}
is really due to our willingness 
to allow for non-trivial minimizers, which is not an easy task due to the rigidity of smooth locally convex functions on ${\mathcal P}_2({\mathbb R})$. 
\end{remark}

The rest of this subsection is dedicated to the proof of Theorem 
\ref{thm:meta}.

\subsubsection{Proof of the Metastability Estimate: Lower Bound}
 
The lower bound in Theorem 
\ref{thm:meta} is a consequence of the following proposition
{(with the same notations as in 
\eqref{eq:notations:visco:1},
\eqref{eq:notations:visco:2}
and
\eqref{eq:monotonicity:weak})}: 
\begin{proposition}
\label{eq:lower:bound}
There exists a constant $c$, only depending on $\lambda$, such that, 
for a parameter $\sigma>0$ and 
under the condition
\begin{equation}
\label{eq:condition:lower:bound}
 a^2 \frac{\kappa -  \sigma \max(\vartheta_1 ,  \vartheta_2 , \varepsilon^{2  {\lambda} \beta})}{\varepsilon^2} 
\geq \frac{(\vartheta_1 + \psi \vartheta_2 +  c \varepsilon^{ {2 \lambda \beta}})}{\varepsilon^{\beta}}
+ \frac{1}{\varepsilon^{2-2\alpha}} \| \nabla X_0 \|_2^2, 
\end{equation} 
it holds that
\begin{equation} 
\label{eq:lower:bound:result:bound:bound}
\begin{split}
\exp \bigl( \sigma \frac{a^2}{2 \varepsilon^2} \bigr) 
&\leq 
\Bigl( \frac{\sigma}{\varepsilon^{\beta}}  \bigl( \vartheta_1 + \psi \vartheta_2 +   c \varepsilon^{ {2\lambda  \beta}} \bigr)   
 + 
\frac{\sigma}{\varepsilon^{2-2 \alpha}}
\|   \nabla X_0 \|^2_2 
\Bigr) 
  {\mathbb E} ( {\tau_{\varepsilon}} ).
  \end{split}
  \end{equation}  
\end{proposition}

\begin{proof} 
In what follows, we use the process 
\begin{equation*}
{\mathcal E}_t^{\varepsilon} :=
 \exp \Bigl(  {\sigma} \frac{ \| X_t^{\varepsilon} - X_0^\varepsilon \|_2^2}{\varepsilon^2} \Bigr),
 \quad t \geq 0,
\end{equation*}
for the same parameter $\sigma >0$ as in the statement. By following the It\^o expansion performed in the proof of \cite[Proposition 4.11 and Corollary 4.12]{delarueHammersley2022rshe} ({we apply the standard It\^o formula to the square of each Fourier mode of a mollification of the solution process, sum over the Fourier modes and then remove the mollification}), we have 
\begin{equation}
\label{eq:expansion:etvarepsilon:0} 
\begin{split} 
\d_t \Bigl( \| X_t^{\varepsilon}  -  X_0 \|_2^2\Bigr) 
&\leq  - 2 \bigl\langle {\mathscr D} V\bigl(X_t^{\varepsilon}\bigr), X_t^{\varepsilon} - X_0  \bigr\rangle_2 \d t
- 
2 \varepsilon^{2 \alpha} \langle \nabla_x X_t^{\varepsilon}, \nabla_x (X_t^{\varepsilon} - X_0)  \bigr\rangle_2 \d t
\\
&\hspace{-15pt} 
+ \varepsilon^2 \Bigl( \vartheta_1 \varepsilon^{-\beta} + 
\vartheta_2
{\psi} \varepsilon^{-\beta}  
+
\sum_{k >  {\psi} \varepsilon^{-\beta}} \frac1{k^{2\lambda}}
\Bigr) \d t
+ 2 \varepsilon \bigl\langle X_t^{\varepsilon} - X_0 , \d W_t^\varepsilon \bigr\rangle_2,
\end{split}
\end{equation} 
which follows intuitively from the fact that 
$\langle X_t^{\varepsilon} - X_0, \d \eta_t^{\varepsilon} \rangle$ has a negative contribution. 
Therefore,
\begin{equation} 
\label{eq:expansion:etvarepsilon0}
\begin{split}
\d_t {\mathcal E}_t^{\varepsilon} &{\leq} - \frac{2\sigma}{\varepsilon^2} 
{\mathcal E}_t^{\varepsilon} 
 \bigl\langle {\mathscr D} V\bigl(X_t^{\varepsilon}\bigr), X_t^{\varepsilon} - X_0 \bigr\rangle_2 \d t
- 
\frac{2\sigma}{\varepsilon^{2-2 \alpha}}
{\mathcal E}_t^{\varepsilon}   \langle \nabla_x X_t^{\varepsilon} , \nabla_x (X_t^{\varepsilon} - X_0) \rangle_2 \d t
\\
&\hspace{5pt}
+
\sigma {\mathcal E}_t^{\varepsilon} 
\Bigl( \vartheta_1 \varepsilon^{-\beta} + 
\vartheta_2
{\psi} \varepsilon^{-\beta}  
+
\sum_{k >  {\psi} \varepsilon^{-\beta}} \frac1{k^{2\lambda}}
\Bigr)
 \d t  +
\frac{2\sigma}{\varepsilon} 
{\mathcal E}_t^{\varepsilon} 
  \bigl\langle X_t^{\varepsilon}- X_0 , \d W_t^\varepsilon \bigr\rangle_2
\\
&\hspace{5pt} +
\frac{2\sigma^2}{\varepsilon^2} 
{\mathcal E}_t^{\varepsilon} 
\Bigl( \vartheta_1 \sum_{k \leq    \varepsilon^{-\beta}}
\bigl\vert \widehat{X_t^{\varepsilon}}^k
- \widehat{X_0}^k
 \bigr\vert^2 
+ 
 { \vartheta_2 
\sum_{ \varepsilon^{-\beta} < k \leq \psi \varepsilon^{-\beta}}
\bigl\vert \widehat{X_t^{\varepsilon}}^k
- \widehat{X_0}^k
 \bigr\vert^2 }
 \\
&\hspace{30pt} +
 \sum_{k > {\psi} \varepsilon^{-\beta}}
\frac1{k^{2 \lambda}}
\bigl\vert \widehat{X_t^{\varepsilon}}^k - \widehat{X_0}^k \bigr\vert^2 
\Bigr) \d t.
\end{split}
\end{equation} 
Now, using the property 
\eqref{eq:monotonicity:weak} {and Young's inequality}, there exists a constant $c$, depending on $\lambda$ such that ({for $t < \tau_\varepsilon$})
\begin{equation} 
\label{eq:expansion:etvarepsilon}
\begin{split}
\d_t {\mathcal E}_t^{\varepsilon} &\leq - \frac{2\sigma \kappa}{\varepsilon^2} 
{\mathcal E}_t^{\varepsilon} 
\| X_t^{\varepsilon} - X_0 \|_2^2 \d t
- 
\frac{\sigma}{\varepsilon^{2-2 \alpha}}
{\mathcal E}_t^{\varepsilon}   
\|   \nabla_x X_t^{\varepsilon} \|^2_2 \d t + 
\frac{\sigma}{\varepsilon^{2-2 \alpha}}
{\mathcal E}_t^{\varepsilon}   
\|   \nabla_x X_0 \|^2_2  \d t
\\
&\hspace{5pt}
+
\frac{\sigma}{\varepsilon^{\beta}} {\mathcal E}_t^{\varepsilon} 
 \bigl( \vartheta_1 + \psi \vartheta_2 +  c \varepsilon^{  {2 \lambda} \beta} \bigr) 
 \d t
 +
\frac{2\sigma}{\varepsilon} 
{\mathcal E}_t^{\varepsilon} 
  \bigl\langle X_t^{\varepsilon}- X_0 , \d W_t^\varepsilon \bigr\rangle_2
\\
&\hspace{5pt} +
\frac{2\sigma^2}{\varepsilon^2} 
{\mathcal E}_t^{\varepsilon} 
\max  \bigl( \vartheta_1 , \vartheta_2 , \varepsilon^{ {2 \lambda \beta}} \bigr) 
\| X_t^{\varepsilon} - X_0 \|_2^2 \d t.
\end{split}
\end{equation} 
Recalling that 
\begin{equation*}
\tau_{\varepsilon} := \inf \bigl\{ t \geq 0 : \| X_t^{\varepsilon}  - X_0^\varepsilon \|_2 \geq a 
\bigr\},
\end{equation*}
and assuming that {${\mathbb E}(\tau_\varepsilon) < \infty$} ({if not true, the bound
\eqref{eq:lower:bound:result:bound:bound} in the statement is obvious}), 
we obtain 
(integrating 
\eqref{eq:expansion:etvarepsilon}
from $0$ to $T \wedge \tau_\varepsilon$ and then letting $T$ tend to $\infty$)
\begin{equation*}
\begin{split}
&\exp \bigl( \sigma \frac{a^2}{\varepsilon^2} \bigr) 
+ 2 \sigma \frac{\kappa-  \sigma 
\max (\vartheta_1 , \vartheta_2 , \varepsilon^{2  {\lambda} \beta} )}{ \varepsilon^{2}} {\mathbb E} \int_0^{\tau_{\varepsilon}} 
{\mathcal E}_t^{\varepsilon} 
\| X_t^{\varepsilon} - X_0\|_2^2 
  \d t
\\
&\leq \frac{\sigma}{\varepsilon^{\beta}}  \bigl( \vartheta_1 + \psi \vartheta_2 +  c \varepsilon^{ {2\lambda  \beta}} \bigr)   {\mathbb E} \int_0^{\tau_{\varepsilon}} {\mathcal E}_t^{\varepsilon} 
 \d t
 + 
\frac{\sigma}{\varepsilon^{2-2 \alpha}}
\|   \nabla_{{x}} X_0 \|^2_2 
 {\mathbb E} \int_0^{\tau_{\varepsilon}} 
{\mathcal E}_t^{\varepsilon}   
 \d t.
\end{split}
\end{equation*} 
And then,
\begin{equation*}
\begin{split}
&\exp \bigl( \sigma \frac{a^2}{\varepsilon^2} \bigr) 
+  \sigma a^2 \frac{\kappa-   \sigma 
\max(\vartheta_1 , \vartheta_2 , \varepsilon^{ {2 \lambda \beta}} )}{ \varepsilon^{2}} {\mathbb E} \int_0^{\tau_{\varepsilon}} 
{\mathcal E}_t^{\varepsilon} 
{\mathbf 1}_{\{ 
\| X_t^{\varepsilon} - X_0\|_2 \geq a/\sqrt{2} \}}
  \d t
\\
&\leq \frac{\sigma}{\varepsilon^{\beta}}  \bigl( \vartheta_1 + \psi \vartheta_2 +  c \varepsilon^{2\lambda  \beta} \bigr)   {\mathbb E} \int_0^{\tau_{\varepsilon}} {\mathcal E}_t^{\varepsilon} 
 \d t
 + 
\frac{\sigma}{\varepsilon^{2-2 \alpha}}
\|   \nabla_{ {x}} X_0 \|^2_2 
 {\mathbb E} \int_0^{\tau_{\varepsilon}} 
{\mathcal E}_t^{\varepsilon}   
 \d t.
\end{split}
\end{equation*} 
Under the condition
\eqref{eq:condition:lower:bound}, we get 
\begin{equation*}
\begin{split}
&\exp \bigl( \sigma \frac{a^2}{\varepsilon^2} \bigr)  
\leq 
\Bigl( \frac{\sigma}{\varepsilon^{\beta}} \bigl( \vartheta_1 + \psi \vartheta_2  + c \varepsilon^{2\lambda  \beta} \bigr)   
 + 
\frac{\sigma}{\varepsilon^{2-2 \alpha}}
\|   \nabla_{ {x}} X_0 \|^2_2 
\Bigr) 
 {\mathbb E} \int_0^{\tau_{\varepsilon}} 
{\mathcal E}_t^{\varepsilon}   
{\mathbf 1}_{\{ 
\| X_t^{\varepsilon} - X_0\|_2 < a/\sqrt{2} \}}
 \d t,
\end{split}
\end{equation*} 
which allows us to obtain 
\begin{equation*} 
\begin{split} 
\exp \bigl( \sigma \frac{a^2}{\varepsilon^2} \bigr) 
&\leq 
\Bigl( \frac{\sigma}{\varepsilon^{\beta}} \bigl( \vartheta_1 + \psi \vartheta_2 +  c \varepsilon^{2\lambda  \beta} \bigr)   
 + 
\frac{\sigma}{\varepsilon^{2-2 \alpha}}
\|   \nabla_{ {x}} X_0 \|^2_2 
\Bigr) \exp \bigl( \sigma \frac{a^2}{2\varepsilon^2} \bigr) 
\\
&\hspace{15pt} \times {\mathbb E} \int_0^{\tau_{\varepsilon}}    
{\mathbf 1}_{\{ 
\| X_t^{\varepsilon} - X_0\|_2 < a/\sqrt{2} \}}
 \d t
 \\
&\leq 
\Bigl( \frac{\sigma}{\varepsilon^{\beta}} \bigl( \vartheta_1 + \psi \vartheta_2 + c \varepsilon^{2\lambda  \beta} \bigr)   
 + 
\frac{\sigma}{\varepsilon^{2-2 \alpha}}
\|   \nabla_{ {x}} X_0 \|^2_2 
\Bigr) \exp \bigl( \sigma \frac{a^2}{2\varepsilon^2} \bigr) 
  {\mathbb E} ( {\tau_{\varepsilon}} ), 
 \end{split}  
 \end{equation*}
which completes the proof of the lower bound. 
\end{proof} 

\subsubsection{Proof of the Metastability Estimate: Upper Bound}
 
\begin{proposition}
\label{eq:upper:bound}
{Set $\alpha=\beta=2$ in 
\eqref{eq:notations:visco:2}}
and 
fix $a_0>{1}$. Then, there exist positive reals $\theta$, $\chi$ and $\varpi$ independent of $a_0$
and satisfying 
$\theta >  {3} \varpi$, such that, for
\begin{equation*} 
\vartheta_1 = \vartheta, \quad \vartheta_2 = \vartheta^{\theta}, \quad \gamma = \vartheta^{\chi}, 
\quad \psi = \vartheta^{-\varpi},
\end{equation*} 
and for $\vartheta$ small enough (in terms of the parameter $a_0$ and the properties of $V$), 
for $C_\vartheta$ large enough (in terms of the parameters $a_0$ and $\vartheta$ and of the 
properties of $V$)
and then for $\varepsilon$ small enough, it holds {for $a \leq a_0$}:
\begin{equation*} 
 \frac{1}{C_\vartheta \varepsilon^2} 
 {\mathbb E} ( \tau_{\varepsilon})
\leq
{\mathbb E} \Bigl[ \exp \bigl( \frac{C_\vartheta {\max(1,a^2)}}{\varepsilon^2}
\bigl[ a^2+ \| X_0 - {\bar{X_0}} \|_2^2 \bigr] \bigr) 
 \Bigr]. 
\end{equation*}
\end{proposition}

\begin{proof}
The proof of the upper bound is much more difficult. There are two reasons for this: 
\begin{enumerate}[(i)]
\item
 The first one is that 
we cannot use the same functional $({\mathcal E}_t^{\varepsilon})_{t \geq 0}$ as in the proof 
of the lower bound. This is due to the fact that, in the expansion 
\eqref{eq:expansion:etvarepsilon:0}, the contribution 
$\langle X_t^{\varepsilon} - X_0, \d \eta_t^{\varepsilon} \rangle$ of the reflected process has 
a negative sign, which is the right sign to obtain a lower bound. However, this becomes the wrong sign 
to obtain an upper bound. This forces us to use another functional below, but at the price of 
weakening the resulting upper bound. 
\item The second issue is that the contribution of 
the Laplacian (in the dynamics \eqref{eq:SGD:small:visco})
also has the wrong sign in the derivation of the upper bound. This forces us to  get relevant bounds for 
 $\| \nabla X_t^{\varepsilon} \|_2^2$, but this is not an easy task. 
\end{enumerate} 

Another observation, is that we can assume without any loss of generality 
that 
${\bar X_0}=0$. This just amounts to 
replace $(X_t^{\varepsilon})$ by $(X_t^{\varepsilon} 
- {\bar X_0})_{t \geq 0}$
and thus to 
use the shifted potential 
$u \mapsto V({\bar X_0} + u)$. 
\vskip 4pt
 
\textit{First Step.} 
In place of $({\mathcal E}_t^{\varepsilon})_{t \geq 0}$, as used in the analysis of the lower bound, we now use 
\begin{equation*} 
{\mathcal D}_t^{\varepsilon} := 
\exp \Bigl( \frac{\sigma}{\varepsilon^2} \bigl[ \| X_t^{\varepsilon} \|_2^2 - \|X_0 \|_2^2 \bigr] \Bigr). 
\end{equation*} 
The very good point is that the expansion for the argument in the exponential is straightforward
(compare for instance with 
\eqref{eq:expansion:etvarepsilon:0})
\begin{equation*} 
\begin{split}
\d_t \bigl[ \| X_t^{\varepsilon} \|_2^2 - \|X_0 \|^2_2 \bigr] 
&{\geq} - 2 \bigl\langle {\mathscr D} V\bigl(X_t^{\varepsilon}\bigr), X_t^{\varepsilon}  \bigr\rangle_2 \d t
- 
2 \varepsilon^{2 \alpha} \bigl\| 
\nabla_x X_t^{\varepsilon}  \bigr\|_2^2 \d t
\\
&\hspace{-15pt} 
+ \varepsilon^2 \Bigl( \vartheta_1 \varepsilon^{-\beta} + 
\vartheta_2
\lfloor (\psi-1) \varepsilon^{-\beta} \rfloor 
+
\sum_{k >  {\psi} \varepsilon^{-\beta}} \frac1{k^{2\lambda}}
\Bigr) \d t
+ 2 \varepsilon \bigl\langle X_t^{\varepsilon} , \d W_t^\varepsilon \bigr\rangle_2,
\end{split}
\end{equation*} 
And then (similar to 
\eqref{eq:expansion:etvarepsilon}), {by using the local 
Lipschitzianity of {${\mathscr D} V$}, we get for $t < \tau_\varepsilon$,}
\begin{equation}
\label{eq:expansion:dtvarepsilon} 
\begin{split} 
\d_t {\mathcal D}_t^{\varepsilon} &\geq  
- C  \frac{\sigma}{\varepsilon^2} 
{\mathcal D}_t^{\varepsilon}
\| X_t^{\varepsilon} \|_2 
\| X_t^{\varepsilon} - X_0 \|_2 
\d t
- 
2 
\frac{\sigma}{\varepsilon^{2-2 \alpha}} 
{\mathcal D}_t^{\varepsilon}
  \bigl\| 
\nabla_x X_t^{\varepsilon}  \bigr\|_2^2 \d t
\\
&\hspace{+15pt} 
+ \sigma {\mathcal D}_t^{\varepsilon} \Bigl( \vartheta_1 \varepsilon^{-\beta} + 
\vartheta_2
\lfloor (\psi-1) \varepsilon^{-\beta} \rfloor 
+
\sum_{k >  {\psi} \varepsilon^{-\beta}} \frac1{k^{2\lambda}}
\Bigr) \d t + 2 \frac{\sigma}{\varepsilon} {\mathcal D}_t^{\varepsilon} \bigl\langle X_t^{\varepsilon} , \d W_t^\varepsilon \bigr\rangle_2
\\
&\hspace{15pt} + \frac{2\sigma^2}{\varepsilon^2} 
{\mathcal D}_t^{\varepsilon} 
\Bigl( \vartheta_1 \sum_{k \leq    \varepsilon^{-\beta}}
\bigl\vert \widehat{X_t^{\varepsilon}}^k
 \bigr\vert^2 
+ 
 { \vartheta_2 
\sum_{ \varepsilon^{-\beta} < k \leq \psi \varepsilon^{-\beta}}
\bigl\vert \widehat{X_t^{\varepsilon}}^k
 \bigr\vert^2 } 
  +
 \sum_{k > {\psi} \varepsilon^{-\beta}}
\frac1{k^{2 \lambda}}
\bigl\vert \widehat{X_t^{\varepsilon}}^k  \bigr\vert^2 
\Bigr) \d t,
\end{split} 
\end{equation} 
{for a constant $C$ depending on the local Lipschitz property of ${\mathscr D} V$.}
\vskip 4pt 

\textit{Second Step.} In order to circumvent the second of the two difficulties reported in the presentation of the proof, we introduce an auxiliary functional that measures the size of the large Fourier modes of $(X_t^\varepsilon)_{t \geq 0}$. For $\gamma \in (0,1)$, 
we thus let 
$A$ be the symmetric operator
\begin{equation*} 
A := \Bigl[ I - \exp \bigl(   \tfrac{1}{4\pi^2} \gamma  \varepsilon^{2 \beta}  \Delta_x \bigr) \Bigr]^{1/2},
\end{equation*} 
and we expand 
$(\| A X_t^{\varepsilon} \|_2^2)_{t \geq 0}$ by a convenient version of It\^o's formula. However, this requires a special attention to the reflection term. 
In order to make the proof rigorous, the first step is to replace $A$ by 
\begin{equation*} 
A_{\delta} : =  \Bigl[ \exp \bigl( {\tfrac1{4 \pi^2}} \delta \Delta_x \bigr) - \exp \bigl( {\tfrac1{4 \pi^2}}  ( \gamma  \varepsilon^{2 \beta} +\delta) \Delta_x \bigr) \Bigr]^{1/2},
\end{equation*} 
{for a small parameter $\delta >0$}. 
Following the proof of \cite[(4.30)--(4.31)]{delarueHammersley2022rshe}, we get, with probability 1, for all $t \geq 0$, 
\begin{equation} 
\label{eq:Ito:A}
\begin{split}
\d_t  \bigl\| A_{\delta} X_t^{\varepsilon} \bigr\|_2^2 
&= - 2 \varepsilon^{2 \alpha} \bigl\langle A_\delta \nabla_x X_t^{\varepsilon}, 
A_\delta \nabla_x X_t^{\varepsilon} \bigr\rangle_2 \d t 
+ 2 \varepsilon \bigl\langle A_\delta^2 X_t^{\varepsilon}, \d W_t^{\varepsilon} \bigr\rangle_2 
\\
&\hspace{15pt} - 2 \bigl\langle A_\delta^2 X_t^{\varepsilon}, {\mathscr D} V(X_t^{\varepsilon}) \bigr\rangle_2 \d t
+ \varepsilon^2 \d  \bigl\llangle A_\delta W_t^{\varepsilon} \bigr\rrangle_t
 + 2 \bigl\langle A_\delta^2 X_t^{\varepsilon}, \d \eta_t^{\varepsilon} \bigr\rangle_2, 
 \\
 &=:  \d A_{\delta,t}^1  +
  \d A_{\delta,t}^2 + 
   \d A_{\delta,t}^3 +
    \d A_{\delta,t}^4 + 
     \d A_{\delta,t}^5,      
\end{split}
\end{equation} 
where, by definition {(recalling \eqref{eq:notations:visco:2})},  
$$
\varepsilon^2 \frac{\d A_{\delta,t}^4 }{\d t} = 
\varepsilon^2 \frac{\d}{\d t} \bigl\llangle A_\delta   \d W_t^{\varepsilon} \bigr\rrangle_t =
\varepsilon^2 \sum_{k \geq 0} \bigl( \lambda_k^{\varepsilon} \bigr)^2 \exp\bigl[ - \delta k^2\bigr] \bigl( 1 - \exp\bigl[ -  \gamma \varepsilon^{2 \beta} k^2\bigr] \bigr).$$
We now make several observations. 
The first one is that 
\begin{equation*}
\begin{split}
\d A_{\delta,t}^5 = {2} \bigl\langle A_\delta^2 X_t^{\varepsilon}, \d \eta_t^{\varepsilon} \bigr\rangle_2 
&={2}
\Bigl\langle \exp\bigl(  {\tfrac1{4 \pi^2}} 
\delta \Delta_x \bigr) X_t^{\varepsilon}, \d \eta_t^{\varepsilon} \Bigr\rangle_2 
-
{2}\Bigl\langle \exp\bigl( {\tfrac1{4 \pi^2}}   (\delta+  \gamma \varepsilon^{2 \beta}  ) \Delta_x\bigr) X_t^{\varepsilon}, \d \eta_t^{\varepsilon} \Bigr\rangle_2 
\\
&\leq 
{2} \Bigl\langle \exp\bigl(  {\tfrac1{4 \pi^2}}  \delta \Delta_x\bigr) X_t^{\varepsilon}, \d \eta_t^{\varepsilon} \Bigr\rangle_2,
\end{split}
\end{equation*}
with the latter following from  
\cite[Lemma 2.9]{delarueHammersley2022rshe}.
As $\delta$ tends to $0$, the right-hand side converges to $0$ in the following 
sense:
\begin{equation*} 
\sup_{0 \leq t \leq T} \biggl\vert \int_0^t \Bigl\langle \exp\bigl( {\tfrac1{4 \pi^2}} \delta \Delta_x\bigr) X_s^{\varepsilon}, \d \eta_s^{\varepsilon} \Bigr\rangle_2
\biggr\vert \longrightarrow 0
\end{equation*} 
in probability as $\delta$ tends to $0$, see 
\cite[Proposition 4.11]{delarueHammersley2022rshe}. 

The second observation is that
\begin{equation*}
\begin{split}
 {\frac{1}{2\varepsilon^{2\alpha}}}\frac{\d A_{\delta,t}^1}{\d t}&=
-   \bigl\langle A_\delta \nabla_x X_t^{\varepsilon}, A_\delta \nabla_x X_t^{\varepsilon} 
\bigr\rangle
= - \sum_{k \geq 0} \exp \bigl( - \delta k^2 \bigr) \bigl( 1 - e^{ - \gamma \varepsilon^{2 \beta} k^2}  \bigr) 
\bigl\vert
\widehat{ \nabla_x X_t^{\varepsilon}}^k
\bigr\vert^2
\\
& \leq  - 
\sum_{\varepsilon^{-\beta} < k  \leq  \psi  \varepsilon^{-\beta} } \exp \bigl( - \delta k^2 \bigr)
\bigl( 1 - e^{ - \gamma \varepsilon^{2 \beta} k^2}  \bigr)
 \bigl\vert
\widehat{ \nabla_x X_t^{\varepsilon}}^k
\bigr\vert^2
\\
&\hspace{15pt} -
\bigl( 1 - e^{ - \gamma \psi^2}\bigr) \sum_{k  >  \psi  \varepsilon^{-\beta} } \exp \bigl( - \delta k^2 \bigr) \bigl\vert
\widehat{ \nabla_x X_t^{\varepsilon}}^k
\bigr\vert^2.
\end{split}
\end{equation*} 
The third observation is that 
\begin{equation*}
\begin{split}
\frac{\d  A_{\delta,t}^4}{\d t} &= 
\varepsilon^2 \frac{\d}{\d t} \bigl\llangle A_\delta^2   \d W_t^{\varepsilon} \bigr\rrangle_t 
\\
&\leq
\varepsilon^2 \sum_{k \geq 0} \bigl( \lambda_k^{\varepsilon} \bigr)^2 \bigl( 1 - e^{ - \gamma \varepsilon^{2 \beta} k^2}  \bigr) 
\\
&\leq \vartheta_1 \varepsilon^2 \sum_{k \leq   \varepsilon^{-\beta}} 
\bigl( 1 - e^{ - \gamma \varepsilon^{2 \beta} k^2}  \bigr) 
+
 {\vartheta_2  
\varepsilon^2 \sum_{  \varepsilon^{- \beta} < k \leq \psi \varepsilon^{- \beta} }}  
\bigl( 1 - e^{ - \gamma \varepsilon^{2 \beta} k^2}  \bigr) 
+ \varepsilon^2 
\sum_{k >  {\psi} \varepsilon^{-\beta}} \frac{1}{k^{2 \lambda}}
\\
&\leq  \gamma  {( \vartheta_1 + \vartheta_2 {\psi^3})} \varepsilon^{2-\beta}  
+ {c} \varepsilon^{2+\beta(2 \lambda-1)},
\end{split}
\end{equation*} 
{for a constant $c$ depending on $\lambda$.}
 
Using the fact that 
${\mathscr D} V(X_0)= 0$, the fourth observation is that 
\begin{equation*} 
\begin{split} 
&\frac{\d  A_{\delta,t}^3 }{\d t}
\leq 2  \bigl\| A_\delta^2 X_t^\varepsilon \bigr\|_2 \, \|{\mathscr D}  V(X_t^{\varepsilon}) - {\mathscr D}  V(X_0) \|_2 
\\
&{\leq C}  \bigl\| A_\delta^2 X_t^\varepsilon \bigr\|_2 \, \|  X_t^{\varepsilon} - X_0 \|_2
\\
&{\leq C}  \|  X_t^{\varepsilon} - X_0 \|_2
\biggl[   \gamma
 \sum_{k \leq \varepsilon^{-\beta}} 
 \bigl\vert \widehat{X_t^{\varepsilon}}^k 
 \bigr\vert^2 + 
 \sum_{ \varepsilon^{-\beta} < k \leq \psi \varepsilon^{-\beta}}
 \bigl( 1 - e^{ - \gamma \varepsilon^{2 \beta} k^2}  \bigr) 
 \bigl\vert 
  \widehat{X_t^{\varepsilon}}^k 
 \bigr\vert^2
 + \sum_{k > \psi \varepsilon^{-\beta}} 
\bigl\vert 
\widehat{X_t^{\varepsilon}}^k \bigr\vert^2
\biggr]^{1/2}. 
\end{split}
\end{equation*} 
Letting $\delta$ tend to $0$ in 
\eqref{eq:Ito:A}, we finally have
(for $\varepsilon$ small enough with respect to $\gamma$, 
{$\lambda$}, 
$\vartheta_1$ and $\vartheta_2$, {and for a constant $C$ only depending on the local properties of $V$})
\begin{equation*} 
\begin{split}
&\d_t \bigl[ \| A X_t^{\varepsilon} \|_2^2 \bigr] + 
 {2} \varepsilon^{2 \alpha}
  \Bigl[
  \sum_{\varepsilon^{-\beta} < k \leq \psi \varepsilon^{-\beta}} \bigl( 1 - e^{ - \gamma \varepsilon^{2 \beta} k^2}  \bigr)  \bigl\vert
\widehat{ \nabla_x X_t^{\varepsilon}}^k \bigr\vert^2
+
 \bigl( 1 - e^{ - \gamma \psi^2} \bigr) \sum_{k > \psi \varepsilon^{-\beta}}  \bigl\vert
\widehat{ \nabla_x X_t^{\varepsilon}}^k \bigr\vert^2 \Bigr] \d t 
\\
& \leq  C \Bigl[ \gamma  {(\vartheta_1 + \vartheta_2 {\psi^3})}\varepsilon^{2-  \beta} 
+
{c \varepsilon^{2+\beta(2 \lambda-1)}}
\Bigr]
\d t 
\\
&\hspace{5pt} + C 
 \| X_t^\varepsilon - X_0 \|_2
 \biggl[ \gamma  
 \sum_{k \leq \varepsilon^{-\beta}} 
 \bigl\vert \widehat{X_t^{\varepsilon}}^k 
 \bigr\vert^2 + 
 \sum_{ \varepsilon^{-\beta} < k \leq \psi \varepsilon^{-\beta}}
\bigl( 1 - e^{ - \gamma \varepsilon^{2 \beta} k^2}  \bigr) 
 \bigl\vert 
  \widehat{X_t^{\varepsilon}}^k 
 \bigr\vert^2
 + \sum_{k > \psi \varepsilon^{-\beta}} 
\bigl\vert 
\widehat{X_t^{\varepsilon}}^k \bigr\vert^2
\biggr]^{1/2}
\d t 
\\
&\hspace{5pt} + 
2 \varepsilon \bigl\langle A^2 X_t^{\varepsilon}, \d W_t^{\varepsilon}\bigr\rangle_2.
\end{split}
\end{equation*}
We now let
\begin{equation*}
\begin{split}
{\mathcal A}_t^{\varepsilon} := 
\exp \Bigl(- \varrho \frac{\| A X_t^{\varepsilon} \|_2^2}{\varepsilon^2} 
\Bigr),
\end{split}
\end{equation*} 
and we compute
$\d_t {\mathcal A}_t^{\varepsilon}$.  
We get 
\begin{equation*}
\begin{split}
&\d_t {\mathcal A}_t^{\varepsilon}
 {=}  - \frac{\varrho}{\varepsilon^2} 
{\mathcal A}_t^{\varepsilon} \d_t \bigl[ \| A X_t^{\varepsilon} \|_2^2 \bigr]
+ 2 \frac{\varrho^2}{\varepsilon^2} 
{\mathcal A}_t^{\varepsilon} \bigl\llangle \bigl\langle A^2 X_t^{\varepsilon}, \d W_t^{\varepsilon} \bigr\rangle_2 \bigr\rrangle,
\end{split}
\end{equation*} 
where 
\begin{equation*} 
\begin{split}
\bigl\llangle \bigl\langle A^2 X_t^{\varepsilon}, \d W_t^{\varepsilon} \bigr\rangle_2 \bigr\rrangle
&= \sum_{k \geq 0} \bigl( \lambda_k^{\varepsilon} \bigr)^2 \bigl( 1 - \exp\bigl[ -  \gamma \varepsilon^{2 \beta} k^2\bigr] \bigr)^2
\bigl\lvert \widehat{X_t^{\varepsilon}}^k 
\bigr\rvert^2.
\end{split} 
\end{equation*}
And then, 
\begin{align}
&\d_t {\mathcal A}_t^{\varepsilon}
  \geq 
  -  C \frac{\varrho \gamma ( \vartheta_1 + \vartheta_2 
  {\psi^3})
  +
  {c  \varepsilon^{2 \lambda \beta}}
  }{ \varepsilon^{ \beta}} {\mathcal A}_t^{\varepsilon} \d t 
\nonumber
\\
&\hspace{5pt} +
 2 \frac{\varrho^2}{\varepsilon^2} 
{\mathcal A}_t^{\varepsilon}
\biggl[
 \vartheta_2
\sum_{\varepsilon^{-\beta} < k \leq \psi \varepsilon^{-\beta}}   \bigl( 1 -  e^{ -  \gamma \varepsilon^{2 \beta} k^2} \bigr)^2
\bigl\lvert \widehat{X_t^{\varepsilon}}^k 
\bigr\rvert^2 
+
\sum_{k> \psi \varepsilon^{-\beta}} \frac{1}{k^{2 \lambda}}  \bigl( 1 -  e^{ -  \gamma \varepsilon^{2 \beta} k^2} \bigr)^2
\bigl\lvert \widehat{X_t^{\varepsilon}}^k 
\bigr\rvert^2
\biggr]
\d t \nonumber
\\
&\hspace{5pt} 
+
 {2}
 \frac{\varrho }{   \varepsilon^{2-2 \alpha}}  {\mathcal A}_t^{\varepsilon}  
 \biggl[ 
 \sum_{\varepsilon^{-\beta} < k \leq  \psi \varepsilon^{-\beta}}
 \bigl( 1 -  e^{ -  \gamma \varepsilon^{2 \beta} k^2} \bigr)  \bigl\vert
\widehat{ \nabla_x X_t^{\varepsilon}}^k \bigr\vert^2 
+ 
  ( 1 - e^{ - \gamma \psi^2} )
 \sum_{k > \psi \varepsilon^{-\beta}}  \bigl\vert
\widehat{ \nabla_x X_t^{\varepsilon}}^k \bigr\vert^2 
\biggr] \d t 
\label{eq:expansion:atvarepsilon}
\\
&\hspace{5pt}-
C \frac{\varrho}{\varepsilon^2} {\mathcal A}_t^{\varepsilon} 
 \| X_t^\varepsilon - X_0 \|_2
  \biggl[ \gamma
 \sum_{k \leq \varepsilon^{-\beta}} 
 \bigl\vert \widehat{X_t^{\varepsilon}}^k 
 \bigr\vert^2 + 
 \sum_{ \varepsilon^{-\beta} < k \leq \psi \varepsilon^{-\beta}}
\bigl( 1 - e^{ -  \gamma \varepsilon^{2 \beta} k^2} \bigr)
 \bigl\vert 
  \widehat{X_t^{\varepsilon}}^k 
 \bigr\vert^2
 + \sum_{k > \psi \varepsilon^{-\beta}} 
\bigl\vert 
\widehat{X_t^{\varepsilon}}^k \bigr\vert^2
\biggr]^{1/2}
\hspace{-3pt} 
\d t
\nonumber
\\
&\hspace{5pt} -
2 \frac{\varrho}{\varepsilon} {\mathcal A}_t^{\varepsilon} \bigl\langle A^2 X_t^{\varepsilon}, \d W_t^{\varepsilon}\bigr\rangle_2.
\nonumber
\end{align} 

\textit{Third Step.} 
By combining 
\eqref{eq:expansion:dtvarepsilon}
and
\eqref{eq:expansion:atvarepsilon}, we now
 compute $\d_t ( {\mathcal A}_t^{\varepsilon} {\mathcal D}_t^{\varepsilon})$. Obviously, 
\begin{equation*}
\begin{split}
\d_t \bigl( {\mathcal A}_t^{\varepsilon} {\mathcal D}_t^{\varepsilon} \bigr)
&=  {\mathcal A}_t^{\varepsilon} \d_t  {\mathcal D}_t^{\varepsilon} 
+ {\mathcal D}_t^{\varepsilon}  \d_t  {\mathcal A}_t^{\varepsilon}  
 + \d_t \bigl\llangle {\mathcal A}_t^{\varepsilon}, {\mathcal D}_t^{\varepsilon} \bigr\rrangle_t.
\end{split}
\end{equation*}
We first compute
\begin{equation}
\label{eq:expansion:d(at,dt)varepsilon}
\begin{split}
\frac{\d 
\bigl\llangle {\mathcal A}_t^{\varepsilon}, {\mathcal D}_t^{\varepsilon} \bigr\rrangle_t}{\d t}
 &= - 4 \frac{\sigma \varrho}{\varepsilon^2} {\mathcal A}_t^{\varepsilon} 
  {\mathcal D}_t^{\varepsilon}
 \sum_{k \geq 0}\bigl( \lambda_k^{\varepsilon} \bigr)^2 
 \widehat{A^2 X_t^{\varepsilon}}^k \widehat{X_t^{\varepsilon}}^k \d t
 \\
 &= - 4 \frac{\sigma \varrho}{\varepsilon^2} {\mathcal A}_t^{\varepsilon} 
  {\mathcal D}_t^{\varepsilon}
 \sum_{k \geq 0}\bigl( \lambda_k^{\varepsilon} \bigr)^2 
 \bigl( 1 - \exp(-\gamma \varepsilon^{2 \beta} k^2) \bigr) \bigl\vert \widehat{X_t^{\varepsilon}}^k
   \bigr\vert^2 \d t
 \\
 &\geq - 4   \frac{\sigma \varrho}{\varepsilon^2} {\mathcal A}_t^{\varepsilon} 
  {\mathcal D}_t^{\varepsilon}
  \biggl[
  \gamma
  \vartheta_1  
  \sum_{k \leq \varepsilon^{-\beta}}
  \bigl\vert  \widehat{X_t^{\varepsilon}}^k
 \bigr\vert^2 
 + 
\vartheta_2 
  \sum_{ \varepsilon^{-\beta} < k \leq \psi \varepsilon^{-\beta} }
  \bigl( 1 - e^{ -  \gamma \varepsilon^{2 \beta} k^2} \bigr)
  \bigl\vert  \widehat{X_t^{\varepsilon}}^k
 \bigr\vert^2 
 \\
&\hspace{15pt} +
  \sum_{k >  \psi \varepsilon^{-\beta} }
  \frac{1 - e^{ -  \gamma \varepsilon^{2 \beta} k^2}}{k^{2\lambda}}  \bigl\vert  \widehat{X_t^{\varepsilon}}^k
 \bigr\vert^2 
  \biggr].
 \end{split} 
\end{equation}
And then, by  
\eqref{eq:expansion:dtvarepsilon}, 
\eqref{eq:expansion:atvarepsilon}
and
\eqref{eq:expansion:d(at,dt)varepsilon}, 
\begin{equation*}
\begin{split}
&\d_t \bigl( {\mathcal A}_t^{\varepsilon} {\mathcal D}_t^{\varepsilon}
\bigr) 
\\
&\hspace{5pt}+
C \frac{\varrho \gamma ( \vartheta_1 + \vartheta_2 {\psi^3})
+
{c  \varepsilon^{2 \lambda \beta}}
}{ \varepsilon^{ \beta}} {\mathcal A}_t^{\varepsilon}  {\mathcal D}_t^{\varepsilon} \d t 
\\
&\hspace{5pt}
+C
\frac{\varrho}{\varepsilon^2} {\mathcal A}_t^{\varepsilon}  {\mathcal D}_t^{\varepsilon}
 \| X_t^\varepsilon - X_0 \|_2
 \biggl[ \gamma  
 \sum_{k \leq \varepsilon^{-\beta}} 
 \bigl\vert \widehat{X_t^{\varepsilon}}^k 
 \bigr\vert^2 + 
 \hspace{-5pt} 
 \sum_{ \varepsilon^{-\beta} < k \leq \psi \varepsilon^{-\beta}}
 \bigl( 1 - e^{ -  \gamma \varepsilon^{2 \beta} k^2} \bigr)
   \bigl\vert 
  \widehat{X_t^{\varepsilon}}^k 
 \bigr\vert^2
 + \sum_{k > \psi \varepsilon^{-\beta}} 
\bigl\vert 
\widehat{X_t^{\varepsilon}}^k \bigr\vert^2
\biggr]^{1/2}
\hspace{-3pt} 
\d t
\\
&\hspace{5pt} + 
4   \frac{\sigma \varrho}{\varepsilon^2} {\mathcal A}_t^{\varepsilon} 
  {\mathcal D}_t^{\varepsilon}
  \biggl[
  \gamma \vartheta_1
  \sum_{k \leq \varepsilon^{-\beta}}
  \bigl\vert  \widehat{X_t^{\varepsilon}}^k
 \bigr\vert^2 
 + 
\vartheta_2
  \sum_{ \varepsilon^{-\beta} < k \leq \psi \varepsilon^{-\beta} }
 \bigl( 1 - e^{  - \gamma \varepsilon^{2\beta} k^2} \bigr)
  \bigl\vert  \widehat{X_t^{\varepsilon}}^k
 \bigr\vert^2 
 \\
&\hspace{75pt} +
  \sum_{k >  \psi \varepsilon^{-\beta} }
  \frac{1 - e^{ -  \gamma \varepsilon^{2 \beta} k^2}}{k^{2\lambda}}  \bigl\vert  \widehat{X_t^{\varepsilon}}^k
 \bigr\vert^2 
  \biggr] \d t
  \\
&\hspace{5pt} + C \frac{\sigma}{\varepsilon^2} 
 {\mathcal A}_t^{\varepsilon}  {\mathcal D}_t^{\varepsilon} 
\| X_t^{\varepsilon} - X_0 \|_2 
\| X_t^{\varepsilon} \|_2 
 \d t
 +
\frac{2\sigma}{\varepsilon^{2-2 \alpha}}
 {\mathcal A}_t^{\varepsilon}  {\mathcal D}_t^{\varepsilon}   \| \nabla_x X_t^{\varepsilon}  \|_2^2 \d t 
\\ 
&\geq  
{\sigma} 
 {\mathcal A}_t^{\varepsilon}
 {\mathcal D}_t^{\varepsilon} \biggl[ \vartheta_1 \varepsilon^{-\beta} + 
\vartheta_2
\lfloor (\psi-1) \varepsilon^{-\beta} \rfloor 
+
\sum_{k >  {\psi} \varepsilon^{-\beta}} \frac1{k^{2\lambda}}
\biggr] \d t
\\
&\hspace{5pt} 
+ \frac{2\sigma^2}{\varepsilon^2} 
 {{\mathcal A}_t^{\varepsilon}}
{\mathcal D}_t^{\varepsilon} 
\biggl[ \vartheta_1 \sum_{k \leq    \varepsilon^{-\beta}}
\bigl\vert \widehat{X_t^{\varepsilon}}^k
 \bigr\vert^2 
+ 
 { \vartheta_2 
\sum_{ \varepsilon^{-\beta} < k \leq \psi \varepsilon^{-\beta}}
\bigl\vert \widehat{X_t^{\varepsilon}}^k
 \bigr\vert^2 }
  +
 \sum_{k > {\psi} \varepsilon^{-\beta}}
\frac1{k^{2 \lambda}}
\bigl\vert \widehat{X_t^{\varepsilon}}^k  \bigr\vert^2 
\biggr] \d t  
\\
&\hspace{5pt} 
+
 2 \frac{\varrho^2}{\varepsilon^2} 
{\mathcal A}_t^{\varepsilon}
{\mathcal D}_t^{\varepsilon}
\biggl[
 \vartheta_2
\sum_{\varepsilon^{-\beta} < k \leq \psi \varepsilon^{-\beta}}   \bigl( 1 -  e^{ -  \gamma \varepsilon^{2 \beta} k^2} \bigr)^2
\bigl\lvert \widehat{X_t^{\varepsilon}}^k 
\bigr\rvert^2 
+
\sum_{k> \psi \varepsilon^{-\beta}} \frac{1}{k^{2 \lambda}}  \bigl( 1 -  e^{ -  \gamma \varepsilon^{2 \beta} k^2} \bigr)^2
\bigl\lvert \widehat{X_t^{\varepsilon}}^k 
\bigr\rvert^2 
\biggr]
\d t 
\\
&\hspace{5pt} 
+
 {2}
 \frac{\varrho }{   \varepsilon^{2-2 \alpha}}  {\mathcal A}_t^{\varepsilon}  {\mathcal D}_t^{\varepsilon} 
 \biggl[ 
 \sum_{\varepsilon^{-\beta} < k \leq  \psi \varepsilon^{-\beta}}
 \bigl( 1 -  e^{ -  \gamma \varepsilon^{2 \beta} k^2} \bigr)  \bigl\vert
\widehat{ \nabla_x X_t^{\varepsilon}}^k \bigr\vert^2 
+ 
  ( 1 - e^{ - \gamma \psi^2} )
 \sum_{k > \psi \varepsilon^{-\beta}}  \bigl\vert
\widehat{ \nabla_x X_t^{\varepsilon}}^k \bigr\vert^2 
\biggr] \d t
+ \d m_t,
\end{split} 
\end{equation*}
where $(m_t)_{t \geq 0}$ is a generic martingale term. 
Then, we observe that, on the fifth line,  
\begin{equation*}
\begin{split}
\bigl\| \nabla_x X_t^{\varepsilon}  \bigr\|_2^2 &=
\sum_{k \leq  \psi \varepsilon^{- \beta}} 
\bigl\vert
\widehat{ \nabla_x X_t^{\varepsilon}}^k \bigr\vert^2
+
\sum_{k > \psi \varepsilon^{- \beta}} 
\bigl\vert
\widehat{ \nabla_x X_t^{\varepsilon}}^k \bigr\vert^2
\leq 
\psi^{2} \varepsilon^{-2 \beta} \sum_{k \leq \psi \varepsilon^{-\beta}}
\bigl\vert \widehat{X_t^{\varepsilon}}^k \bigr\vert^2
+
\sum_{k > \psi \varepsilon^{- \beta}} 
\bigl\vert
\widehat{ \nabla_x X_t^{\varepsilon}}^k \bigr\vert^2
. 
\end{split}
\end{equation*} 
We get, {by considering {$t < \tau_\varepsilon$}},
\begin{align}
&\d_t \bigl( {\mathcal A}_t^{\varepsilon} {\mathcal D}_t^{\varepsilon}
\bigr) \nonumber
\\
&\hspace{5pt}+
C \frac{\varrho \gamma ( \vartheta_1 + \vartheta_2 {\psi^3})
+
{c  \varepsilon^{2 \lambda \beta}}
}{ \varepsilon^{ \beta}} {\mathcal A}_t^{\varepsilon}  {\mathcal D}_t^{\varepsilon} \d t 
\label{eq:ad:1}
\\
&\hspace{5pt}
+C
\frac{\varrho a}{\varepsilon^2} {\mathcal A}_t^{\varepsilon}  {\mathcal D}_t^{\varepsilon}
 \biggl[ \gamma  
 \sum_{k \leq \varepsilon^{-\beta}} 
 \bigl\vert \widehat{X_t^{\varepsilon}}^k 
 \bigr\vert^2 + 
 \sum_{ \varepsilon^{-\beta} < k \leq \psi \varepsilon^{-\beta}}
 \bigl( 1 - e^{ -  \gamma \varepsilon^{2 \beta} k^2} \bigr)
 \bigl\vert 
  \widehat{X_t^{\varepsilon}}^k 
 \bigr\vert^2
 + \sum_{k > \psi \varepsilon^{-\beta}} 
\bigl\vert 
\widehat{X_t^{\varepsilon}}^k \bigr\vert^2
\biggr]^{1/2}
\d t
\label{eq:ad:2}
\\
&\hspace{5pt}   \label{eq:ad:3} + 
4   \frac{\sigma \varrho}{\varepsilon^2} {\mathcal A}_t^{\varepsilon} 
  {\mathcal D}_t^{\varepsilon}
  \biggl[
  \gamma 
   \vartheta_1
  \sum_{k \leq \varepsilon^{-\beta}}
  \bigl\vert  \widehat{X_t^{\varepsilon}}^k
 \bigr\vert^2 
 + 
\vartheta_2
  \sum_{ \varepsilon^{-\beta} < k \leq \psi \varepsilon^{-\beta} }
 \bigl( 1 - e^{-\gamma \varepsilon^{2 \beta} k^2} \bigr) 
  \bigl\vert  \widehat{X_t^{\varepsilon}}^k
 \bigr\vert^2
 \\ 
 &\hspace{75pt} 
 +
  \sum_{k >  \psi \varepsilon^{-\beta} }
  \frac{  1 - e^{-\gamma \varepsilon^{2 \beta} k^2}  }{k^{2\lambda}}  \bigl\vert  \widehat{X_t^{\varepsilon}}^k
 \bigr\vert^2 
  \biggr] \d t \nonumber
  \\
  &\hspace{5pt} + C \frac{ \sigma a}{\varepsilon^2} {\mathcal A}_t^{\varepsilon} 
{\mathcal D}_t^{\varepsilon} 
   \biggl[  \vartheta_1
 \sum_{k \leq \varepsilon^{-\beta}} 
 \bigl\vert \widehat{X_t^{\varepsilon}}^k 
 \bigr\vert^2 + 
\vartheta_2
 \sum_{ \varepsilon^{-\beta} < k \leq \psi \varepsilon^{-\beta}}
 \bigl\vert 
  \widehat{X_t^{\varepsilon}}^k 
 \bigr\vert^2
 + \sum_{k > \psi \varepsilon^{-\beta}} 
\bigl\vert 
\widehat{X_t^{\varepsilon}}^k \bigr\vert^2
\biggr]^{1/2}
\label{eq:ad:4}
 \\
&\hspace{5pt} +
\frac{2\sigma}{\varepsilon^{2-2 \alpha}}
 {\mathcal A}_t^{\varepsilon}  {\mathcal D}_t^{\varepsilon}   
 \biggl[
 \psi^{2} \varepsilon^{-2 \beta} \sum_{k \leq \psi \varepsilon^{-\beta}}
\bigl\vert \widehat{X_t^{\varepsilon}}^k \bigr\vert^2
+
 \sum_{k >\psi \varepsilon^{- \beta}} 
\bigl\vert
\widehat{ \nabla_x X_t^{\varepsilon}}^k \bigr\vert^2
\biggr]
 \d t
 \label{eq:ad:5} 
\\
&\geq 
\frac{\sigma \vartheta_1}{\varepsilon^\beta} {\mathcal A}_t^{\varepsilon} {\mathcal D}_t^{\varepsilon}     \d t
\label{eq:ad:6}
\\
&\hspace{5pt} 
+
\frac{2\sigma^2}{\varepsilon^2} 
{\mathcal A}_t^{\varepsilon} 
{\mathcal D}_t^{\varepsilon} 
\biggl[ \vartheta_1 \sum_{k \leq   \varepsilon^{-\beta}}
\bigl\vert \widehat{X_t^{\varepsilon}}^k \bigr\vert^2 
+ 
 { \vartheta_2
\sum_{\varepsilon^{-\beta} < k \leq \psi \varepsilon^{-\beta}}
\bigl\vert \widehat{X_t^{\varepsilon}}^k \bigr\vert^2 }
 +
  \sum_{k > {\psi} \varepsilon^{-\beta}}
\frac1{k^{2 \lambda}}
\bigl\vert \widehat{X_t^{\varepsilon}}^k \bigr\vert^2 
\biggr] \d t 
\label{eq:ad:7}
\\
&\hspace{5pt} 
+
\frac{2\varrho^2}{\varepsilon^2} 
{\mathcal A}_t^{\varepsilon} 
{\mathcal D}_t^{\varepsilon} 
\biggl[ 
\vartheta_2
\sum_{\varepsilon^{-\beta} \leq k \leq \psi \varepsilon^{-\beta}}
 \bigl( 1 - e^{-\gamma \varepsilon^{2 \beta} k^2} \bigr)^2  \bigl\vert \widehat{X_t^{\varepsilon}}^k \bigr\vert^2 
 +
\sum_{k > \psi \varepsilon^{-\beta}}
\frac{ \bigl( 1 - e^{-\gamma \varepsilon^{2 \beta} k^2} \bigr)^2}{k^{2 \lambda}}  \bigl\vert \widehat{X_t^{\varepsilon}}^k \bigr\vert^2  
 \biggr]  \d t 
\label{eq:ad:9}
\\
&\hspace{5pt} 
+
 \frac{2\varrho }{   \varepsilon^{2-2 \alpha}}  {\mathcal A}_t^{\varepsilon}  {\mathcal D}_t^{\varepsilon} 
 \biggl[ 
 \sum_{\varepsilon^{-\beta} < k \leq  \psi \varepsilon^{-\beta}}
 \bigl( 1 -  e^{ -  \gamma \varepsilon^{2 \beta} k^2} \bigr)  \bigl\vert
\widehat{ \nabla_x X_t^{\varepsilon}}^k \bigr\vert^2 
+ 
  ( 1 - e^{ - \gamma \psi^2} )
 \sum_{k > \psi \varepsilon^{-\beta}}  \bigl\vert
\widehat{ \nabla_x X_t^{\varepsilon}}^k \bigr\vert^2 
\biggr] \d t
\label{eq:ad:8}
\\
&\hspace{5pt} + \d m_t.
\nonumber
\end{align}

\textit{Fourth Step.}
We treat all the terms above, one by one, choosing $\alpha=\beta=2$. Each term on the left-hand side (except the first one) must be less than 
(up to a multiplicative constant)
the sum of some 
of the terms on the right-hand side. 

We first compare \eqref{eq:ad:1} and \eqref{eq:ad:6}. 
This prompts us to require (modifying the value of the constant
$C$ if necessary) 
\begin{equation}
\label{eq:condition:1}
\sigma \vartheta_1 \geq  C \Bigl[ {\varrho} \gamma \bigl( \vartheta_1+ \vartheta_2  { \psi^3}\bigr)
+
{c  \varepsilon^{2 \lambda \beta}} \Bigr].
\end{equation} 

We now handle \eqref{eq:ad:2}. The idea is to compare it to the three terms \eqref{eq:ad:6}, \eqref{eq:ad:7} and 
\eqref{eq:ad:8}. In order to do so, 
we make use of Young's inequality, writing 
\begin{equation*} 
\begin{split} 
&C \varrho a 
 \biggl[ \gamma  
 \sum_{k \leq \varepsilon^{-\beta}} 
 \bigl\vert \widehat{X_t^{\varepsilon}}^k 
 \bigr\vert^2 + 
 \sum_{ \varepsilon^{-\beta} < k \leq \psi \varepsilon^{-\beta}}
  \bigl( 1 -  e^{ -  \gamma \varepsilon^{2 \beta} k^2} \bigr) 
 \bigl\vert 
  \widehat{X_t^{\varepsilon}}^k 
 \bigr\vert^2
 + \sum_{k > \psi \varepsilon^{-\beta}} 
\bigl\vert 
\widehat{X_t^{\varepsilon}}^k \bigr\vert^2
\biggr]^{1/2}
\\
&\hspace{5pt} \leq
\sigma \vartheta_1
+ 
 \frac{C^2 \varrho^2 a^2}{ {4}\sigma \vartheta_1}  
\biggl[ \gamma  
 \sum_{k \leq \varepsilon^{-\beta}} 
 \bigl\vert \widehat{X_t^{\varepsilon}}^k 
 \bigr\vert^2 + 
 \sum_{ \varepsilon^{-\beta} < k \leq \psi \varepsilon^{-\beta}}
  \bigl( 1 -  e^{ -  \gamma \varepsilon^{2 \beta} k^2} \bigr) 
 \bigl\vert 
  \widehat{X_t^{\varepsilon}}^k 
 \bigr\vert^2
 + \sum_{k > \psi \varepsilon^{-\beta}} 
\bigl\vert 
\widehat{X_t^{\varepsilon}}^k \bigr\vert^2
\biggr]
\\
&\hspace{5pt}  \leq 
\sigma \vartheta_1
+ 
 \frac{C^2 \varrho^2 a^2}{ {4}\sigma \vartheta_1}  
\biggl[ \gamma  
 \sum_{k \leq \varepsilon^{-\beta}} 
 \bigl\vert \widehat{X_t^{\varepsilon}}^k 
 \bigr\vert^2 + 
 \sum_{ \varepsilon^{-\beta} < k \leq \psi \varepsilon^{-\beta}}
 \bigl( 1 -  e^{ -  \gamma \varepsilon^{2 \beta} k^2} \bigr) 
  \bigl\vert 
  \widehat{X_t^{\varepsilon}}^k 
 \bigr\vert^2
 + \frac{\varepsilon^{2 \beta}}{\psi^2} \sum_{k > \psi \varepsilon^{-\beta}} 
k^2 \bigl\vert 
\widehat{X_t^{\varepsilon}}^k \bigr\vert^2
\biggr]. 
\end{split} 
\end{equation*} 
Assuming that there exists $\delta >0$ such that $\gamma^{-\delta} 
\leq \psi$, we notice that 
\begin{equation*}
\begin{split} 
&\sum_{ \varepsilon^{-\beta} < k \leq \psi \varepsilon^{-\beta} }
 \bigl( 1 - e^{-\gamma \varepsilon^{2 \beta} k^2} \bigr) 
  \bigl\vert  \widehat{X_t^{\varepsilon}}^k
 \bigr\vert^2
 \\
 &{=} 
  \sum_{ \varepsilon^{-\beta} < k \leq  \gamma^{-\delta} \varepsilon^{-\beta}   }
 \bigl( 1 - e^{-\gamma \varepsilon^{2 \beta} k^2} \bigr) 
  \bigl\vert  \widehat{X_t^{\varepsilon}}^k
 \bigr\vert^2 +
  \sum_{    \gamma^{-\delta} \varepsilon^{-\beta} <
  k \leq \psi \varepsilon^{-\beta}  }
 \bigl( 1 - e^{-\gamma \varepsilon^{2 \beta} k^2} \bigr) 
  \bigl\vert  \widehat{X_t^{\varepsilon}}^k
 \bigr\vert^2
 \\
 &\leq \gamma^{1-2 \delta} 
  \sum_{ \varepsilon^{-\beta} < k \leq  \gamma^{-\delta}  \varepsilon^{-\beta}   }
  \bigl\vert  \widehat{X_t^{\varepsilon}}^k
 \bigr\vert^2
 + 
 \varepsilon^{2 \beta} 
\gamma^{2 \delta} 
 \sum_{ \gamma^{-\delta} \varepsilon^{-\beta} < k\leq 
 \psi \varepsilon^{-\beta} }
 k^2 
  \bigl( 1 - e^{-\gamma \varepsilon^{2 \beta} k^2} \bigr) 
   \bigl\vert  \widehat{X_t^{\varepsilon}}^k
 \bigr\vert^2.
\end{split}
\end{equation*} 
And then, 
\begin{equation} 
\label{eq:new:bound:epsi:k:epsi}
\begin{split} 
&C \varrho a 
 \biggl[ \gamma  
 \sum_{k \leq \varepsilon^{-\beta}} 
 \bigl\vert \widehat{X_t^{\varepsilon}}^k 
 \bigr\vert^2 + 
 \sum_{ \varepsilon^{-\beta} < k \leq \psi \varepsilon^{-\beta}}
  \bigl( 1 -  e^{ -  \gamma \varepsilon^{2 \beta} k^2} \bigr) 
 \bigl\vert 
  \widehat{X_t^{\varepsilon}}^k 
 \bigr\vert^2
 + \sum_{k > \psi \varepsilon^{-\beta}} 
\bigl\vert 
\widehat{X_t^{\varepsilon}}^k \bigr\vert^2
\biggr]^{1/2}
\\
&\hspace{0pt}  \leq 
\sigma \vartheta_1
+ 
 \frac{C^2 \varrho^2 a^2}{ {4}\sigma \vartheta_1}  
\biggl[ \gamma  
 \sum_{k \leq \varepsilon^{-\beta}} 
 \bigl\vert \widehat{X_t^{\varepsilon}}^k 
 \bigr\vert^2 +
 \gamma^{1-2 \delta} 
  \sum_{ \varepsilon^{-\beta} < k \leq  \gamma^{-\delta}  \varepsilon^{-\beta}   }
  \bigl\vert  \widehat{X_t^{\varepsilon}}^k
 \bigr\vert^2
 \\
&\hspace{15pt} + \varepsilon^{2 \beta}
\gamma^{2 \delta} 
 \sum_{ \gamma^{-\delta} \varepsilon^{-\beta} < k\leq 
 \psi \varepsilon^{-\beta} }
 k^2 
  \bigl( 1 - e^{-\gamma \varepsilon^{2 \beta} k^2} \bigr) 
   \bigl\vert  \widehat{X_t^{\varepsilon}}^k
 \bigr\vert^2
 + \frac{\varepsilon^{2 \beta}}{\psi^2} \sum_{k > \psi \varepsilon^{-\beta}} 
k^2 \bigl\vert 
\widehat{X_t^{\varepsilon}}^k \bigr\vert^2
\biggr]. 
\end{split}
\end{equation}
Up to the leading term 
${\varepsilon^{-2}}{\mathcal A}_t^{\varepsilon} {\mathcal D}_t^{\varepsilon}$, 
the first term {in the right hand side} is equal to 
\eqref{eq:ad:6} (changing the value of the constant 
$C$ in front of the left-hand side, we may just compare to a fraction of 
\eqref{eq:ad:6}). The first and second terms inside the bracket are compared with \eqref{eq:ad:7}, but for some reasons that are clarified next, 
the index $k$ in the second term is just summed between $\varepsilon^{-\beta}$
and $\gamma^{-\delta} \varepsilon^{-\beta}$. 
In particular, we do not proceed in the same way for the third term: we could do so, but then we would get stuck for 
\eqref{eq:ad:3}. 
This prompts us to 
require 
(up to a new value of $C$) 
\begin{equation}
\label{eq:ad:treat:2}  
C \gamma \varrho^2 a^2 \leq \vartheta_1^2 \sigma^3,
\quad 
C \gamma^{1-2 \delta} \varrho^2 a^2 \leq  \vartheta_1 \vartheta_2 \sigma^3.
\end{equation}
{Notice that there should be a factor $8$ appearing in each of the two right-hand sides above, but one can easily get rid of them by increasing the value of the constant $C$ in the corresponding left-hand side. We do this repeatedly in the computations below.} The third term  {inside the 
right-hand bracket of }\eqref{eq:new:bound:epsi:k:epsi}
is compared to 
the first term in 
\eqref{eq:ad:8}, which gives the condition 
\begin{equation}
\label{eq:condition:3}
C \varrho a^2 \gamma^{2 \delta} 
\leq  \sigma \vartheta_1,
\end{equation}
and the fourth term {inside the right-hand bracket of } 
\eqref{eq:new:bound:epsi:k:epsi}
is compared to the last term in 
\eqref{eq:ad:8}, which 
gives the condition
\begin{equation} 
\label{eq:condition:4}
C \varrho a^2 \leq   \psi^2 \sigma  \vartheta_1 \bigl( 1 - e^{-\gamma \psi^2} \bigr). 
\end{equation}

Next, we turn to 
\eqref{eq:ad:3}.
We compare the first term to the first term in 
\eqref{eq:ad:7}. We hence require
\begin{equation}
\label{eq:condition:5} 
C  \varrho \gamma  \leq \sigma.
\end{equation} 
{As above, there should be a factor 2 in the right-hand side, but one may easily remove it.} 

We then address the second term in \eqref{eq:ad:3}.
Proceeding as in the derivation of 
\eqref{eq:new:bound:epsi:k:epsi}, we write
\begin{align}
&4 \frac{\sigma \varrho}{\varepsilon^2} {\mathcal A}_t^{\varepsilon} 
  {\mathcal D}_t^{\varepsilon}
 \vartheta_2
  \sum_{ \varepsilon^{-\beta} < k \leq \psi \varepsilon^{-\beta} }
 \bigl( 1 - e^{-\gamma \varepsilon^{2 \beta} k^2} \bigr) 
  \bigl\vert  \widehat{X_t^{\varepsilon}}^k
 \bigr\vert^2
 \nonumber
\\
&\leq
 4 \frac{\sigma \varrho}{\varepsilon^2} {\mathcal A}_t^{\varepsilon} 
  {\mathcal D}_t^{\varepsilon}
\vartheta_2 \gamma^{1-2 \delta} 
  \sum_{ \varepsilon^{-\beta} < k \leq \gamma^{-\delta} \varepsilon^{-\beta} }
  \bigl\vert  \widehat{X_t^{\varepsilon}}^k
 \bigr\vert^2
 \label{eq:ad:10}
 \\
&\hspace{15pt} +
  2 \frac{\sigma^2}{\varepsilon^2} {\mathcal A}_t^{\varepsilon} 
  {\mathcal D}_t^{\varepsilon}
 \vartheta_2 
  \sum_{ \gamma^{-\delta} \varepsilon^{-\beta} < k \leq \psi \varepsilon^{-\beta} }
  \bigl\vert  \widehat{X_t^{\varepsilon}}^k
 \bigr\vert^2
  +
2 \frac{\varrho^2}{\varepsilon^2} {\mathcal A}_t^{\varepsilon} 
  {\mathcal D}_t^{\varepsilon}
 \vartheta_2 
  \sum_{ \gamma^{-\delta} \varepsilon^{-\beta} < k \leq \psi \varepsilon^{-\beta} }
\bigl( 1 - e^{-\gamma \varepsilon^{2 \beta} k^2} \bigr)^2 
  \bigl\vert  \widehat{X_t^{\varepsilon}}^k
 \bigr\vert^2.
 \nonumber 
\end{align} 
The first term in the right-hand side is compared with the second term in 
\eqref{eq:ad:7} (but only for indices 
$k$ between $\varepsilon^{-\beta}$ and
$\gamma^{-\delta} \varepsilon^{-\beta}$), which prompts us to require 
\begin{equation}
\label{eq:condition:6}
C  \gamma^{1-2 \delta}  \varrho 
\leq \sigma,
\end{equation}
and {the sums on} the last line (of \eqref{eq:ad:10}) 
coincide with terms indexed by the same indices in 
\eqref{eq:ad:7}
and
\eqref{eq:ad:9} (which is a side property of the bracket structure of all these terms). By the same {application of Young's inequality}, we handle the last term in 
\eqref{eq:ad:3}.

 As far as \eqref{eq:ad:4} is concerned, we follow the argument used to handle \eqref{eq:ad:2}. 
 Similarly to 
\eqref{eq:new:bound:epsi:k:epsi}, one obtains 
 \begin{equation} 
\label{eq:new:bound:epsi:k:epsi:2}
\begin{split} 
&C \sigma a 
 \biggl[ \vartheta_1 
 \sum_{k \leq \varepsilon^{-\beta}} 
 \bigl\vert \widehat{X_t^{\varepsilon}}^k 
 \bigr\vert^2 + 
 \vartheta_2
 \sum_{ \varepsilon^{-\beta} < k \leq \psi \varepsilon^{-\beta}}
 \bigl\vert 
  \widehat{X_t^{\varepsilon}}^k 
 \bigr\vert^2
 + \sum_{k > \psi \varepsilon^{-\beta}} 
\bigl\vert 
\widehat{X_t^{\varepsilon}}^k \bigr\vert^2
\biggr]^{1/2}
\\
&\hspace{0pt}  \leq 
\sigma \vartheta_1
+ 
 \frac{C^2 \sigma a^2}{{4}\vartheta_1}  
\biggl[ \vartheta_1  
 \sum_{k \leq \varepsilon^{-\beta}} 
 \bigl\vert \widehat{X_t^{\varepsilon}}^k 
 \bigr\vert^2 +
 \vartheta_2
  \sum_{ \varepsilon^{-\beta} < k \leq  \gamma^{-\delta}  \varepsilon^{-\beta}   }
  \bigl\vert  \widehat{X_t^{\varepsilon}}^k
 \bigr\vert^2
 \\
&\hspace{15pt} + 
{\varepsilon^{2 \beta}}
\vartheta_2 \gamma^{2 \delta} 
 \sum_{ \gamma^{-\delta} \varepsilon^{-\beta} < k\leq 
 \psi \varepsilon^{-\beta} }
 k^2 
   \bigl\vert  \widehat{X_t^{\varepsilon}}^k
 \bigr\vert^2
 + \frac{\varepsilon^{2 \beta}}{\psi^2} \sum_{k > \psi \varepsilon^{-\beta}} 
k^2 \bigl\vert 
\widehat{X_t^{\varepsilon}}^k \bigr\vert^2
\biggr]. 
\end{split}
\end{equation}
This gives 
 \begin{equation} 
\label{eq:condition:7}
\begin{split}
& C   a^2   \leq \sigma \vartheta_1  \quad \textrm{\rm (twice)},
 \\
&C \sigma a^2  \vartheta_2 \gamma^{2 \delta} \leq 
 \varrho
 \vartheta_1 \bigl( 1 - e^{-\gamma^{1-2 \delta}} \bigr), 
 \quad  
 C \sigma a^2 \leq \varrho \vartheta_1 \psi^2  \bigl( 1 - e^{-\gamma \psi^2} \bigr). 
 \end{split}
 \end{equation} 
 
 Lastly, in \eqref{eq:ad:5}, we compare
the first term 
with the corresponding one 
in \eqref{eq:ad:7}. This gives 
 \begin{equation}
\label{eq:condition:8} 
C   \psi^2 \leq \sigma \min(\vartheta_1,\vartheta_2). 
 \end{equation} 
And the second term in 
\eqref{eq:ad:5}
is compared to \eqref{eq:ad:8}. We derive the following condition:
 \begin{equation}
\label{eq:condition:9} 
 C \sigma \leq \varrho \bigl( 1 - e^{-\gamma \psi^2} \bigr). 
 \end{equation} 

\textit{Fifth Step.}
In order to find suitable values of the parameters $\sigma,\varrho,\vartheta_1,\vartheta_2,\gamma$ and $\psi$, 
we assume $\sigma >1$ and we choose 
\begin{equation*} 
\varrho = \rho \sigma, \quad \vartheta_1 = \vartheta, \quad \vartheta_2 = \vartheta^{\theta}, \quad \gamma = \vartheta^{\chi}, 
\quad \psi = \vartheta^{-\varpi},
\end{equation*} 
with $\rho,\vartheta,\theta,\chi,\varpi>0$
and with $\vartheta$ being small. The list of conditions 
\eqref{eq:condition:1} {(wherein we can easily drop the term depending on $\varepsilon$)},
\eqref{eq:ad:treat:2}, 
\eqref{eq:condition:3}, 
\eqref{eq:condition:4},
\eqref{eq:condition:5},
\eqref{eq:condition:6},
\eqref{eq:condition:7},
\eqref{eq:condition:8}
and
\eqref{eq:condition:9} identified in the previous step rewrite:
\begin{equation}
\label{eq:condition:all} 
\begin{split} 
&\vartheta \geq  C \rho  \vartheta^{\chi} \bigl( \vartheta + \vartheta^{\theta- {3} \varpi} \bigr) 
\\
&C \rho^2 a^2 \vartheta^\chi  \leq   \vartheta^2 \sigma,
\quad 
C \rho^2 a^2 \leq   \vartheta^{1+\theta-\chi(1-2 \delta)} \sigma, 
\\
&C \rho a^2 \vartheta^{2 \chi \delta-1} \leq  {1},  
\\
&C \rho a^2 \leq    \vartheta^{1-2 \varpi} { \bigl( 1 - \exp(- \vartheta^{\chi-2\varpi})\bigr),}
\\
&\hspace{50pt} {\textrm{\rm which holds if} \ 
 C \rho a^2 \leq \tfrac12 \vartheta^{1-2 \varpi}, \quad 
 2 \varpi > \chi, 
 \quad \vartheta \ll 1}, 
\\
&C  \rho  \vartheta^{\chi} \leq 1,
\\
&C \rho  \vartheta^{\chi(1-2 \delta)} 
\leq 1,
\\
&C a^2 \leq \sigma \vartheta, 
\\
&C  a^2 \vartheta^{\theta+2\delta \chi} 
\leq {\rho \vartheta \bigl( 
1 
 - \exp(- \vartheta^{\chi(1-2 \delta)})\bigr),}
 \\
&\hspace{50pt} 
{\textrm{\rm which holds if} \ 
C  a^2 \vartheta^{\theta+2\delta \chi} 
\leq 
\tfrac12  \rho \vartheta^{1+ \chi (1- 2 \delta)}, \quad \delta < 1/2, \quad \vartheta \ll 1,}  
\\
&C   a^2 \leq \rho \vartheta^{1-2 \varpi} { \bigl( 1 - \exp(- \vartheta^{\chi-2\varpi})\bigr),}
\\
&\hspace{50pt} {\textrm{\rm which holds if} \ 
 C a^2 \leq \tfrac12 \rho  \vartheta^{1-2 \varpi}, \quad 
 2 \varpi > \chi, 
 \quad \vartheta \ll 1},
\\
&C \vartheta^{2 \varpi} \leq \sigma \min( \vartheta, \vartheta^\theta),
\\
&C  \leq \rho { \bigl( 1 - \exp(- \vartheta^{\chi-2\varpi})\bigr),}
\quad  {\textrm{\rm which holds if} \ 
 C  \leq \tfrac12 \rho, \quad 
 2 \varpi > \chi, 
 \quad \vartheta \ll 1}.
\end{split} 
\end{equation}  
We then choose $\rho={2C}$ as required in the last line. Next, under the conditions $\theta > {3} \varpi$, 
$\chi \delta >1/2$, $2 \varpi >\chi>1$, $\delta <1/2$, $\theta + 4 \delta \chi \geq 1 + \chi$ (which are not empty: fix $\delta$ first, $\chi$ second and then 
choose $\varpi$ and finish with $\theta$), we get that all the conditions without $\sigma$ nor $a$ inside are automatically satisfied for $\vartheta$ small enough. Moreover, there exists an exponent $\psi >0$ such that all the equations featuring $a$ but not $\sigma$
are satisfied under the condition $C a \leq \vartheta^{-\psi}$, {which is indeed true if 
 $\vartheta \leq [C a_0]^{-1/\psi}$
 (recall that $a_0$ is given and $a \leq a_0$).
 Finally, once
 $\vartheta$
has been chosen, all the conditions on $\sigma$ are satisfied if $\sigma \geq C_\vartheta \max(1,a^2)$.} 

Back to \eqref{eq:ad:1}--\eqref{eq:ad:8}, we
integrate from $0$ to $T \wedge \tau_{\varepsilon}$,
{for a fixed $T>0$}, 
 where 
we recall that 
$\tau_{\varepsilon} = \inf \{ t \geq 0 : \| X_t^{\varepsilon}  - X_0^\varepsilon \|_2 \geq a \}$.
Our analysis  says that, 
among all the terms in
\eqref{eq:ad:1}--\eqref{eq:ad:8}, 
we can just retain 
{a fraction of}
\eqref{eq:ad:6}
({plus a martingale term})
as a lower bound for 
${\mathcal A}_{T \wedge \tau_{\varepsilon}}^{\varepsilon} {\mathcal D}_{T 
\wedge \tau_{\varepsilon}}^{\varepsilon}$. Replacing 
$\varrho$ in the definition of ${\mathcal A}_{T \wedge \tau_{\varepsilon}}$ by its values and upper bounding 
$\| A X_{T \wedge \tau_{\varepsilon}}^{\varepsilon} \|_2$ by
$\| A (X_{T \wedge \tau_{\varepsilon}}^{\varepsilon} -X_0^{\varepsilon}) \|_2
+
\| A X_0 \|_2 \leq 
a + \| X_0 \|_2$, we obtain 
\begin{equation*} 
{\mathbb E} \Bigl[ \exp \bigl( \frac{C_\vartheta {\max(1, a^2)}}{\varepsilon^2}
\bigl[ a^2+ \| X_0 \|_2^2 \bigr] \bigr) 
 \Bigr] \geq
  \frac{1}{
   {C_\vartheta}
  \varepsilon^2} 
 {\mathbb E}( \tau_{\varepsilon}),
\end{equation*}
{for a possibly new value of $C_\vartheta$}. 
This completes the proof. 
\end{proof} 

\subsection{Completion of the Proof of
Theorem 
\ref{thm:meta}}

The proof is 
easily completed by 
combining 
the two Propositions 
\ref{eq:lower:bound} 
and
\ref{eq:upper:bound}. 
There is however one subtlety:
one must check that 
the right-hand side in 
\eqref{eq:condition:lower:bound} 
can be chosen 
as small as needed (and in particular smaller than the left-hand side). This is indeed compatible 
with the choice of the parameters 
in Proposition
\ref{eq:upper:bound}:
the two parameters $\vartheta_1$ and 
$\vartheta_2$ are small 
and the product 
$\psi \vartheta_2$
is also small because 
of the constraint 
$\theta > {3} \varpi$. As for the left-hand side in 
\eqref{eq:condition:lower:bound} , we can take $\sigma$
as small as needed (the choice of 
$\sigma$ in the statement of 
Proposition
\ref{eq:lower:bound}
has nothing to do with the choice of $\sigma$ 
in the proof of the upper bound). 
{Lastly, it must be observed that one can replace 
the exponent $\max(1,a^2)$
 appearing in
 the statement of Proposition 
\ref{eq:upper:bound}
by $a^2$, as done in 
\eqref{eq:stoppingtime:bounds:CV}, just by using the fact that 
$a \geq 1/a_0$.}

\begin{remark}
\label{rem:4.5}
When $X_0$ is  {constant}, say $X_0=0$, the proof of the upper bound is easier. In 
\eqref{eq:ad:2}
and
\eqref{eq:ad:4}, there is no need to 
upper bound 
$\| X_t^{\varepsilon} - X_0\|_2
= \| X_t^{\varepsilon} \|_2$ 
by the distance $a$ and, then, 
one just has to 
upper bound 
\begin{equation*}  
\varrho 
\| X_t^{\varepsilon} \|_2 
 \biggl[ \gamma  
 \sum_{k \leq \varepsilon^{-\beta}} 
 \bigl\vert \widehat{X_t^{\varepsilon}}^k 
 \bigr\vert^2 + 
 \sum_{ \varepsilon^{-\beta} < k \leq \psi \varepsilon^{-\beta}}
  \bigl( 1 -  e^{ -  \gamma \varepsilon^{2 \beta} k^2} \bigr) 
 \bigl\vert 
  \widehat{X_t^{\varepsilon}}^k 
 \bigr\vert^2
 + \sum_{k > \psi \varepsilon^{-\beta}} 
\bigl\vert 
\widehat{X_t^{\varepsilon}}^k \bigr\vert^2
\biggr]^{1/2}
\end{equation*} 
by 
\begin{equation*} 
\begin{split}
&\frac{\varrho { \vartheta_1} }{\zeta^2} 
\biggl[ 
 \sum_{k \leq \varepsilon^{-\beta}} 
 \bigl\vert \widehat{X_t^{\varepsilon}}^k 
 \bigr\vert^2 + 
 \sum_{ \varepsilon^{-\beta} < k \leq \psi \varepsilon^{-\beta}}
  \bigl\vert 
  \widehat{X_t^{\varepsilon}}^k 
 \bigr\vert^2
 + \frac{\varepsilon^{2 \beta}}{\psi^2} \sum_{k > \psi \varepsilon^{-\beta}} 
k^2 \bigl\vert 
\widehat{X_t^{\varepsilon}}^k \bigr\vert^2
\biggr]
\\
&\hspace{15pt} + \frac{C^2 \varrho \zeta^2}{{\vartheta_1}}  
\biggl[ \gamma  
 \sum_{k \leq \varepsilon^{-\beta}} 
 \bigl\vert \widehat{X_t^{\varepsilon}}^k 
 \bigr\vert^2 + 
 \sum_{ \varepsilon^{-\beta} < k \leq \psi \varepsilon^{-\beta}}
 \bigl( 1 -  e^{ -  \gamma \varepsilon^{2 \beta} k^2} \bigr) 
  \bigl\vert 
  \widehat{X_t^{\varepsilon}}^k 
 \bigr\vert^2
 + \frac{\varepsilon^{2 \beta}}{\psi^2} \sum_{k > \psi \varepsilon^{-\beta}} 
k^2 \bigl\vert 
\widehat{X_t^{\varepsilon}}^k \bigr\vert^2
\biggr],
\end{split}
\end{equation*} 
and similarly for
\eqref{eq:ad:4}. Very briefly, the second term right above can be treated by substituting 
$\zeta$ for $a$ in the fifth step (see in particular 
\eqref{eq:condition:all}), 
which is always possible by choosing 
$\zeta = \vartheta^{-\epsilon}$ for
$\epsilon >0$ as small as needed
(this is very easy to check as $a$ should be thought 
as $O(\vartheta^0)$: we are just increasing the exponent 
by $\epsilon$, which is always possible because all the inequalities 
involving the various exponents in 
\eqref{eq:condition:all} are strict, thus leaving for some room). 

As for the first term, one may follow
\eqref{eq:condition:7}. This puts the additional constraints: 
\begin{equation*} 
\begin{split}
&\frac{C \rho}{\zeta^2} \leq {\sigma},
\quad 
\frac{C \rho}{\zeta^2} {\vartheta_1} \leq {\sigma}   \vartheta_2,
\\
&\frac{C}{\zeta^2}  
\gamma^{2 \delta}
{ \vartheta_1}
 \leq 
 \bigl( 1 - e^{-\gamma^{1-2 \delta}} \bigr), 
 \quad  
  \frac{C {  \vartheta_1}}{\zeta^2}     \leq  \psi^2  \bigl( 1 - e^{-\gamma \psi^2} \bigr). 
 \end{split}
\end{equation*} 
The first line does not raise any difficulties since $\sigma$ 
can be chosen as large as possible (this is exactly the principle
of the {argument} in the fifth step of the proof of Proposition \ref{eq:upper:bound}). 
As for the second line, we replace $\gamma$ by $\vartheta^{\chi}$. One obtains 
the condition {$1+4 \chi \delta > \chi + 2 \epsilon$}, which is 
{satisfied under the condition $\delta > 1/4$. This is compatible with our conditions since $ \delta <1/2$ (as before, we have enough `room' to do this).} 
As for the last one, 
one obtains the condition $-2\varpi < 1  + 2\epsilon$, which is also compatible with our conditions. 

In the end, there is no longer any restriction on the choice of $a$ in the proof of the upper bound and $\sigma$ can be chosen independently of $a$. Moreover, since we no longer compare 
\eqref{eq:ad:2}
or \eqref{eq:ad:4}
with 
\eqref{eq:ad:6}, there is no longer need to choose $\beta=2$. In turn, we can choose $\alpha=\beta<2$ and there is no longer any restriction on the choice 
of $a$ in the proof of the lower bound. Hence, we recover the standard inequality 
\begin{equation*} 
\exp \bigl( \frac{a^2}{C \varepsilon^2} 
\bigr) \leq {\mathbb E} (\tau_\varepsilon )
\leq \exp \bigl( \frac{C a^2}{\varepsilon^2} 
\bigr). 
\end{equation*}

\end{remark}